\documentclass[11pt,a4paper,reqno]{amsart}

\usepackage{a4wide}
\usepackage{graphicx}
\usepackage{latexsym}
\usepackage{epsfig}
\usepackage{amssymb,amsmath}
\usepackage{amstext}
\usepackage{amsgen}
\usepackage{amsxtra}
\usepackage{amsgen}
\usepackage{amsthm}
\usepackage{tcolorbox}
\usepackage{mathrsfs,dsfont}
\usepackage[T1]{fontenc}

\usepackage[top=3cm, bottom=3cm,left=3cm, right=3cm]{geometry}

\usepackage{color}
\usepackage[colorlinks=true, linkcolor=blue, citecolor=blue]{hyperref}
\usepackage{orcidlink}

\usepackage[pagewise]{lineno}

\newtheorem{thm}{Theorem}
\newtheorem*{thm*}{Theorem}
\newtheorem{prop}{Proposition}[section]
\newtheorem{lemma}[prop]{Lemma}
\newtheorem{cor}[prop]{Corollary}

\theoremstyle{remark}
\newtheorem{rem}[prop]{Remark}

\theoremstyle{definition}

\newtheorem{definition}[prop]{Definition}

\numberwithin{equation}{section}

\newcommand{\per}{\mathrm{per}}

\newcommand{\RR}{\mathbb{R}}

\DeclareMathOperator{\dv}{div}
\DeclareMathOperator{\esssup}{ess \, sup}

\renewcommand{\span}{\mathrm{span}}

\DeclareMathOperator{\essinf}{ess \, inf}

\renewcommand{\P}{{\bf P}}
\newcommand{\p}{{\bf p}}

\newcommand{\e}{{\bf e}}
\newcommand{\x}{{\bf x}}

\newcommand{\bxi}{\xi}

\renewcommand{\d}{\, \mathrm{d} }
\newcommand{\der}{\mathrm{d}}

\newcommand{\bF}{\mathbf{F}}

\newcommand{\loc}{\mathrm{loc}}
\usepackage{xcolor}

\newcommand{\ff}[1]{f^{(#1)}}
\newcommand{\rr}[1]{\rho^{(#1)}}

\def\de{\partial}

\DeclareMathOperator{\supp}{\mathrm{supp}}

\begin{document}
\title[Regularity and Uniqueness for a Model of Active Particles]{Regularity and Uniqueness for a Model of Active Particles with Angle-Averaged Diffusions
}

\author[L.~C.~B.~Alasio]{Luca C.~B.~Alasio}
\address[L.~C.~B.~Alasio]{
INRIA, Laboratoire Jacques-Louis Lions (LJLL), Sorbonne Universit\'e, Universit\'e Paris Cit\'e, CNRS, 4 Place Jussieu, 75005 Paris, France}
\email{luca.alasio@inria.fr \orcidlink{0000-0001-8141-2055}}

\author[S.~M.~Schulz]{Simon M.~Schulz}
\address[S.~M.~Schulz]{Scuola Normale Superiore, Centro di Ricerca Matematica Ennio De Giorgi, Piazza dei Cavalieri, 3,  56126 Pisa, Italy}\email{simon.schulz@sns.it \orcidlink{0000-0001-8081-0924}}



\begin{abstract}
  We study the regularity and uniqueness of weak solutions of a degenerate parabolic equation, arising as the limit of a stochastic lattice model of self-propelled particles. The angle-average of the solution appears as a coefficient in the diffusive and drift terms, making the equation nonlocal. We prove that, under unrestrictive non-degeneracy assumptions on the initial data, weak solutions are smooth for positive times. Our method rests on deriving a drift-diffusion equation for a particular function of the angle-averaged density and applying De Giorgi's method to show that the original equation is uniformly parabolic for positive times. We employ a Galerkin approximation to justify rigorously the passage from divergence to non-divergence form of the equation, which yields improved estimates by exploiting a cancellation. By imposing stronger constraints on the initial data, we prove the uniqueness of the weak solution, which relies on Duhamel's principle and gradient estimates for the periodic heat kernel to derive $L^\infty$ estimates for the angle-averaged density. 
\end{abstract}

\maketitle

\vspace*{-1cm}

\setcounter{tocdepth}{1}
\begin{scriptsize}
    \tableofcontents
\end{scriptsize}

\vspace*{-1.4cm}

\section{Introduction}

This work is devoted to the regularity and uniqueness theory for the weak solutions $f(t,x,\theta)$ of the nonlocal degenerate parabolic equation 
\begin{equation}\label{eq:main eqn}
    \partial_t f +   \dv \!\big(  (1- \rho) f \e(\theta)\big)
= \dv \!\big( (1-\rho)\nabla f + f \nabla \rho \big) + \partial_{\theta}^2 f, 
\end{equation}
where $\e(\theta) = (\cos \theta, \sin \theta)$, and the \emph{angle-averaged density} $\rho$ is defined by 
\begin{equation}\label{eq:rho def}
    \rho(t,x) := \int_0^{2\pi} \!\!\! f(t,x,\theta) \d \theta.
\end{equation}
The symbols $\dv$ and $\nabla$ in \eqref{eq:main eqn} represent divergence and gradient with respect to the $x$ variable only. This equation is supplemented with periodic boundary conditions both in the space variable $x \in \Omega = (0,2\pi)^2$ and the angle variable $\theta \in (0,2\pi)$. In the sequel, we denote the space-angle coordinate by $\bxi = (x,\theta)$, which belongs to $\Upsilon = \Omega \times (0,2\pi) = (0,2\pi)^3$. 

\noindent\textbf{Background.} Nonlocal equations of the form \eqref{eq:main eqn} arise naturally in the physical, biological, and social sciences. They are conventionally used to model the evolution of a system of active (\textit{i.e.~}self-propelled) particles with repulsive interactions and diffusion in position and orientation. This may be, for instance, in the context of the modelling of bacterial suspensions \cite{bacterial suspensions}, colloids \cite{self propelled colloids}, collective movement in robotics and in animal species \cite{collective}, pedestrian and traffic flows \cite{pedestrian}, or in condensed matter physics (see also \cite{Cates:2013ia,Romanczuk:2012iz,
Yeomans:2015dt} and the references therein).  For phase separation in active systems we refer to \cite{Redner.2013,
Speck:2015um}. The work in \cite{DegondEtAl} provides a discussion on methods and models for the dynamics of aligning self-propelled rigid bodies, and \cite{MinTang1} studies a related Fokker--Planck system with confinement and boundary terms. 

Equation \eqref{eq:main eqn} was formally derived from a stochastic lattice model of interacting particles in \cite{bbes model}; see also the related work \cite{JamesMariaErignoux}, where a rigorous derivation of a similar equation is performed via a hydrodynamic limit. We also refer the reader to \cite{BrunaBurgerDeWit} for new directions in this topic. The precise structure of the terms in \eqref{eq:main eqn} depends on the scaling with which the number of particles tends to infinity in proportion to how the radius of the repulsive interactions vanishes; this scaling dictates the nature of the degeneracy in the diffusion. From the point of view of applications, a thorough understanding of the regularity and uniqueness properties of the equation is necessary when developing efficient numerical schemes for simulation. Indeed, without knowledge of uniqueness prior to the design of the computational method, the numerical algorithm may ``jump'' between solutions without ever converging. This motivates the need for the present paper, in addition to answering noteworthy theoretical questions. 

\noindent\textbf{Literature review.}
The modern study of the regularity of weak solutions of parabolic equations builds on the fundamental work of De Giorgi \cite{degiorgi1}, Nash \cite{nash1}, and Moser \cite{moser1}. While the literature on uniformly parabolic equations is rich, the same cannot be said for degenerate or nonlocal evolution equations, which often pose new challenges. 
The Cauchy problem for a large class of degenerate parabolic systems has been studied by Amann in the 1990s, see \textit{e.g.}~\cite{Amann}, and recently revisited in \cite{MoussaGallagher}.
New regularity properties for degenerate systems of Keller-Segel-type have been obtained in the last few years, see \textit{e.g.}~\cite{DesvillettesLaurencotEtAl,  WinklerRefinedReg, WinklerOrlicz}. We refer to \cite{BonforteDolbeaultNazaretSimonov} for several remarkable results on regularity, the Harnack inequality and applications to the fast diffusion equation. Further advances in terms of regularity of weak solutions of cross-diffusion systems have been shown in \cite{DungLe,LMSS}. For nonlocal equations similar to \eqref{eq:main eqn}, we highlight the previous work of the authors \cite{reg1}, in which they obtained regularity and uniqueness results for weak and very weak solutions of a simpler equation, which incorporated the same advection term as \eqref{eq:main eqn} and a local non-degenerate diffusion. We highlight that there exists, by now, a well-established literature on the existence theory for nonlocal degenerate equations by means of, \textit{e.g.}, gradient flow methods (\textit{cf.~}\cite{AGS,Filippo}). One of the aims of the present work is to add emphasis to the regularity theory for these equations, which goes beyond existence results in measure-theoretic settings.

\noindent\textbf{Novelty.} A global existence theory for \eqref{eq:main eqn} was produced in \cite{Martin} by the second author \emph{et al.}, using a generalisation of the \emph{boundedness-by-entropy method}; originally established in \cite{BoundEntropy} for a related cross-diffusion system. The underlying principle in this approach is to interpret \eqref{eq:main eqn} as a perturbation of a Wasserstein-type gradient flow  associated to the entropy functional 
\begin{equation*}
    \mathcal{E}[f] = \int_\Upsilon \Big( f \log f + (1-\rho) \log (1-\rho)  \Big) \d x \d \theta, 
\end{equation*}
and to rewrite the equation in terms of the first variation $\mathcal{E}'[f] = \log f - \log(1-\rho).$ Using this formulation of the problem, the authors of \cite{Martin} showed (\textit{cf.}~\cite[Theorem 1]{Martin}) the existence of a nonnegative weak solution $f$, in the sense dual to $L^6_t W^{1,6}_{x,\theta}$, satisfying also $0 \leq \rho \leq 1$. The authors also showed the uniqueness of this solution (\textit{cf.}~\cite[Theorem 2]{Martin}) in the special case where the \emph{nonlocal drift} term $\dv((1-\rho)f \e(\theta))$ is omitted from \eqref{eq:main eqn}; in this case, the equation is an exact gradient flow, and uniqueness is expected (\textit{cf.}~\cite{AGS,Filippo}). A weak-strong uniqueness result near constant stationary states (\textit{cf.}~\cite[Theorem 4]{Martin}) was also obtained. The regularity properties of the weak solutions obtained from the boundedness-by-entropy method have not been studied since, primarily because of the intricacies encountered due to the presence of the \emph{nonlocal diffusion} $\dv((1-\rho)\nabla f)$ and the \emph{infinite cross-diffusion} $\dv(f\nabla \rho)$. Indeed, this former term causes a loss of strict parabolicity on the locus $\{\rho=1\}$, while the latter term encodes competing diffusive effects from all of the uncountably many angles $\theta \in (0,2\pi)$, all-in-all giving rise to a panoply of degeneracies. The purpose of the present manuscript is to provide a first thorough regularity analysis for \eqref{eq:main eqn}, so as to address the aforementioned gap in the theory. We then employ our regularity result to prove the uniqueness of weak solutions of \eqref{eq:main eqn} without restriction on the nonlocal drift.

\noindent\textbf{Approach.} In the context of our problem, the regularising effect of the equation \eqref{eq:main eqn} is seen first and foremost on angle-averaged quantities. Indeed, by integrating \eqref{eq:main eqn} with respect to the variable $\theta$, we recover the uniformly parabolic evolution equation for $\rho$, 
\begin{equation}\label{eq:rho eq}
    \partial_t \rho + \dv((1-\rho)\p) = \Delta \rho, 
\end{equation}
where the vector field $\p$, called the \emph{polarisation}, is defined by 
\begin{equation}\label{eq:polarisation def}
    \p(t,x) := \int_0^{2\pi} \!\!\! f(t,x,\theta) \, \e(\theta) \d \theta. 
\end{equation}
Our first main observation is that, under rather minimal non-degeneracy assumptions on the initial data (\textit{cf.}~Definition \ref{def:reg initial data}), one can derive a uniform lower bound on $1\!-\!\rho$ for positive times from \eqref{eq:rho eq}, \textit{cf.}~Proposition \ref{prop:lower bound 1-rho}. This then implies the strong parabolicity of \eqref{eq:main eqn}, and makes it possible to obtain $H^1$ estimates on $f$. These $H^1$ estimates transfer to $\p$ by \eqref{eq:polarisation def}, whence bootstrapping into \eqref{eq:rho eq} yields the $H^2$ regularity of $\rho$. By interpolating, this boundedness then yields sufficient control on $\rho$ and its derivatives to perform a De Giorgi type iteration on \eqref{eq:main eqn} and obtain an $L^\infty$ estimate on $f$ for positive times. The second main observation is concerned with the space-diffusions in \eqref{eq:main eqn}. Provided the passage from divergence form to non-divergence form is justified, a cancellation occurs in terms of the form $\nabla f \!\cdot\! \nabla \rho$, and \eqref{eq:main eqn} may be rewritten 
\begin{equation}\label{eq:main eqn non div form}
    \partial_t f + \dv\!\big((1-\rho) f \e(\theta)\big) = (1-\rho)\Delta f + f \Delta \rho + \partial^2_\theta f. 
\end{equation}
The passage from equation \eqref{eq:main eqn} to its non-divergence form analogue \eqref{eq:main eqn non div form} is rigorously justified by means of a Galerkin approximation, \textit{cf.}~\S \ref{sec:proof of H2 via galerkin} and the proof of Lemma \ref{lem:H2 for f}. We remark that other approaches, such as by difference quotients or mollification, destroy the aforementioned cancellation of the terms $\nabla f  \cdot  \nabla \rho$. Testing \eqref{eq:main eqn non div form} with $\Delta f$, and then subsequently with $\partial^2_\theta f$, yields the $H^2$ regularity of $f$. By bootstrapping and applying time derivatives in \eqref{eq:main eqn}, we are then able to deduce the smoothness of $f$ for positive times. Once the regularity result is established, we use it to show that the weak solution of \eqref{eq:main eqn} is unique, under additional assumptions on the initial data. First, we show uniqueness on a small time interval of positive length using Duhamel's principle and gradient estimates for the heat equation. More precisely, we obtain an estimate of the form $\Vert \rho_1 - \rho_2 \Vert_{L^\infty((0,t)\times\Omega)} \leq \Vert f_1 - f_2 \Vert_{L^2((0,t)\times\Upsilon)}$ for any two weak solutions $f_1,f_2$ and $t$ sufficiently small, which yields uniqueness by the usual $H^1$ estimate on the equation for $f_1-f_2$. We then extend this result to all time intervals using the smoothness of $f_1,f_2$ for all positive times.

\noindent\textbf{Structure of the paper.}
The rest of the paper is organised as follows. In \S \ref{sec:setup}, we introduce the notations and some technical lemmas used throughout the paper, and provide our notion of weak solution of \eqref{eq:main eqn}. We also recall the main theorem of \cite{Martin}, which provides the existence of the weak solutions. In \S \ref{sec:main results}, we state our two main theorems, and provide a concise summary of our strategy of proof. \S \ref{subsec:angle averaged quantities} is concerned with certain estimates that hold up to the initial time $t=0$, and which are then used to prove the strong parabolicity of \eqref{eq:main eqn} for positive times, in \S \ref{sec:strong parab}. In \S \ref{sec:boundedness positive times}, we obtained improved boundedness estimates for positive times, by means of a Galerkin approximation which allows us to rigorously pass from the divergence form \eqref{eq:main eqn} to the non-divergence form \eqref{eq:main eqn non div form} and obtain $H^2$ estimates by exploiting a cancellation. In \S \ref{sec:higher reg}, we bootstrap our previous observations to show smoothness of the weak solutions for positive times, and thus we prove our first main theorem. Finally, in \S \ref{sec:uniqueness}, we prove our second main theorem, concerned with the uniqueness of weak solutions. Appendix \ref{sec:appendix} contains the proofs of various technical lemmas.

\vspace{-0.1cm}
\section{Set-up}\label{sec:setup}

\subsection{Notations} Throughout the paper, we denote the concatenated space-angle variable by $\bxi = (x,\theta)$. For $T>0$, we let $\Omega_T := (0,T)\times\Omega$ and $\Upsilon_T := (0,T)\times\Upsilon$, and, for $t \in (0,T)$, $$\Omega_{t,T} := (t,T) \times \Omega, \qquad \Upsilon_{t,T} := (t,T) \times \Upsilon.$$ 
Let 
$t_0$ denote a Lebesgue point in $(0,T)$ and, without loss of generality, suppose that such point is common to $f$ and all relevant quantities determined from $f$ (such as $\rho$ and $\p$).  $\mathbb{R}_+ := (0,\infty)$ denotes the positive real numbers. 

The topological dual of a function space $E$ is denoted by $E'$, and the bracket $\langle \cdot, \cdot \rangle$  denotes the duality pairing. We recall the periodic functions spaces, denoted by $Z_\per(S)$ for $S \in \{(0,2\pi)^d\}_{d=1}^3$, \textit{i.e.}~$d=2$ corresponding to $\Omega$ and $d=3$ to $\Upsilon$, with $Z \in \{L^p,W^{k,p},C^k\}$, understood to mean, with $\{\e_i\}_{i=1}^d$ the standard basis of $\mathbb{R}^d$, 
\begin{equation*}
    Z_\per(S) := \big\{ g:\mathbb{R}^d \to \mathbb{R}: \, \Vert g \Vert_{Z(S)} < \infty, \text{ and } g(y+2\pi \e_i) = g(y) ~ \forall y \in \mathbb{R}^d, i \in \{1,\dots,d\}  \big\}. 
\end{equation*}
 
For two tensors $\mathbf{a} = (a_{ij\dots m})_{ij\dots m}$ and $\mathbf{b} = (b_{kl\dots n})_{kl\dots n}$, we write the product $\mathbf{a} \otimes \mathbf{b} = (a_{ij\dots m} b_{kl\dots n})_{ij \dots m kl \dots n}$. For two matrices $\mathbf{A} = (A_{ij})_{ij}$ and $\mathbf{B} = (B_{ij})_{ij}$ of equal size, we denote the Frobenius product by $\mathbf{A} : \mathbf{B} = \sum_{ij} A_{ij}B_{ij}$. For a vector field $\mathbf{w}=(w_i)_i$, we write the derivative matrix $\nabla\mathbf{w} = (\partial_j w_i)_{ij}$, and for a $n$-tensor field $\mathbf{W} = (W_{ij\dots kl})_{ij\dots kl}$, we write its divergence as the $(n-1)$-tensor field $\dv \mathbf{W} = (\sum_l \partial_l W_{ij\dots kl})_{ij\dots k}$.

\subsection{Notion of solution}

We recall the notion of weak solution of \eqref{eq:main eqn} introduced in \cite{Martin}, and explain under which conditions the existence of weak solutions was proved.

\begin{definition}[Function Space $\mathcal{X}$]\label{def:function space}
   We define the function space $\mathcal{X}$ to be the family of functions $f \in L^3(0,T;L^3_\per(\Upsilon))$ which satisfy, with $\rho := \int_0^{2\pi} f \d \theta$, the following: 
    \begin{itemize}
        \item[(i)] $f \geq 0$ a.e.~in $\Upsilon_T$; 
        \item[(ii)] $\partial_\theta \sqrt{f} \in L^2(\Upsilon_T)$ and $\partial_t f \in (L^6(0,T;W^{1,6}_\per(\Upsilon))'$; 
        \item[(iii)] $0 \leq \rho \leq 1$ a.e.~in $\Omega_T$, $\partial_t \rho \in (L^2(0,T;H^1_\per(\Omega))'$, and $\nabla\sqrt{1-\rho} \in L^2(\Omega_T)$; 
        \item[(iv)] $\sqrt{1-\rho}\nabla\sqrt{f} \in L^2(\Upsilon_T)$. 
    \end{itemize}
\end{definition}
Note that the previous definition implies that any $f \in \mathcal{X}$ satisfies the following: 
\begin{equation}\label{eq:early boundedness nabla f weighted in intro}
    \sqrt{1-\rho}\nabla f , \partial_\theta f \in L^{\frac{3}{2}}(\Upsilon_T). 
\end{equation}

We now state the definition of weak solution. A weak solution must satisfy two weak formulations: the first is tailored to test functions that are periodic 
with respect to the space-angle variable (\textit{cf.}~\cite[Definition 1.1]{Martin}), and the second allows for non-periodic test functions (\textit{cf.}~\cite[Lemma 4.2]{Martin}). The former formulation is simpler to use when performing energy estimates (as the periodicity removes the need to localise), while the latter is better adapted to the rescalings involved in De Giorgi's method (\textit{cf.}~\cite[\S 2]{reg1}). 

    \begin{definition}[Weak Solution]\label{def:weak sol}
        We say that $f \in \mathcal{X}$ is a \emph{weak solution} of \eqref{eq:main eqn} if, for all $\varphi \in L^6(0,T;W^{1,6}_\per(\Upsilon)) \cap\{\partial_t \varphi \in L^{\frac{3}{2}}(\Upsilon_T)\}$ and for a.e.~$t_1,t_2 \in [0,T]$, there holds 
  \begin{equation}\label{eq:weak sol}
    \begin{aligned}
     \int_{t_1}^{t_2}\! \int_{\Upsilon} f & \partial_t \varphi \d \bxi \d t + \int_{t_1}^{t_2} \int_{\Upsilon} (1-\rho) f \e(\theta) \cdot \nabla \varphi \d \bxi \d t \\ 
      =&\int_{t_1}^{t_2} \!\int_{\Upsilon} \!\Big( \big((1-\rho)\nabla f \!+\! f \nabla \rho \big) \!\cdot\! \nabla \varphi \!+\! \partial_\theta f\partial_\theta \varphi \Big) \d \bxi \d t \!+\! \int_{\Upsilon} f \varphi \d\bxi \Big|_{t_2} \!\!-\!\! \int_{\Upsilon} f \varphi \d\bxi \Big|_{t_1}, 
    \end{aligned}
\end{equation}
        and, for all $\psi \in C^\infty([0,T]\times\mathbb{R}^3)$ with $\psi(t,\cdot) \in C^\infty_c(\mathbb{R}^3)$ for all $t$ and for a.e.~$t_1,t_2 \in [0,T]$, 
      \begin{equation}\label{eq:weak sense eq}
    \begin{aligned}
      \int_{t_1}^{t_2}\! & \int_{\mathbb{R}^3} f  \partial_t \psi \d \bxi \d t + \int_{t_1}^{t_2} \int_{\mathbb{R}^3} (1-\rho) f \e(\theta) \cdot \nabla \psi \d \bxi \d t \\ 
      =&\int_{t_1}^{t_2} \!\int_{\mathbb{R}^3} \!\Big( \big((1-\rho)\nabla f \!+\! f \nabla \rho \big) \!\cdot\! \nabla \psi \!+\! \partial_\theta f\partial_\theta \psi \Big) \d \bxi \d t \!+\! \int_{\mathbb{R}^3} f \psi \d\bxi \Big|_{t_2} \!\!-\!\! \int_{\mathbb{R}^3} f \psi \d\bxi \Big|_{t_1}. 
    \end{aligned}
\end{equation}
    \end{definition}

The global existence of weak solutions, given admissible initial data, was proved in \cite{Martin}:

    \begin{thm}[Theorem 1 and Lemma 4.2 of \cite{Martin}]\label{rem:existence of weak sol}
Let $T>0$ and $f_0 \in L^q_\per(\Upsilon)$ for $q>1$ be given non-negative initial data satisfying, with $\rho_0 := \int_0^{2\pi} f_0 \d \theta$, the condition $0 \leq \rho_0 \leq 1$ {a.e.~in }$\Omega$. Then, there exists $f \in \mathcal{X}$ a weak solution of \eqref{eq:main eqn} in the sense of Definition \ref{def:weak sol}, and $\lim_{t \to 0^+} \Vert f(t,\cdot) - f_0 \Vert_{(W^{1,6}_\per(\Upsilon))'} = 0$, $\lim_{t \to 0^+} \Vert \rho(t,\cdot) - \rho_0 \Vert_{L^2(\Omega)} = 0$. 
\end{thm}

Since the final time $T$ is arbitrary in the previous theorem, we choose it arbitrarily once and for all at this stage, \textit{i.e.}, $T$ does not change later in the manuscript.

\subsection{Technical results used in the paper}

Throughout the manuscript, we repeatedly make use of the following technical results: the first is a well-known interpolation lemma for spaces depending on time,  and the second is a Calder\'on--Zygmund Theorem for periodic functions.

\begin{lemma}[\S 1 Proposition 3.2, \cite{DiBenedetto}]\label{lem:dibenedetto classic}
    Let $n \in \mathbb{N}$ and $\varpi \subset \mathbb{R}^n$ have piecewise smooth boundary, and let $p,m \geq 1$. There exists a positive constant $C_I=C_I(n,p,m,\partial\varpi,\varpi)$ such that, for all $v \in L^\infty(0,T;L^m(\varpi)) \cap L^p(0,T;W^{1,p}(\varpi))$, there holds, with $q = p(1+m/n)$, 
    \begin{equation*}
        \Vert v \Vert_{L^q((0,T)\times\varpi)} \leq C_I ( 1 + {T}) \Big( \Vert v \Vert_{L^\infty(0,T;L^m(\varpi))} + \Vert v \Vert_{L^p(0,T;W^{1,p}(\varpi))} \Big). 
    \end{equation*}
\end{lemma}

The next result is a periodic Calder\'on--Zygmund inequality.

\begin{lemma}[Lemma 1.4, \cite{reg1}]\label{lem:CZ periodic}
   Let $p \in (1,\infty)$. There exists a positive constant $C=C(p,\Upsilon)$ such that for all $v \in W^{1,p}_\per(\Upsilon)$ with $\Delta_{\bxi} v \in L^p(\Upsilon)$, there holds 
    \begin{equation*}
        \Vert \nabla^2_{\bxi} v \Vert_{L^p(\Upsilon)} \leq C\Big( \Vert \Delta_{\bxi} v \Vert_{L^p(\Upsilon)} + \Vert v \Vert_{W^{1,p}(\Upsilon)} \Big), \qquad \Vert \nabla^2_{\bxi} v \Vert_{L^2(\Upsilon)} = \Vert \Delta_{\bxi} v \Vert_{L^2(\Upsilon)}. 
    \end{equation*}
\end{lemma}

Finally, we recall a classical Schauder result for parabolic equations in non-divergence form, which is well-known to specialists and may be found in greater generality in standard reference texts for function spaces that do not encode periodicity (\textit{e.g.}~\cite[Ch.7, \S 6]{lieberman}).

\begin{lemma}[Schauder-type estimate]\label{lem:Lp schauder}
    Let $t \in [0,T]$, $q \in [2,\infty)$, $\alpha \in (0,1)$, $c > 1$. Assume the matrix $A \in C^{0,\alpha}(\Upsilon_{t,T})$ satisfies,  a.e.~$\Upsilon_{t,T}$, the assumption $c^{-1} |w|^2 \leq Aw \cdot w \leq c|w|^2$ for all $w \in \mathbb{R}^3$. Suppose that $f \in L^\infty(\Upsilon_{t,T}) \cap L^2(t,T;H^2_\per(\Upsilon))$, with $\partial_t f \in L^2(t,T;(H^1_\per(\Upsilon))')$, satisfies the linear non-divergence inclusion 
    \begin{equation*}
        \partial_t f - A:\nabla^2_{\bxi} f \in L^q(\Upsilon_{t,T}). 
    \end{equation*}
Then, for a.e.~$\tau \in (t,T)$, there holds $\partial_t f , \nabla^2_{\bxi} f \in L^q(\Upsilon_{\tau,T})$. 
\end{lemma}

    \section{Main Results}\label{sec:main results}

    We state our main results, and then provide a summary of our strategy of proof. 

    \subsection{Statements of Main Results}

    Our first main theorem is about the smoothness of weak solutions of \eqref{eq:main eqn} under non-degeneracy assumptions on the initial data.

    \begin{definition}[Regular initial data]\label{def:reg initial data}
        We say that $f_0$ is a \emph{regular initial datum} if: 
        \begin{enumerate}
            \item[(i)] $f_0$ is nonnegative a.e.~and $f_0 \in L^q_\per(\Upsilon)$ for some $q>1$; 
            \item[(ii)] 
            there exists a nonnegative function $h \in C^2(\mathbb{R}_+)$ satisfying the assumptions
        \begin{equation}\label{eq:conditions on h}
        \lim_{s \to 0_+}h(s) = \infty, \quad   h' < 0 < h''  \quad \text{on } (0,1], \qquad \frac{s h''(s)}{h'(s)} \in L^\infty_{\loc}(\mathbb{R}_+),
        \end{equation}
        and the quantity $\rho_0 = \int_0^{2\pi} f_0 \d \theta$ satisfies the condition $0 \leq \rho_0 < 1$ a.e.~in $\Omega$ and 
        \begin{equation}\label{eq:additional condition for strong parab}
            h(1\!-\!\rho_0) \in L^2(\Omega). 
            \end{equation}
        \end{enumerate}
    \end{definition}

For convenience, we highlight that a prototypical choice of $h$ is, \textit{e.g.}, $h(s) = s^{-q}$ ($q>0$). The next remark shows that Definition \ref{def:reg initial data} permits very general initial data, including those  concentrating the solution of \eqref{eq:main eqn} near---but not in---the degenerate zone $\{\rho=1\}$.

\begin{rem}[Definition \ref{def:reg initial data} is unrestrictive] Let $q > 1$ and $f_0 \in L^q_\per(\Upsilon)$ be a nonnegative initial datum. Suppose that, for any $x_* \in \Omega$ and any $m > 0$, $f_0$ is such that 
$$1-\rho_0(x) = \mathcal{O}(e^{-\frac{1}{|x-x_*|^{m}}}) \quad \text{in a neighbourhood } \mathcal{N} \text{ of } x_*, \qquad \esssup_{\Omega \setminus \mathcal{N}}\rho_0 <1.$$Then, $f_0$ is a regular initial datum, as the choice $h(s) = \log(-\log(s))$ for $0 < s \leq e^{-2}$, with appropriate extension for $s > e^{-2}$, satisfies assumptions \eqref{eq:conditions on h}--\eqref{eq:additional condition for strong parab}. This argument can be repeated for any finite number of such points $x_*$. Similarly, letting $\exp^{-n} := e^{-(\cdot)^{-1}} \circ \dots \circ e^{-(\cdot)^{-1}}$, $\log^n := \log \circ \dots \circ \log$ be the $n$-th iterates ($n\in\mathbb{N}$), if $f_0$ is such that 
$$1-\rho_0(x) = \mathcal{O}(\exp^{-n}(e^{-\frac{1}{|x-x_*|^{m}}})) \quad \text{in a neighbourhood } \mathcal{N} \text{ of } x_*, \qquad \esssup_{\Omega \setminus \mathcal{N}}\rho_0 <1,$$ then, select a suitable extension of $h(s) = \log^n(\log(-\log(s)))$, and $f_0$ is regular. 
\end{rem}

\begin{rem}[Consequence of condition \eqref{eq:conditions on h}]
The conditions \eqref{eq:conditions on h} on the function $h$ imply  the estimate $( sh'(s))' - h'(s) = s h''(s) > 0$ on $\mathbb{R}_+$. Integrating from a fixed $s_* > 0$ implies $s h'(s) - s_* h'(s_*) \geq   h(s) - h(s_*) $, whence, since $h' < 0\leq h$, there exists $M \geq 0$ such that 
\begin{equation}\label{eq:third condition on initial data}
    s |h'(s)| \leq h(s) + M \quad \text{for all } s \in (0,1]. 
\end{equation}    
Because of the opposite signs of $h'$ and $h''$, the condition \eqref{eq:conditions on h} may be interpreted as a McCann convexity ``limited growth'' condition (\textit{cf.}~the growth condition in \cite[\S 3]{mccann}). 
\end{rem}

We are ready to state our main theorems, starting with our main regularity result. 

    \begin{thm}[Smoothness for positive times]\label{thm:smooth}
        Let $f$ be a weak solution of \eqref{eq:main eqn} in the sense of Definition \ref{def:weak sol}, with regular initial datum $f_0$ in the sense of Definition \ref{def:reg initial data}. Then, there holds $f \in C^\infty((0,T)\times\mathbb{R}^3)$. 
    \end{thm}

In order to show the above, we prove that, for a.e.~$t \in (0,T)$, there holds $f \in C^\infty((t,T)\times\mathbb{R}^3)$. Similarly, estimates on the solution will be performed on time intervals $(t,T)$ for $t \in (0,T)$.

Our second main result is the uniqueness of weak solutions, for which we require additional assumptions on the initial data. This result is proved in two stages: first we show uniqueness on a small time interval, and then use Theorem \ref{thm:smooth} to extend to all of $[0,T]$.

\begin{thm}[Uniqueness for stronger initial data]\label{thm:uniqueness}
    Let $f_0 \in {L^2_\per(\Upsilon)}$ be a regular initial datum, in the sense of Definition \ref{def:reg initial data}, satisfying the additional requirements$:$ 
    \begin{equation}\label{eq:assumptions uniqueness}
        \esssup_\Omega \rho_0 < 1, \qquad \rho_0 \in H^1_\per(\Omega). 
    \end{equation}
    Then, the weak solution of \eqref{eq:main eqn} with initial data $f_0$ is unique. 
\end{thm}

This theorem complements \cite[Theorems 2 and 4]{Martin}. A more general uniqueness theory with less restrictive assumptions on the initial data, in particular the condition on $\esssup_\Omega \rho_0$, is the subject of ongoing investigation.

\subsection{Summary of the Strategy}

    Our approach is as follows. 

\begin{itemize}
    \item[(i)] In \S \ref{subsec:angle averaged quantities}, we prove several estimates up to the initial time $t=0$: in particular,  $\sqrt{1-\rho}\nabla \p \in L^2(\Omega_T)$. 
    \item[(ii)] In \S \ref{sec:strong parab}, we prove a lower bound on $1\!-\!\rho$ for positive times by considering the equation for $h(1\!-\!\rho)$, where $h$ is given by Definition \ref{def:reg initial data}, and applying De Giorgi's method (see \S \ref{subsec:strong parab}). This lower bound implies that \eqref{eq:main eqn} is uniformly parabolic for positive times, and $\p \in H^1(\Omega_{t,T})$ a.e.~$t\in(0,T)$. Assuming stronger initial data (\textit{cf.}~assumptions of Theorem \ref{thm:uniqueness}), this lower bound holds up to the initial time $t=0$ (see \S \ref{subsec:strong parab up to t zero}); we remark that it is obtained by a different iteration to that used in \S \ref{subsec:strong parab}. We emphasise that we only require the lower bound to hold for a.e.~$t \in (0,T)$ for the results of \S \ref{sec:boundedness positive times}--\ref{sec:higher reg} to hold; the lower bound up to the initial time $t=0$ is only used for the uniqueness result in \S\ref{sec:uniqueness}.
    \item[(iii)] The goal of \S \ref{sec:boundedness positive times} is threefold: we obtain  $L^4_t W^{2,4}_x$ estimates for $\rho$, perform the $H^1$-estimate for $f$ using a mollification argument, and obtain $f \in L^\infty$ by means of De Giorgi type iterations; we do this in several steps. Firstly, using the $H^1$-regularity of $\p$, we perform the $H^2$-estimate for $\rho$, yielding  $\nabla \rho \in L^4$ by interpolation. This is enough to then perform the $H^2$-estimate for $\p$, by exploiting a cancellation only visible in the non-divergence form of the equation (see \S \ref{sec:proof of H2 via galerkin}). In order to do this, we rely on an application of the parabolic Gehring's Lemma to the equation \eqref{eq:polarisation eq} for $\p$, followed by a Galerkin approximation; interestingly, this Galerkin method is only used to show regularity, not existence. This upgrades the regularity of  $\nabla \p$ to $L^4$, from which we get   $\partial_t \rho, \nabla^2 \rho \in L^4$. We deduce $f \in L^2_t H^1_{x,\theta}$ by a mollification argument, and $f \in L^\infty$ for positive times using De Giorgi's iterations. 
    \item[(iv)] In \S \ref{sec:higher reg}, we prove Theorem \ref{thm:smooth} by bootstrapping our previous observations. We proceed by taking repeated time derivatives in \eqref{eq:main eqn} and following the previous sequence of steps, relying again on the passage from divergence to non-divergence form of \eqref{eq:main eqn} by means of Galerkin's method to obtain $H^2$-estimates of $f$ and its time derivatives. Details of the inductive proof for the bootstrap are contained in Appendix \ref{app: proof of big induction}. 
    \item[(v)] \S \ref{sec:uniqueness} contains the proof of Theorem \ref{thm:uniqueness}. We consider the difference of two solutions $\bar{f} = f_1-f_2$ and $\bar{\rho} = \rho_1-\rho_2$. Using gradient estimates for the periodic heat kernel and Duhamel's principle, we derive the bound $\Vert \bar{\rho} \Vert_{L^\infty(\Omega_t)} \leq \Vert \bar{f} \Vert_{L^2(\Omega_t)}$ for small times $t \in [0,t_*]$. Then, using the additional assumptions on the initial data, we prove $\bar{f} = 0$ in the interval $[0,t_*]$ by performing the usual $H^1$ estimate for $\bar{f}$. Finally, we extend the uniqueness to $[0,T]$ using Theorem \ref{thm:smooth}. 
\end{itemize}

Throughout \S\ref{subsec:angle averaged quantities}--\S\ref{sec:higher reg}, $f$ will always denote a weak solution of \eqref{eq:main eqn}, in the sense of Definition \ref{def:weak sol}, with $\rho$ and $\p$ as prescribed by \eqref{eq:rho def}--\eqref{eq:polarisation def}.

\section{Preliminary Estimates up to $t=0$}\label{subsec:angle averaged quantities}

The purpose of this section is to state several results which hold up to the initial time $t=0$; indeed, many of the estimates which we shall obtain later, especially for higher order derivatives, only hold for positive times. However, in order to successfully prove the uniqueness result of Theorem \ref{thm:uniqueness}, we require certain estimates to hold up to the initial time $t=0$. The results in this section do not require the initial data to satisfy Definition \ref{def:reg initial data}.

In what follows, we make use of the boundedness of $\p$ in order to deduce higher integrability of $\nabla\rho$ via the equation \eqref{eq:rho eq}. To this end, by multiplying \eqref{eq:main eqn} against $\e(\theta)$ and integrating with respect to the angle variable, we see that, by defining the $2\times 2$ matrix $\P$ 
\begin{equation}\label{eq:P matrix def}
    \P(t,x) := \int_0^{2\pi} f(t,x,\theta) \, \e(\theta)\!\otimes\!\e(\theta) \d \theta, 
\end{equation}
the polarisation $\p$, defined in \eqref{eq:polarisation def}, satisfies in the distributional sense the equation 
\begin{equation}\label{eq:polarisation eq}
    \partial_t \p + \dv((1-\rho)\P) = \dv((1-\rho)\nabla \p + \p \otimes \nabla \rho) - \p 
\end{equation}
\textit{i.e.}, in coordinate form with summation convention, with $\p = (p_i)_{i}$ and $\P = (P_{ij})_{ij}$, 
\begin{equation}\label{eq:polarisation eq coord form}
    \partial_t p_i + \partial_j((1-\rho)P_{ij}) = \partial_j((1-\rho)\partial_j p_i + p_i \partial_j \rho) - p_i. 
\end{equation}
Note the uniform boundedness: $|\p(t,x)| \leq \int_0^{2\pi} f \d \theta = \rho \leq 1$, $|\P(t,x)| \leq 4 \int_0^{2\pi} f \d \theta \leq 4$: 
\begin{equation}\label{eq:uniform P}
    \p,\P \in L^\infty(\Omega_T). 
\end{equation}

\begin{rem}[Weak formulation of polarisation equation]\label{rem:sense of p eqn to begin with}
    All terms in \eqref{eq:polarisation eq} are well-defined since $(1-\rho)\nabla \p \in L^{\frac{3}{2}}(\Omega_T)$ from \eqref{eq:early boundedness nabla f weighted in intro} and $\p \otimes \nabla \rho \in L^2(\Omega_T)$ from Definition \ref{def:function space}. By the previous bounds, \eqref{eq:polarisation eq} holds in duality to $L^3(0,T;(W^{1,3}_\per(\Omega))')$, $\partial_t \p \in L^{\frac{3}{2}}(0,T;(W^{1,3}_\per(\Omega))')$. 
\end{rem}

The boundedness \eqref{eq:uniform P} yields the following $H^1$-type estimate for $\p$. 

\begin{lemma}[$H^1$-type estimate for $\p$]\label{lem:H1 p}
  Let $f$ be a weak solution of \eqref{eq:main eqn}, with $\rho$ and $\p$ as per \eqref{eq:rho def}--\eqref{eq:polarisation def}. Then,  if for some $t\geq 0$ it holds $\essinf_{\Omega_{t,T}}(1-\rho)>0$, then there holds  $\nabla \p \in L^2(\Omega_{t,T})$, $\partial_t \p \in L^2(t,T;(H^1_\per(\Omega))')$. 
\end{lemma}
\begin{proof}
    The underlying idea is to test the equation \eqref{eq:polarisation eq} with $\p$ itself. Formally, 
    \begin{equation*}
    \begin{aligned}
         \frac{1}{2}\frac{\der}{\der t}\!\int_\Omega\! |\p|^2 \d x \!+\! \frac{1}{2}\int_\Omega\! (1\!-\!\rho) |\nabla \p|^2 \d x \!\leq\! |\Omega| \!+\! \Vert \P \Vert_{L^\infty(\Omega_T)} \! \int_\Omega\! |\nabla \sqrt{1-\rho}|^2 \d x \!-\! \frac{1}{2}\int_\Omega |\p|^2 \d x, 
    \end{aligned}
\end{equation*}
and integrating in time yields the sought estimate. However, this argument is not rigorous since \emph{a priori} we do not know that $\p$ is an admissible test function. To overcome this, we test the equation with a smoothed version of $\p$ and conclude with a limiting argument. 

    \smallskip 

    \noindent 1. \textit{Mollification}: Let $\eta$ be the usual Friedrichs bump function, and $\{\eta^\varepsilon\}_{\varepsilon>0}$ the usual sequence of Friedrichs mollifiers in $\mathbb{R}^2$; we only mollify with respect to $x$. In what follows, we use the notations $\p^\varepsilon = \eta^\varepsilon*\p$, $D = (1-\rho)\P$, and $D^\varepsilon = \eta^\varepsilon*D$. Since $\eta^\varepsilon$ is an admissible test function to insert into the weak formulation of \eqref{eq:polarisation eq}, we mollify the equation: 
    \begin{equation}\label{eq:mollified p eqn}
        \partial_t \p^\varepsilon + \dv D^\varepsilon = \dv((1-\rho)\nabla\p^\varepsilon + \p^\varepsilon \otimes \nabla \rho) - \p^\varepsilon + \dv (E^\varepsilon + F^\varepsilon), 
    \end{equation}
    where the error terms $E^\varepsilon$ and $F^\varepsilon$ are given by 
    \begin{equation}\label{eq:error terms mollified p eqn}
        E^\varepsilon = \eta^\varepsilon*((1-\rho)\nabla \p) - (1-\rho)\eta^\varepsilon*\nabla \p , \qquad F^\varepsilon = \eta^\varepsilon*(\p \otimes \nabla \rho) - (\eta^\varepsilon*\p )\otimes \nabla \rho. 
    \end{equation}

    \smallskip 

    \noindent 2. \textit{Estimates on drift term and error terms}: We estimate the error terms $D^\varepsilon, E^\varepsilon, F^\varepsilon$. Our aim is to show the following: there exists a positive constant $C$ independent of $\varepsilon$ such that 
    \begin{equation}\label{eq:unif bounds for H1 for p with mollifications}
        \Vert D^\varepsilon \Vert_{L^\infty(\Omega_T)} + \Vert E^\varepsilon \Vert_{L^2(\Omega_T)} + \Vert F^\varepsilon \Vert_{L^2(\Omega_T)} \leq C. 
    \end{equation}
    Once the above is established, we will be able to proceed as per the formal argument presented at the start of the proof. First, $\Vert D^\varepsilon (t,\cdot) \Vert_{L^\infty(\Omega)} \leq C \Vert D(t,\cdot) \Vert_{L^\infty(\Omega)} \leq C \Vert \P \Vert_{L^\infty(\Omega_T)} \leq C,$ for some positive constant $C$ independent of $\varepsilon$, where we used standard properties of the mollifier to establish the first inequality. It follows from taking the essential supremum in the time variable that, for some positive constant independent of $\varepsilon$, 
    \begin{equation}\label{eq:first error term H1 proof p}
        \Vert D^\varepsilon \Vert_{L^\infty(\Omega_T)} \leq C. 
    \end{equation}
Next, the error term $F^\varepsilon$ is controlled using elementary properties of the mollifier, as follows: 
\begin{equation*}
    \Vert \eta^\varepsilon*(\p \otimes \nabla \rho)(t,\cdot) \Vert_{L^2(\Omega)} \leq C \Vert (\p \otimes \nabla \rho) (t,\cdot) \Vert_{L^2(\Omega)} \leq C \Vert \p \Vert_{L^\infty(\Omega_T)} \Vert \nabla \rho(t,\cdot) \Vert_{L^2(\Omega)}, 
\end{equation*}
for a.e.~$t \in(0,T)$, for some positive constant $C$ independent of $\varepsilon$, and similarly 
\begin{equation*}
    \Vert (\eta^\varepsilon*\p )\otimes \nabla \rho (t,\cdot) \Vert_{L^2(\Omega)} \!\leq\! \Vert (\eta^\varepsilon*\p )(t,\cdot) \Vert_{L^\infty(\Omega)}\Vert \nabla \rho(t,\cdot) \Vert_{L^2(\Omega)} \!\leq\! C \Vert \p \Vert_{L^\infty(\Omega_T)} \Vert \nabla \rho(t,\cdot) \Vert_{L^2(\Omega)}, 
\end{equation*}
whence, by integrating in time, there exists a positive $C$ independent of $\varepsilon$ such that 
\begin{equation}\label{eq:second error term H1 proof p}
    \Vert F^\varepsilon \Vert_{L^2(\Omega_T)} \leq C. 
\end{equation}
For $E^\varepsilon$, write $E^\varepsilon(t,x) = \int_{\mathbb{R}^2} \eta^\varepsilon(x-y)\big[ \rho(t,x)-\rho(t,y) \big]\nabla\p(t,y) \d y$. Integrating by parts, 
\begin{equation*}
   \begin{aligned}
       E^\varepsilon(t,x) =\!\!  \underbrace{\int_{\mathbb{R}^2} \!\!\p(t,y) \otimes \nabla \eta^\varepsilon(x-y) \big[ \rho(t,x)-\rho(t,y) \big] \d y}_{=: E^\varepsilon_1(t,x)} + \!\underbrace{\int_{\mathbb{R}^2} \!\!\eta^\varepsilon(x-y) \p(t,y) \otimes \nabla \rho(t,y) \d y}_{=: E^\varepsilon_2(t,x)} . 
   \end{aligned} 
\end{equation*}
Using standard properties of mollifiers, $\Vert E^\varepsilon_2(t,\cdot) \Vert_{L^2(\Omega)} \leq C \Vert \nabla \rho(t,\cdot) \Vert_{L^2(\Omega)} \Vert \p \Vert_{L^\infty(\Omega_T)},$ hence 
\begin{equation}\label{eq:Eeps2 bound unif}
    \Vert E^\varepsilon_2 \Vert_{L^2(\Omega_T)} \leq C \Vert \nabla \rho \Vert_{L^2(\Omega_T)} \Vert \p \Vert_{L^\infty(\Omega_T)}, 
\end{equation}
for some positive constant $C$ independent of $\varepsilon$. Meanwhile, note that, for a.e.~$x,y\in \RR^2$,
\begin{equation*}
   \begin{aligned} 
  \rho(t,x)-\rho(t,y) = \int_0^1 \frac{\der}{\der \sigma}\rho(t,\sigma x +(1-\sigma)y)  \d \sigma = \int_0^1 \nabla \rho(t,\sigma x +(1-\sigma)y) \!\cdot\! (x-y)  \d \sigma, 
   \end{aligned}
\end{equation*}
and since $\supp \eta^\varepsilon = B(0,\varepsilon)$, it follows that, for all $y \in \supp \eta^\varepsilon(x-\cdot)$, there holds 
\begin{equation}\label{eq:rho diff for mollification argument}
   \begin{aligned} 
  |\rho(t,x)-\rho(t,y)| \leq \varepsilon \int_0^1 |\nabla \rho(t,\sigma x +(1-\sigma)y)| \d \sigma. 
   \end{aligned}
\end{equation}
In turn, using the uniform boundedness of $\p$, there holds 
\begin{equation*}
    \begin{aligned}
        E^\varepsilon_1(t,x) &\leq \varepsilon^{-2} \int_{B(x,\varepsilon)} \big|\nabla \eta(\frac{x-y}{\varepsilon}) \big| \bigg(  \int_0^1 |\nabla \rho(t,\sigma x +(1-\sigma)y)| \d \sigma\bigg) \d y \\ 
        &= \int_{B(0,1)} |\nabla \eta (w)| \bigg(  \int_0^1 |\nabla \rho \big(t, x - (1-\sigma)\varepsilon w \big)| \d \sigma\bigg) \d w, 
    \end{aligned}
\end{equation*}
and, by Jensen's inequality, we get $|E^\varepsilon_1(t,x)|^2 \leq C \int_{B(0,1)} \int_0^1 |\nabla \rho \big(t, x - (1-\sigma)\varepsilon w \big)|^2 \d \sigma \d w$, for some positive constant $C$ independent of $\varepsilon$. An application of the Fubini--Tonelli Theorem then implies, using periodicity of $\nabla \rho$ and translation invariance of Lebesgue measure, 
\begin{equation*}
    \Vert E^\varepsilon_1(t,\cdot) \Vert^2_{L^2(\Omega)} \leq C \int_{B(0,1)} \int_0^1 \int_\Omega |\nabla \rho(t,x)|^2 \d x \d \sigma \d w \leq C \Vert \nabla \rho(t,\cdot) \Vert^2_{L^2(\Omega)}. 
\end{equation*}
By integrating in time, and combining with \eqref{eq:Eeps2 bound unif}, we deduce $\Vert E^\varepsilon \Vert_{L^2(\Omega_T)} \leq C$, independently of $\varepsilon$. Combining this bound with \eqref{eq:first error term H1 proof p} and \eqref{eq:second error term H1 proof p}, we obtain \eqref{eq:unif bounds for H1 for p with mollifications}. 

\smallskip 

    \noindent 3. \textit{$H^1$-type estimate}: Note that $\p^\varepsilon \in L^\infty(0,T;C^\infty_{\per}(\Omega))$, whence, in view of Remark \ref{rem:sense of p eqn to begin with}, $\p^\varepsilon$ is an admissible test function to insert in the weak formulation of \eqref{eq:polarisation eq}. Testing the equation with $\p^\varepsilon$, using the uniform bounds \eqref{eq:unif bounds for H1 for p with mollifications} and Young's inequality, we get 
    \begin{equation*}
        \begin{aligned}
            \frac{1}{2}\frac{\der}{\der t}\int_\Omega |\p^\varepsilon|^2 \d x + \essinf_{\Omega_{t,T}}(1-\rho) \int_\Omega |\nabla \p^\varepsilon|^2 \d x \leq C_{\essinf_{\Omega_{t,T}}(1-\rho)} + \Vert \nabla \sqrt{1-\rho(t,\cdot)} \Vert^2_{L^2(\Omega)} , 
        \end{aligned}
    \end{equation*}
with $C$ independent of $\varepsilon$. Integrating in time, as $\Vert \p^\varepsilon(0,\cdot)\Vert_{L^\infty(\Omega)} \leq C \Vert \p(0,\cdot)\Vert_{L^\infty(\Omega)} \leq C$, 
    \begin{equation}\label{eq:just before H1 final bound for p}
        \begin{aligned}
            \Vert \p^\varepsilon\Vert_{L^\infty(t,T;L^2(\Omega))}^2 + \Vert \nabla \p^\varepsilon\Vert_{L^2(\Omega_{t,T})}^2 &\leq C\big( 1+ \Vert \nabla \sqrt{1-\rho} \Vert^2_{L^2(\Omega_T)}). 
        \end{aligned}
    \end{equation}
    We take the limit $\varepsilon \to 0$, and the result follows from the lower semicontinuity of the norms. 

    \smallskip 

    \noindent 4. \textit{Boundedness of time derivative}: We obtain $\partial_t \p \in L^2(t,T;(H^1(\Omega))')$ directly from the equation, using the boundedness $\sqrt{1-\rho}\nabla \p \in L^2(\Omega_T)$ and $\nabla \rho \in L^2(\Omega_T)$. 
\end{proof}

More generally, we see that the $n$-th order tensor $\pi^n$, defined by 
\begin{equation}\label{eq:tensor def}
    \pi^n(t,x) := \int_0^{2\pi} f(t,x,\theta) \, \underbrace{\e(\theta)\!\otimes\!\dots \! \otimes \!\e(\theta)}_{n \text{ times}} \d \theta,
\end{equation}
satisfies the equation 
\begin{equation}\label{eq:tensor order n eq}
    \partial_t \pi^n + \dv((1-\rho)\pi^{n+1}) = \dv((1-\rho)\nabla \pi^n + \pi^n \otimes \nabla \rho) + \tilde{\pi}^n 
\end{equation}
in the sense of distributions, where 
\begin{equation*}
    \tilde{\pi}^n := \int_0^{2\pi} f(t,x,\theta) \, \partial^2_\theta (\underbrace{\e(\theta)\!\otimes\!\dots \! \otimes \!\e(\theta)}_{n \text{ times}}) \d \theta. 
\end{equation*}
As per \eqref{eq:uniform P}, we get $|\pi^n(t,x)|, |\tilde{\pi}^n(t,x)| \leq n^2 \int_0^{2\pi} f \d \theta \leq n^2$, from which it follows 
\begin{equation}\label{eq:tensor unif bded}
    \pi^n , \tilde{\pi}^n \in L^\infty(\Omega_T) \quad \text{for all } n \in \mathbb{N}. 
\end{equation}
As per Lemma \ref{lem:H1 p}, this boundedness yields a $H^1$-type estimate for all tensors $\pi^n$. 

\begin{lemma}[$H^1$-type estimate for $\pi^n$]\label{lem:H1 for tensors}
Let $f$ be a weak solution of \eqref{eq:main eqn}, $n\in\mathbb{N}$, and $\pi^n$ be as per \eqref{eq:tensor def}. Then, if for some $t\geq 0$ it holds $\essinf_{\Omega_{t,T}}(1-\rho)>0$, then there holds $\nabla \pi^n \in L^2(\Omega_{t,T})$ and $\partial_t \pi^n \in L^2(t,T;(H^1_\per(\Omega))')$.
\end{lemma}
\begin{proof}
We test the equation \eqref{eq:tensor order n eq} against $\pi^n$, which can be made rigorous by simple adaptations of the proof of Lemma \ref{lem:H1 p}; we omit the details for succinctness. We obtain 
\begin{equation*}
   \begin{aligned} \frac{1}{2}\frac{\der}{\der t}\int_\Omega |\pi^n|^2 \d x &+ \int_\Omega (1-\rho)|\nabla \pi^n|^2 \d x \\ 
   &= \int_\Omega (1-\rho) \nabla \pi^n : \pi^{n+1} \d x -\int_\Omega \nabla \pi^n : (\pi^n \otimes \nabla \rho) \d x + \int_\Omega \pi^n : \tilde{\pi}^n \d x, 
   \end{aligned}
\end{equation*}
whence, using the boundedness \eqref{eq:tensor unif bded} and Young's inequality, 
\begin{equation*}
   \begin{aligned} \Vert\pi^n\Vert_{L^\infty(0,T;L^2(\Omega))}^2 + \Vert \sqrt{1-\rho}\nabla \pi^n &\Vert_{L^2(\Omega_T)}^2 \d x 
  \\ &\leq \Vert \pi^n(0,\cdot) \Vert^2_{L^2(\Omega)} + 4(n+1)^4|\Omega_T| + 2\Vert \nabla \sqrt{1-\rho}\Vert^2_{L^2(\Omega_T)}.
   \end{aligned}
\end{equation*}
We then deduce the boundedness of $\partial_t \pi^n$ directly from the equation \eqref{eq:tensor order n eq}. 
\end{proof}

\section{Strong Parabolicity}\label{sec:strong parab}

The purpose of this section is to obtain lower bounds on the quantity $1\!-\!\rho$. We will then obtain $H^1$-type bounds on $f$ by performing the classical parabolic estimate on \eqref{eq:main eqn}.

\subsection{Initial upgrade of integrability}

We first write down a lemma which upgrades the integrability of $\partial_t \rho,\nabla^2 \rho$. This first result does not require the initial data to satisfy Definition \ref{def:reg initial data}, merely that $f \in \mathcal{X}$ is a weak solution of \eqref{eq:main eqn} obtained from Theorem \ref{rem:existence of weak sol}. The gain of integrability from this result will then be used in \S \ref{subsec:strong parab} to obtain a uniform lower bound on $1-\rho$ for positive times, provided the initial data satisfies Definition \ref{def:reg initial data}.

\begin{lemma}[Initial upgrade of integrability]\label{lem:heat eq rho in L3/2}
Let $f$ be a weak solution of \eqref{eq:main eqn}, with $\rho$ as per \eqref{eq:rho def}. For a.e.~$t \in (0,T)$, there holds $\partial_t \rho, \nabla^2 \rho \in L^{\frac{3}{2}}(\Omega_{t,T})$ and $\nabla \rho \in L^{3}(\Omega_{t,T}).$
\end{lemma}

\begin{proof}
We have $\dv((1-\rho)\p) = -{\nabla \rho \cdot \p} + (1-\rho) \dv \p$. Observe that $\nabla \rho \cdot \p \in L^2(\Omega_T)$, while, using  Definition \ref{def:function space}, $\sqrt{1-\rho}\nabla_{\bxi}\sqrt{f} \in L^2(\Upsilon_T)$, and $0 \leq \rho \leq 1$, and Jensen's inequality, 
\begin{equation*}
   \begin{aligned} \int_{\Omega_T} |(1-\rho) \dv \p |^{\frac{3}{2}} \d x \d t &= \int_{\Omega_T} (1-\rho)^{\frac{3}{2}} \bigg| \int_0^{2\pi} \nabla f \cdot \e(\theta) \d \theta \bigg|^{\frac{3}{2}} \d x \d t \\ 
   &\leq \sqrt{2\pi}\int_{\Upsilon_T} \bigg[\sqrt{1-\rho} \frac{|\nabla f |}{\sqrt{f}}\bigg]^{\frac{3}{2}} \sqrt{f}^{\frac{3}{2}} \d \bxi \d t \\ 
   &\leq \sqrt{2\pi} \, 4^{\frac{3}{4}} \Vert \sqrt{f} \Vert_{L^{6}(\Upsilon_T)}^{\frac{3}{2}} \Vert \sqrt{1-\rho}\nabla \sqrt{f} \Vert_{L^2(\Upsilon_T)}^{\frac{3}{2}}, 
   \end{aligned}
\end{equation*}
where, to obtain the final line, we used H\"older's inequality with the exponent $4/3$. It therefore follows that  $\dv((1-\rho)\p) \in L^{\frac{3}{2}}(\Omega_T)$, whence the equation \eqref{eq:rho eq} for the angle-averaged density $\rho$ implies the inclusion $\partial_t \rho - \Delta \rho \in L^{\frac{3}{2}}(\Omega_T)$. An application of the standard Calder\'on--Zygmund estimate for the heat equation (\textit{e.g.}~\cite[Theorem 3.1]{RosHeat}) yields the local estimates 
\begin{equation*}
    \partial_t \rho , \nabla^2 \rho \in L^{\frac{3}{2}}(\Omega_{t,T}) \quad \text{for a.e.~}t\in(0,T); 
\end{equation*}
we emphasise that we need only restrict to a subdomain with respect to the time variable due to the periodicity in space-angle (\textit{cf.}~\textit{e.g.}~the argument in \cite[\S 3.1.3]{reg1}). Additionally, an application of the Gagliardo--Nirenberg inequality on the bounded domain $\Omega$ yields 
\begin{equation*}
    \Vert \nabla \rho(t,\cdot) \Vert_{L^{3}(\Omega)} \leq C \Big(\Vert \nabla^2 \rho(t,\cdot) \Vert_{L^{\frac{3}{2}}(\Omega)}^{\frac{1}{2}} \Vert \rho(t,\cdot) \Vert_{L^{\infty}(\Omega)}^{\frac{1}{2}} + \Vert \rho(t,\cdot) \Vert_{L^{\infty}(\Omega)} \Big) \quad \text{a.e.~}t \in (0,T), 
\end{equation*}
\textit{i.e.}, $\Vert \nabla \rho(t,\cdot) \Vert_{L^{3}(\Omega)}^{3} \leq C \big(1+\Vert \nabla^2 \rho(t,\cdot) \Vert_{L^{\frac{3}{2}}(\Omega)}^{\frac{3}{2}} \big)$, and thus, by integrating in time, 
\begin{equation*}
    \Vert \nabla \rho \Vert_{L^{3}(\Omega_{t,T})} \leq C \big( T +\Vert \nabla^2 \rho \Vert_{L^{\frac{3}{2}}(\Omega_{t,T})}^{\frac{3}{2}} \big)^{\frac{1}{3}}, 
\end{equation*}
and the result follows. 
\end{proof}

\subsection{Strong Parabolicity for Positive Times}\label{subsec:strong parab}

Our goal in this subsection is to obtain a \emph{positive} uniform lower bound on the quantity 
\begin{equation*}
    u = 1- \rho, 
\end{equation*}
which satisfies $0 \leq u \leq 1$ a.e., $\nabla u \in L^2(\Omega_T)$, $\partial_t u \in L^2(0,T;(H^1_\per)'(\Omega))$, and 
\begin{equation}\label{eq: eqn 1-rho}
    \partial_t u -   \dv(\p u) = \Delta u.  
\end{equation}

For the remainder of this section and throughout \S \ref{sec:boundedness positive times}--\S \ref{sec:uniqueness}, we assume that the initial data satisfies the assumptions of Definition \ref{def:reg initial data}. The main result of this subsection is as follows. 

\begin{prop}[Lower bound on $1\!-\!\rho$]\label{prop:lower bound 1-rho}
    Let $f$ be a weak solution of \eqref{eq:main eqn} with regular initial data satisfying Definition \ref{def:reg initial data}. Then, there exists $c:(0,T) \to (0,\infty)$ such that 
    \begin{equation*}
    \essinf_{(t,T)\times\Omega} (1-\rho) \geq c(t) \quad \text{for a.e.~} t \in (0,T). 
\end{equation*}
\end{prop}

Our strategy is as follows. We study the evolution of the quantity 
\begin{equation*}
       v := h(u), 
    \end{equation*}
with $h \in C^2(\mathbb{R}_+)$ satisfying the assumptions of Definition \ref{def:reg initial data}. Formally, $v$ is a solution of 
\begin{equation*}
    \partial_t v  - \dv\big( \p u h'(u) \big) + \frac{u h''(u)}{h'(u)} \p \! \cdot \! \nabla v  = \Delta v - \frac{h''(u)}{h'(u)^2}{|\nabla v|^2}, 
\end{equation*}
as can be verified by formal computation using \eqref{eq:rho eq}. Testing the above with $v$, we see that the usual $H^1$-type estimate closes to give the boundedness of $v$ in $L^\infty(0,T;L^2(\Omega_T))$ and $\nabla v$ in $L^2(\Omega_T)$; this relies crucially on the sign of the final term on the right-hand side of the equation. Then, by interpreting $v$ as a nonnegative subsolution of a strongly parabolic equation without this final term, an application of De Giorgi's method will yield the boundedness of $v$ in $L^\infty$ for positive times. This implies a lower bound on $1-\rho$, making \eqref{eq:main eqn} strongly parabolic away from $t=0$. The purpose of the next proposition is to rigorously derive the equation for $v$.

\begin{lemma}\label{prop: w big prop}
    Let $f$ be a weak solution of \eqref{eq:main eqn} with regular initial data $f_0$ satisfying Definition \ref{def:reg initial data}, $u=1-\rho$, and $0 \leq h \in C^2(\mathbb{R}_+)$ satisfy \eqref{eq:conditions on h}--\eqref{eq:additional condition for strong parab}. Then, $v = h(u)$ satisfies $v  \in L^\infty(0,T;L^2_\per(\Omega)) \cap L^2(0,T;H^1_\per(\Omega)) \subset  L^4(\Omega_T) $. Additionally, there holds the estimate 
    \begin{equation}\label{eq: energy est 1 over sqrt}
        \Vert v \Vert^2_{L^\infty(0,T;L^2(\Omega))} + \Vert \nabla v \Vert^2_{L^2(\Omega_T)} \leq   CT e^{C T} \Vert h(1\!-\!\rho_0) \Vert^2_{L^2(\Omega)}, 
    \end{equation}
for some positive constant $C=C(\Omega)$. Furthermore, in duality with $L^2(0,T;H^1_\per(\Omega))$, $v$ is a weak subsolution of the semilinear equation 
    \begin{equation}\label{eq: eqn for w in lem}
   \partial_t v  - \dv\big( \p u h'(u) \big) +  \frac{u h''(u)}{h'(u)}   \p \! \cdot \! \nabla v  = \Delta v. 
\end{equation}
\end{lemma}

To prove Lemma \ref{prop: w big prop}, we use the result of Lemma \ref{lem:heat eq rho in L3/2}, which transfers trivially to $u=(1-\rho)$: 
\begin{equation}\label{eq:schauder u}
    \partial_t u , \nabla^2 u \in L^{\frac{3}{2}}(\Omega_{t,T}), \quad \nabla u \in L^{3}(\Omega_{t,T}) . 
\end{equation}

\begin{proof}
In order to obtain the estimate \eqref{eq: energy est 1 over sqrt}, we test the equation \eqref{eq: eqn 1-rho} for $u$ with $h(u) h'(u)$. Our reason for doing so is as follows: formally $h'(u)\partial_t u = \partial_t v$ and thus, in order to recover $\partial_t v^2 = 2 v \partial_t v$, we test $\partial_t u$ against $h(u) h'(u)$. In order to do this rigorously, we must test against a regularised version of this quantity. 

\smallskip

\noindent 1. \textit{Estimate on $v$}: We want to derive estimates for $v = h(u)$, however this quantity is not bounded \emph{a priori}, since we only know $0 < u \leq 1$ a.e. To circumvent this problem, define 
\begin{equation*}
    v_\varepsilon := h({u+\varepsilon}) \qquad \text{for } \varepsilon>0; 
\end{equation*}
the above is bounded since the monotonicity of $h$ and $h'$ imply, a.e.~in $\Omega_T$, for all $\varepsilon>0$ 
\begin{equation}\label{eq: boundedness w eps}
  h(1+\varepsilon) \leq v_\varepsilon \leq h(\varepsilon), \quad h'(\varepsilon) \leq h'(u+\varepsilon) \leq h'(1+\varepsilon) < 0, \quad 0 < h''(u+\varepsilon) \leq \max_{[\varepsilon,1+\varepsilon]}h''. 
    \end{equation}
Furthermore, in the sense of distribution, there holds 
\begin{equation}\label{eq: expression nabla v}
    \nabla v_\varepsilon = h'(u+\varepsilon) \nabla u \in L^2(\Omega_T), 
\end{equation}
where we used that $\nabla u \in L^2(\Omega_T)$ and the boundedness \eqref{eq: boundedness w eps}. We now insert $v_\varepsilon h'(u+\varepsilon)$ as a test function into the weak formulation of \eqref{eq: eqn 1-rho}. We get 
\begin{equation*}
    \langle \partial_t u ,\! v_\varepsilon h'(u+\varepsilon)  \rangle \!-  \!\int_\Omega \!\!\big(h'(u+\varepsilon)^2 + v_\varepsilon h''(u+\varepsilon)\big)\nabla u \cdot  \p u \d x \!=\! \!-\!\!\int_\Omega \!\!\!\big(h'(u+\varepsilon)^2 + v_\varepsilon h''(u+\varepsilon)\big) |\nabla u|^2 \d x. 
\end{equation*}
We rewrite the final term on the right-hand side. Using the chain rule for Sobolev functions belonging to $H^1(\Omega)$ for a.e.~$t \in (0,T)$, \textit{cf.}~\textit{e.g.}~\cite[\S 5.10]{Evans}, we get $h'(u+\varepsilon)^2|\nabla u|^2 = |\nabla v_\varepsilon|^2$ and $v_\varepsilon h''(u+\varepsilon)|\nabla u|^2 = v_\varepsilon \frac{h''(u+\varepsilon)}{h'(u+\varepsilon)^2} |\nabla v_\varepsilon|^2$, and similarly, there holds 
\begin{equation*}
   \begin{aligned} h'(u+\varepsilon)^2 \nabla u \cdot \p u &= \p \Big( \frac{u}{u+\varepsilon}\Big) (u+\varepsilon) h'(u+\varepsilon)\cdot \nabla v_\varepsilon, \\ 
        v_\varepsilon h''(u+\varepsilon) \nabla u \cdot \p u &= (u+\varepsilon) \frac{h''(u+\varepsilon)}{h'(u+\varepsilon)}\Big(\frac{u}{u+\varepsilon} \Big) \p \! \cdot \! v_\varepsilon \nabla v_\varepsilon, 
        \end{aligned}
\end{equation*}
and all the equalities hold in $L^1(\Omega)$ for a.e.~$t \in (0,T)$. In turn, 
\begin{equation*}
    \begin{aligned} \langle \partial_t u , & v_\varepsilon h'(u+\varepsilon) \rangle + \int_\Omega |\nabla v_\varepsilon|^2 \d x  + \int_\Omega v_\varepsilon \frac{h''(u+\varepsilon)}{h'(u+\varepsilon)^2} |\nabla v_\varepsilon|^2 \d x \\ 
  &=   \int_\Omega \p \Big( \frac{u}{u+\varepsilon}\Big) (u+\varepsilon) h'(u+\varepsilon)\cdot \nabla v_\varepsilon \d x + \int_\Omega (u+\varepsilon) \frac{h''(u+\varepsilon)}{h'(u+\varepsilon)}\Big(\frac{u}{u+\varepsilon} \Big) \p \! \cdot \! v_\varepsilon \nabla v_\varepsilon \d x, 
  \end{aligned}
\end{equation*}
from which we deduce, using the positivity of the integrand of the final integral on the left-hand side, the boundedness of $\p$, the assumptions on $h$ from Definition \ref{def:reg initial data}, and \eqref{eq:third condition on initial data}, 
\begin{equation}\label{eq:by returning to}
    \begin{aligned} \langle \partial_t u , v_\varepsilon h'(u+\varepsilon) \rangle + \int_\Omega |\nabla v_\varepsilon|^2 \d x  \leq  C \int_\Omega (1+v_\varepsilon) |\nabla v_\varepsilon| \d x. 
  \end{aligned}
\end{equation}
It remains to rewrite the duality product. Note $v_\varepsilon h'(u+\varepsilon) \in L^\infty(\Omega_T) \cap L^2(0,T;H^1_\per(\Omega))$, $\partial_t u \in L^{\frac{3}{2}}(\Omega_{t,T})$ {a.e.~}$t\in(0,T).$ Thus, again by applying the chain rule for Sobolev functions, 
\begin{equation*}
    \langle \partial_t u , v_\varepsilon h'(u+\varepsilon)  \rangle = \frac{1}{2} \langle \partial_t u ,(h^2)'(u+\varepsilon)  \rangle = \frac{1}{2}\frac{\der}{\der t}\int_\Omega h(u+\varepsilon)^{2} \d x = \frac{1}{2}\frac{\der}{\der t}\int_\Omega v_\varepsilon^2 \d x, 
\end{equation*}
and hence returning to \eqref{eq:by returning to} and using Young's inequality, we obtain, for a.e.~$t \in (0,T)$, 
\begin{equation*}
    \frac{\der}{\der t}\Vert v_\varepsilon(t,\cdot) \Vert^2_{L^2(\Omega)} + \Vert \nabla v_\varepsilon(t,\cdot) \Vert^2_{L^2(\Omega)} \leq   C (1 + \Vert v_\varepsilon(t,\cdot) \Vert^2_{L^2(\Omega)}). 
\end{equation*}
Applying Gr\"onwall's inequality, we get 
\begin{equation}\label{eq:veps est after gronwall}
    \esssup_{t \in [0,T]}\Vert v_\varepsilon(t,\cdot) \Vert^2_{L^2(\Omega)} + \Vert \nabla v_\varepsilon \Vert^2_{L^2(\Omega_T)} \leq  CT \Vert v_\varepsilon(0,\cdot) \Vert_{L^2(\Omega)} e^{ CT}. 
\end{equation}
Note that, for all $\varepsilon>0$, the monotonicity of $h$ implies 
$\Vert v_{\varepsilon}(0,\cdot) \Vert_{L^2(\Omega)}^2 \leq \Vert h(1\!-\!\rho_0) \Vert_{L^2(\Omega)}^2 < \infty;$
this latter term is finite by virtue of assumption \eqref{eq:additional condition for strong parab}. In particular, we have 
\begin{equation}\label{eq: bound sup nabla weps}
    \sup_{\varepsilon>0}\Vert \nabla v_\varepsilon \Vert^2_{L^2(\Omega_T)} \leq  CT \Vert h(1\!-\!\rho_0) \Vert_{L^2(\Omega)}^2 e^{C T}. 
\end{equation}
Since $v_\varepsilon \to v$ a.e.~$\Omega_T$, the weak lower semicontinuity of the norms and \eqref{eq:veps est after gronwall} imply \eqref{eq: energy est 1 over sqrt}. Lemma \ref{lem:dibenedetto classic} imply $v \in L^4(\Omega_T)$ , and, using the monotonicity of $h$, 
\begin{equation}\label{eq: L4 bound for weps}
   \sup_{\varepsilon>0}\Vert v_\varepsilon \Vert_{L^4(\Omega_T)} \leq \Vert v \Vert_{L^4(\Omega_T)} < \infty. 
\end{equation}

\smallskip 

\noindent 2. \textit{Weak subsolution}: As per \eqref{eq: expression nabla v}, we compute the distributional derivative 
\begin{equation}\label{eq: t deriv expression weps}
    \partial_t v_\varepsilon = h'(u+\varepsilon) \partial_t u \in L^{\frac{3}{2}}(\Omega_{t,T}), 
\end{equation}
where we used that $\partial_t u \in L^{\frac{3}{2}}(\Omega_{t,T})$ from Lemma \ref{lem:heat eq rho in L3/2} and the boundedness \eqref{eq: boundedness w eps}. It follows from the equation \eqref{eq: eqn 1-rho} for $u$ that the following equality holds as elements of $L^{\frac{3}{2}}(\Omega_{t,T})$: 
\begin{equation}\label{eq:expand dt v eps}
    \partial_t v_\varepsilon = {h'(u+\varepsilon)}   {\dv(u\p)} + {h'(u+\varepsilon)}  {\Delta u}. 
\end{equation}
We write down both the drift and diffusion terms in terms of $v_\varepsilon$. Observe from the formula \eqref{eq: expression nabla v} that, by the chain rule for Sobolev functions, $\Delta v_\varepsilon = h''(u+\varepsilon)|\nabla u|^2 + h'(u+\varepsilon) \Delta u \in L^{\frac{3}{2}}(\Omega_{t,T})$, where we used the boundedness \eqref{eq: boundedness w eps}. In consequence, \eqref{eq:expand dt v eps} may be rewritten 
\begin{equation*}
    \begin{aligned}\partial_t v_\varepsilon = {h'(u+\varepsilon)}    \dv(u\p) + \Delta v_\varepsilon - \frac{h''(u+\varepsilon)}{h'(u+\varepsilon)^2}|\nabla v_\varepsilon|^2 , 
    \end{aligned}
\end{equation*}
where the equality holds in  $L^{\frac{3}{2}}(\Omega_{t,T})$, and we rewrite the first term on the right-hand as 
\begin{equation*}
    \begin{aligned}
        {h'(u+\varepsilon)}    \dv(u\p) &= \dv\big( \p u h'(u+\varepsilon) \big) - \frac{u h''(u+\varepsilon)}{h'(u+\varepsilon)}\nabla v_\varepsilon \cdot \p. 
    \end{aligned}
\end{equation*}
In turn, there holds, as an equality in $L^2(\Omega_{t,T})$ for a.e.~$t \in (0,T)$ and in $L^2(0,T;(H^1_\per(\Omega))')$, 
\begin{equation}\label{eq:eqn for v eps}
    \begin{aligned}\partial_t v_\varepsilon - \dv\big( \p u h'(u+\varepsilon) \big) + \frac{u h''(u+\varepsilon)}{h'(u+\varepsilon)}\nabla v_\varepsilon \!\cdot \!\p = \Delta v_\varepsilon - \frac{h''(u+\varepsilon)}{h'(u+\varepsilon)^2}|\nabla v_\varepsilon|^2. 
    \end{aligned}
\end{equation}
We verify that $v$ is a subsolution of \eqref{eq: eqn for w in lem}. By integrating against a nonnegative test function $\varphi \in C^\infty_c(0,T;C^\infty_\per(\Omega))$ and using the signs of the derivatives of $h$, we obtain 
\begin{equation*}
   \begin{aligned} -\!\!\int_{\Omega_T} \!\!\!\! v_\varepsilon \partial_t \varphi \d x \d t \! + \!\! \int_{\Omega_T}\!\!\! \p u h'(u\!+\!\varepsilon)\! \cdot\! \nabla \varphi  \d x \d t \! + \! \! \int_{\Omega_T} \!\!\!\frac{u h''(u\!+\!\varepsilon)}{h'(u\!+\!\varepsilon)} \nabla v_\varepsilon \!\cdot \! \p \varphi \d x \d t \!\leq\! \! -\!\! \int_{\Omega_T} \!\!\!\!\!\nabla v_\varepsilon\!\cdot\!\nabla \varphi \d x \d t. 
   \end{aligned}
\end{equation*}
Let $\varepsilon \to 0$. Since $v_\varepsilon \to v$ a.e.~and weakly in $L^2(0,T;H^1(\Omega))$, the first term and final terms of the inequality converge to $-\int v \partial_t \varphi$ and $-\int \nabla v \!\cdot\! \varphi$, respectively. For the second term, 
\begin{equation*}
    \begin{aligned}
        \int_{\Omega_T} \p u h'(u+\varepsilon) \cdot \nabla \varphi  \d x \d t = \int_{\Omega_T} \p \overbrace{\Big( \frac{u}{u+\varepsilon}\Big)}^{\leq 1}  (u+\varepsilon) h'(u+\varepsilon)  \cdot \nabla \varphi  \d x \d t, 
    \end{aligned}
\end{equation*}
and note from \eqref{eq:third condition on initial data} and the monotonicity of $h$ that $(u+\varepsilon) |h'(u+\varepsilon)| \leq (v+M) \in L^2(\Omega_T).$ It follows that the integrand $\p u h'(u+\varepsilon) \cdot \nabla \varphi$ is uniformly bounded by the integrable function $(v+M)|\nabla \varphi|$, and converges to $\p u h'(u) \!\cdot\! \nabla \varphi$ a.e.~in $\Omega_T$. The Dominated Convergence Theorem implies $\lim_{\varepsilon \to 0}\int_{\Omega_T} \p u h'(u+\varepsilon) \!\cdot\! \nabla \varphi  \d x \d t = \int_{\Omega_T} \p u h'(u) \!\cdot \!\nabla \varphi  \d x \d t$. Similarly, 
\begin{equation*}
    \int_{\Omega_T} \frac{u h''(u+\varepsilon)}{h'(u+\varepsilon)} \nabla v_\varepsilon \! \cdot \! \p \varphi \d x \d t = \int_{\Omega_T} \overbrace{\Big( \frac{u}{u+\varepsilon} \Big)}^{\leq 1} \underbrace{\frac{(u+\varepsilon) h''(u+\varepsilon)}{h'(u+\varepsilon)}}_{\in L^\infty } \nabla v_\varepsilon \cdot \p \varphi \d x \d t, 
\end{equation*}
and we infer from the Dominated Convergence Theorem, letting $\varepsilon \to 0$, that there holds 
\begin{equation*}
   \begin{aligned} -\int_{\Omega_T} \!\!\! v \partial_t \varphi \d x \d t  + \int_{\Omega_T}\!\!\! \p u h'(u)\! \cdot \! \nabla \varphi  \d x \d t + \int_{\Omega_T} \!\! \frac{u h''(u)}{h'(u)} \nabla v \!\cdot\! \p \varphi \d x \d t  \leq - \int_{\Omega_T} \!\!\! \nabla v\!\cdot\!\nabla \varphi \d x \d t 
   \end{aligned}
\end{equation*}
for all nonnegative $\varphi \in C^\infty_c(0,T;C^\infty_\per(\Omega))$. The result now follows from the density of $C^\infty_c(0,T;C^\infty_\per(\Omega))$ in $L^2(0,T;H^1_\per(\Omega))$, using the boundedness of $v$ and bound \eqref{eq:third condition on initial data}. 
\end{proof}

The result of Lemma \ref{prop: w big prop} showed that $v$ is a weak subsolution of \eqref{eq: eqn for w in lem}. Note that, aside from the term $\dv(u h'(u)\p)$, all of the coefficients of \eqref{eq: eqn for w in lem} belong to $L^\infty(\Omega_T)$, and this former term is controlled using the bound $u|h'(u)| \leq v + M$ from \eqref{eq:third condition on initial data}. We now apply De Giorgi type iterations (\textit{cf.}~\textit{e.g.}~\cite[\S 3.1]{reg1}) to show that $v$ is bounded for positive times. 

\begin{proof}[Proof of Proposition \ref{prop:lower bound 1-rho}]
Equation \eqref{eq: eqn for w in lem} may be interpreted as a linear equation in $v$ with given bounded coefficients $u,\p$. To this end, we let 
    \begin{equation}\label{eq:w def at start of moser proof}
        w(t,\bxi) := \frac{v(t,\bxi)}{L}, 
    \end{equation}
    where $L>0$ is to be determined. Note that $w$ is then a weak subsolution of 
    \begin{equation}\label{eq: eqn for w in lem L rescaling}
    \partial_t w - \frac{1}{L} \dv\big( \p u h'(u) \big) + \frac{u h''(u)}{h'(u)}\p\!\cdot\!\nabla w = \Delta w. 
    \end{equation}
Until the final step of the proof, we work with $w$ instead of $v$. 

    \smallskip 

  \noindent  1. \textit{Caccioppoli inequality}: Let $k > 0$ and define the Stampacchia truncation $\mathscr{T}_k w := (w-k)_+$. Note that $\mathscr{T}_k w \in L^\infty(0,T;L^2_\per(\Omega)) \cap L^2(0,T;H^1_\per(\Omega))$ is nonnegative, whence it is an admissible test function to insert into the weak formulation of the equation for $w$, namely, \eqref{eq: eqn for w in lem L rescaling}. We get, from the definition of weak subsolution, 
    \begin{equation*}
        \begin{aligned}
            \frac{1}{2}\frac{\der}{\der t}\!\!\int_{\Omega} |\mathscr{T}_k w|^2 \d x\! +\!\!\int_{\Omega} |\nabla \mathscr{T}_k w|^2 \d x \leq \!-\!\!\int_{\Omega} \frac{u h'(u)}{L} \p \cdot \nabla \mathscr{T}_k w \d x \! -\!\! \int_\Omega \frac{u h''(u)}{h'(u)}\p\!\cdot\! \mathscr{T}_k w\nabla \mathscr{T}_k w \d x.  
        \end{aligned}
    \end{equation*}
Next, we use the bounds \eqref{eq:conditions on h} and \eqref{eq:third condition on initial data} on $h$ to control the right-hand side. By letting $M' := 1+\max\{M, \Vert s h''/h'
 \Vert_{L^\infty(\mathbb{R}_+)}\}$, with $M$ as given in \eqref{eq:third condition on initial data}, we obtain 
      \begin{equation*}
        \begin{aligned}
            \frac{1}{2}\frac{\der}{\der t}\!\int_{\Omega} |\mathscr{T}_k w|^2 \d x\! +\!\!\int_{\Omega} & |\nabla \mathscr{T}_k w|^2 \d x \leq ( k + M' L^{-1} )\!\int_{\Omega} |\nabla \mathscr{T}_k w| \d x \! + M'\!\int_\Omega \mathscr{T}_k w |\nabla \mathscr{T}_k w| \d x. 
        \end{aligned}
    \end{equation*}
Without loss of generality, we assume $L \geq M'$, whence, by applying Young's inequality, 
      \begin{equation*}
        \begin{aligned}
            \frac{\der}{\der t}\!\int_{\Omega} |\mathscr{T}_k w|^2 \d x + \int_{\Omega} & |\nabla \mathscr{T}_k w|^2 \d x \leq M''( 1+k^2) \int_\Omega ( |\mathscr{T}_k w |^2 + \mathds{1}_{\{\mathscr{T}_k w > 0\}}) \d x, 
        \end{aligned}
    \end{equation*}
where $M'' = 8(M')^2$ is independent of $k,L,w$. Integrating in time, a.e.~$0<s<t<T$, 
\begin{equation}\label{eq:caccioppoli for w}
            \begin{aligned}
            \int_{\Omega} |\mathscr{T}_k w|^2 \d x\bigg|_{t} - \int_{\Omega} |\mathscr{T}_k w|^2 & \d x\bigg|_{s} + \int_s^t \int_{\Omega} |\nabla \mathscr{T}_k w |^2 \d x \d \tau \\ 
            &\leq M''(1+k^2) \int_s^t \int_{\Omega}(|\mathscr{T}_k w|^2 + \mathds{1}_{\{\mathscr{T}_k w>0\}}) \d x \d \tau. 
        \end{aligned}
\end{equation}

    \smallskip 

    \noindent 2. \textit{Setup for iterations}: Fix $t_0 \in (0,T)$ an arbitrary Lebesgue point of $\rho$. Define the times \begin{equation}\label{eq:Tn def}
    T_n := t_0(1- 2^{-n-1}) \qquad \forall n \in \mathbb{N}\cup\{0\}.
    \end{equation}
    $\{T_k\}$ is monotonically increasing: $t_0/2 = T_0 < T_1 < \dots < T_n < t_0 = \lim_{n \to \infty} T_n$. Define 
    \begin{equation}\label{eq:kappa n def} 
    \kappa_n := (1-2^{-n})  \qquad \forall n \in \mathbb{N}\cup\{0\}
    \end{equation}
    shorthand $w_n := \mathscr{T}_{\kappa_n} w$, and $\mathscr{W}_n := \esssup_{t \in [T_n,T]} \int_\Omega |w_n|^2 \d x + \int_{T_n}^T \int_\Upsilon |\nabla w_n|^2 \d x \d t,$ noting 
\begin{equation}\label{eq:W0 def}
\mathscr{W}_0 = \frac{1}{L^{2}}(\Vert v \Vert_{L^\infty(t_0/2,T;L^2(\Omega))}^2 + \Vert \nabla v \Vert_{L^2(\Omega_{t_0/2,T})})^2. 
\end{equation}
Our objective is to prove, for some positive $C_*$ independent of $n$, the recursion estimate 
\begin{equation}\label{eq:recursion estimate moser for w}
    \mathscr{W}_{n} \leq C_*^n \mathscr{W}_{n-1}^{\frac{3}{2}} \qquad \forall n \in \mathbb{N} \cup \{0\}, 
\end{equation}

Substitute $w_{n}$ into the Caccioppoli inequality \eqref{eq:caccioppoli for w}. For $T_{n-1} \leq s \leq T_{n} \leq t$, we get 
    \begin{equation*}
            \begin{aligned}
            \int_{\Omega} \!|w_{n}|^2 \d x \bigg|_{t} \!\!+\!\! \int_{T_{n}}^t \! \int_{\Omega} \!\!|\nabla w_{n} |^2 \d w \d \tau \leq \!\! \int_{\Omega}\! |w_{n}|^2 \d x \bigg|_{s}  \!\!\!+\! M''(1\!+\!\kappa_{n}^2)\!\! \int_{T_{n-1}}^T \!\!\int_{\Omega} \!\!(|w_{n}|^2 \!+\! \mathds{1}_{\{w_{n} >0\}}) \d x \d \tau. 
        \end{aligned}
\end{equation*}
    Integrating the above with respect to the coordinate $s$ over the interval $[T_n,T_{n+1}]$, we get 
        \begin{equation*}
            \begin{aligned}
         \int_{\Omega}\!\! |w_n|^2 \! \d x \bigg|_{t} \!\!+\!\! \int_{T_{n}}^t\!\! \int_{\Omega}\!\! |\nabla w_n |^2 \d x \d \tau \!   \leq \! \frac{2^{n}}{t_0}\!\! \int_{T_{n-1}}^{T_{n}}\!\!\int_{\Omega}\!\! |w_n|^2 \d x \d \tau \!+\! 2M'' \!\!\!\int_{T_{n-1}}^T\!\! \int_{\Omega}\!\! (|w_n|^2 \!+\! \mathds{1}_{\{w_n >0\}})\! \d x\! \d \tau, 
        \end{aligned}
\end{equation*}
whence, by defining $C_0 := 2t_0^{-1} + 2M''$, and absorbing the first term on the right-hand side of the previous inequality into the second term on the right-hand side, we have the bound 
        \begin{equation*}
            \begin{aligned}
         \int_{\Omega} |w_n|^2 \d & x \bigg|_{t} + \int_{T_{n}}^t \int_{\Omega} |\nabla w_n |^2 \d x \d \tau \leq C_0^n \int_{T_{n-1}}^T \int_{\Omega} (|w_n|^2 + \mathds{1}_{\{w_n >0\}}) \d x \d \tau. 
        \end{aligned}
\end{equation*}
Observe also that, for any point $(t,x) \in \{g_n > 0\}$, there holds 
\begin{equation}\label{eq:to get from fn-1 to fn for w}
    w_{n-1}(t,\bxi) = w(t,\bxi) - \kappa_{n-1} = w_n(t,\bxi) + 2^{-n} > 2^{-n}. 
\end{equation}
By squaring the above, we get $\mathds{1}_{\{w_n > 0\}}\! \leq\! 2^{2n} |w_{n-1}|^2 \mathds{1}_{\{w_n>0\}}$. Defining $C_1 := 8C_0$, we get 
    \begin{equation}\label{eq:some of the way there moser for w}
            \begin{aligned}
         \mathscr{W}_n \leq C_1^n \int_{T_{n-1}}^T \int_{\Omega} \mathds{1}_{\{w_n > 0\}} |w_{n-1}|^2 \d x \d \tau. 
        \end{aligned}
\end{equation}

   \smallskip 

    \noindent 3. \textit{De Giorgi type iterations}: The Sobolev inequality yields that, for some positive constant $C_q$ independent of $t,n$, there holds $\Vert w_n(t,\cdot) \Vert_{L^q(\Upsilon)}^2 \leq C_q^2 \Vert w_n(t,\cdot) \Vert_{H^1(\Omega)}^2$ for all $q \in [1,\infty)$, and hence $ \Vert w_n \Vert_{L^2(T_n,T;L^q(\Omega))}^2 \leq C_q \mathscr{W}_n$, whence by using the H\"older inequality 
    \begin{equation}\label{eq:interp 1 for w}
        \Vert w_n \Vert^2_{L^2(\Upsilon_{T_n,T})} \leq |\Omega|^{1-\frac{2}{q}}\Vert w_n \Vert^2_{L^2(T_n,T;L^q(\Omega))} \leq  C_q  |\Omega|^{1-\frac{2}{q}} \mathscr{W}_n. 
    \end{equation}
Meanwhile, the Interpolation Lemma \ref{lem:dibenedetto classic} implies 
\begin{equation}\label{eq:interp 2 for w}
    \Vert w_n^2 \Vert_{L^{2}(\Omega_{T_n,T})} \leq \Vert w_n \Vert^2_{L^{4}(\Omega_{T_n,T})} \leq C_I^2(1+t_0)^2 \mathscr{W}_n. 
\end{equation}
By returning to \eqref{eq:some of the way there moser for w} and using \eqref{eq:interp 2 for w}, as well as the H\"older, Minkowski, and Jensen inequalities, we get $\mathscr{W}_n \leq 4 C_1^n |\{w_n>0\} \cap [T_{n-1},T]|^\frac{1}{2} \Vert w_{n-1}^2 \Vert_{L^{2}(\Upsilon_{T_{n-1},T})}$. In turn, defining $C_2 := 4C_1 (1+C_I)^2(1+t_0)^2$, we have 
\begin{equation}\label{eq:moser almost there for w}
    \mathscr{W}_n \leq C_2^n |\{w_n >0 \} \cap [T_{n-1},T]|^{\frac{1}{2}} \mathscr{W}_{n-1}. 
\end{equation}
Eq.~\eqref{eq:to get from fn-1 to fn for w} implies $\{w_n > 0\} \subset \{ w_{n-1} > 2^{-n}\}$. Markov's inequality and \eqref{eq:interp 1 for w} yield 
\begin{equation*}
   \begin{aligned}
       |\{w_n \!>\!0 \} \!\cap\! [T_{n-1},T]| \leq  |\{w_{n-1} \mathds{1}_{[T_{n-1},T]} \!>\! 2^{-n} \}| \!\leq\! \frac{\Vert w_{n-1}\Vert_{L^2(\Upsilon_{T_{n-1},T})}^2}{ 2^{-2n}} \!\leq\! C_q  |\Omega|^{1-\frac{2}{q}}  2^{2n}\mathscr{W}_{n-1}, 
   \end{aligned}
\end{equation*}
Returning to \eqref{eq:moser almost there for w}, letting $C_* := (1+C_q  |\Omega|^{1-\frac{2}{q}}) C_2$ where $q \in (2,\infty)$ is fixed appropriately (\textit{e.g.} $q=4$ is sufficient), we get $\mathscr{W}_n \leq C_*^n \mathscr{W}_{n-1}^{\frac{3}{2}}$, \textit{i.e.}, \eqref{eq:recursion estimate moser for w}. 

\smallskip 

\noindent 4. \textit{$L^\infty$ estimate}: Estimate \eqref{eq:recursion estimate moser for w} and the iteration lemma \cite[\S 1, Lemma 4.1]{DiBenedetto} imply: if 
\begin{equation}\label{eq:initialisation moser for w}
    \mathscr{W}_0 \leq C_*^{-4}, 
\end{equation}
then there holds $\lim_{n\to\infty}\mathscr{W}_n = 0$. In turn, using \eqref{eq:W0 def}, by choosing, \textit{e.g.}, 
\begin{equation}\label{eq:L choice at end for w}
    L_{t_0} = M' + 2 C_*^{4}(\Vert v \Vert_{L^\infty(t_0,T;L^2(\Omega))}^2 + \Vert \nabla v \Vert_{L^2(\Omega_{t_0,T})})^2, 
\end{equation}
we see that \eqref{eq:initialisation moser for w} is satisfied. Using the Monotone Convergence Theorem and \eqref{eq:interp 1 for w}, 
\begin{equation*}
    \int_{t_0}^T \int_\Upsilon (w-1)_+ \d \bxi \d t = \lim_{n\to\infty} \Vert w_n \Vert^2_{L^2(\Upsilon_{T_n,T})} \leq C_q |\Omega|^{1-\frac{2}{q}} \lim_{n\to\infty}\mathscr{W}_n = 0, 
\end{equation*}
whence $0 \leq w \leq 1$ a.e.~in $\Omega_{t_0,T}$. In other words, using \eqref{eq:w def at start of moser proof}, we have 
\begin{equation*}
    0 \leq v \leq L_{t_0} \qquad \text{a.e.~} \Upsilon_{t_0,T}, 
\end{equation*}
where $L_{t_0} = L_{t_0}(t_0,M, \Vert s h''/h'\Vert_{L^\infty(\mathbb{R}_+)} , T , \Omega)$, as prescribed by \eqref{eq:L choice at end for w}. 

    \smallskip 

    \noindent 5. \textit{Conclusion}: It follows from Step 4 and the genericity of the Lebesgue point $t_0$ that 
\begin{equation}\label{eq: Linfty bound w}
   \esssup_{\Omega_{t,T}} h(u) = \Vert v \Vert_{L^\infty(\Omega_{t,T})} \leq L_t \quad \text{for a.e.~} t \in (0,T),
\end{equation}
where, as per \cite[\S 3]{reg1}, the periodicity of $v$ means that we need not localise the estimate in space, but only in time. Since $h$ is strictly decreasing on $(0,1]$, its inverse $h^{-1}$ is well-defined and also strictly decreasing on $[1,\infty)$, whence 
\begin{equation*}
    u(s,x) = h^{-1}(v(s,x)) \geq h^{-1}(L_t) > 0 \qquad \text{a.e.~}(s,x) \in \Omega_{t,T}, 
\end{equation*}
the strict positivity follows from the assumption $\lim_{s \to 0^+} h(s) = \infty$. Letting $c(t) := h^{-1}(L_t)$ and taking the $\essinf$, we obtain the result. 
\end{proof}

\subsection{Strong Parabolicity up to $t=0$ for stronger initial data}\label{subsec:strong parab up to t zero}

Finally, we record the stronger result (uniform in time up to $t=0$) for the stronger initial condition $\esssup_\Omega \rho_0 < 1$. This assumption and result (Lemma \ref{lem:unif lower bound for stronger initial data}) are not used in the remainder of the regularity part of the paper (\S \ref{sec:boundedness positive times}--\S \ref{sec:higher reg}, which only rely on Proposition \ref{prop:lower bound 1-rho}), but will be used to prove the uniqueness result of Theorem \ref{thm:uniqueness} (\textit{cf.}~\S \ref{sec:uniqueness}). 

\begin{lemma}[Uniform lower bound on $1-\rho$ up to $t=0$ for stronger initial data]\label{lem:unif lower bound for stronger initial data}
    Let $f$ be a weak solution of \eqref{eq:main eqn} with regular initial data satisfying Definition \ref{def:reg initial data} and the additional assumption $\esssup_\Omega \rho_0 < 1$, with $\rho$ as per \eqref{eq:rho def}. Then, there holds $\esssup_{\Omega_T} \rho < 1$. 
    \end{lemma}
\begin{proof}
The assumption on the initial data implies that $\Vert v_0 \Vert_{L^\infty(\Omega)} < \infty$. We have already shown that $v$ is a weak subsolution of the equation \eqref{eq: eqn for w in lem}. Our strategy is to test this equation with $v^{n-1}$ for $n \geq 2$; employing this quantity as a test function is justified by the strategy of proof of Lemma \ref{prop: w big prop}, \textit{i.e.}, considering the test function $v_\varepsilon^{n-1}$, where $v_\varepsilon=h(u+\varepsilon)$. 

    By testing the equation \eqref{eq:eqn for v eps} for $v_\varepsilon$ with $v_\varepsilon^{n-1}$ ($n \geq 2$), we obtain 
    \begin{equation}\label{eq:unif Linfty v eps computation}
       \begin{aligned} \frac{1}{2}\frac{\der}{\der t}\int_\Omega |v_\varepsilon^{\frac{n}{2}}|^2 \d x + 2(1 - \frac{1}{n}) \int_\Omega |\nabla (v_\varepsilon^{\frac{n}{2}}) |^2 \d x = &-(n\!-\!1) \! \int_\Omega \!\! u h'(u+\varepsilon) v_\varepsilon^{\frac{n}{2}-1} \nabla(v_\varepsilon^{\frac{n}{2}}) \! \cdot \! \p \d x \\ 
       &+ \int_\Omega \frac{u h''(u+\varepsilon)}{h'(u+\varepsilon)} v_\varepsilon^{\frac{n}{2}}\nabla(v_\varepsilon^{\frac{n}{2}}) \!\cdot\! \p \d x, 
       \end{aligned}
    \end{equation}
    and we bound each of the integrals on the right-hand side as follows: for the first integral, 
    \begin{equation*}
        \begin{aligned}
            \bigg|\int_\Omega u h'(u+\varepsilon) v_\varepsilon^{\frac{n}{2}-1} \nabla(v_\varepsilon^{\frac{n}{2}}) \! \cdot \! \p \d x \bigg| &\leq \int_\Omega \frac{u}{u+\varepsilon} |(u+\varepsilon) h'(u+\varepsilon) |v_\varepsilon^{\frac{n}{2}-1} |\nabla(v_\varepsilon^{\frac{n}{2}})| \d x\\ 
            &\leq \int_\Omega (v_\varepsilon+M) v_\varepsilon^{\frac{n}{2}-1} |\nabla(v_\varepsilon^{\frac{n}{2}})| \d x, 
        \end{aligned}
    \end{equation*}
    where we used the bound \eqref{eq:third condition on initial data}, and for the second integral, using $\Vert s h''/h' \Vert_{L^\infty(\mathbb{R}_+)} \leq C$, 
        \begin{equation*}
       \begin{aligned} 
      \bigg| \int_\Omega \!\!\frac{u h''(u\!+\!\varepsilon)}{h'(u\!+\!\varepsilon)} v_\varepsilon^{\frac{n}{2}}\nabla(v_\varepsilon^{\frac{n}{2}}) \!\cdot\! \p \d x \bigg| \!\leq\! \int_\Omega\!\! \frac{u}{u\!+\!\varepsilon} \Big|\frac{(u\!+\!\varepsilon) h''(u\!+\!\varepsilon)}{h'(u\!+\!\varepsilon)}\Big| v_\varepsilon^{\frac{n}{2}}|\nabla(v_\varepsilon^{\frac{n}{2}}) | \d x \!\leq \! C\!\! \int_\Omega\!\! v_\varepsilon^{\frac{n}{2}}|\nabla(v_\varepsilon^{\frac{n}{2}}) | \d x. 
       \end{aligned}
    \end{equation*}
    Moreover, since $h' < 0$ and $2 > \sup_{\Omega_T} (1-\rho+\varepsilon)$ for all $\varepsilon \in (0,1)$, there holds 
    \begin{equation*}
      h(2) \leq h\big(\sup_{\Omega_T} (1-\rho+\varepsilon)\big) = \inf_{\Omega_T}h(1-\rho+\varepsilon) = \inf_{\Omega_T} v_\varepsilon \implies v_\varepsilon^{-1} \leq \frac{1}{h(2)}, 
    \end{equation*}
    from which we deduce, by returning to \eqref{eq:unif Linfty v eps computation}, 
     \begin{equation*}
       \begin{aligned} \frac{\der}{\der t}\! \int_\Omega \! |v_\varepsilon^{\frac{n}{2}}|^2 \d x \!+\! \int_\Omega \!|\nabla (v_\varepsilon^{\frac{n}{2}}) |^2 \d x \!\leq\! 2\Big( (n\!-\!1)(1\!+\!Mh(2)^{-1}) \!+\! \Vert \frac{s h''}{h'} \Vert_{L^\infty(\mathbb{R}_+)} \Big) \!\int_\Omega \!v_\varepsilon^{\frac{n}{2}} |\nabla(v_\varepsilon^{\frac{n}{2}})| \d x. 
       \end{aligned}
    \end{equation*}
    In turn, by using Young's inequality and letting $\varepsilon \to 0$, there holds 
  \begin{equation*}
        \frac{\der}{\der t}\Vert v^{\frac{n}{2}}(t,\cdot)\Vert_{L^2(\Omega)}^2 + \Vert \nabla v^{\frac{n}{2}} (t,\cdot) \Vert_{L^2(\Omega)}^2 \leq C_0 n^2 \Vert v^{\frac{n}{2}}(t,\cdot) \Vert_{L^2(\Omega)}^2, 
    \end{equation*}
    for some positive constant $C_0$ independent of $n$. We employ a similar strategy to the one used in \cite[Proof of Theorem 1.7]{reg1}. By the Gagliardo--Nirenberg inequality, a.e.~$t \in (0,T)$, 
    \begin{equation*}
      \Vert v^{\frac{n}{2}}(t,\cdot) \Vert_{L^2(\Omega)} \leq C_{GN}\Big(  \Vert \nabla v^{\frac{n}{2}}(t,\cdot) \Vert_{L^2(\Omega)}^{\frac{1}{2}} \Vert v^{\frac{n}{2}}(t,\cdot) \Vert_{L^1(\Omega)}^{\frac{1}{2}} +  \Vert v^{\frac{n}{2}}(t,\cdot) \Vert_{L^1(\Omega)} \Big), 
    \end{equation*}
whence, using Young's inequality, 
\begin{equation*}
      \Vert v^{\frac{n}{2}}(t,\cdot) \Vert_{L^2(\Omega)}^2 \leq \frac{1}{C_0 n^2} \Vert \nabla v^{\frac{n}{2}}(t,\cdot) \Vert_{L^2(\Omega)}^{2}  + 2C_{GN}^2\Big(1+\frac{C^2_{GN} C_0 n^2}{2}\Big) \Vert v^{\frac{n}{2}}(t,\cdot) \Vert_{L^1(\Omega)}^2, 
    \end{equation*}
    as well as $\Vert \nabla v^{\frac{n}{2}}(t,\cdot) \Vert_{L^2(\Omega)}^{2} \geq C ( \Vert v^{\frac{n}{2}}(t,\cdot) \Vert_{L^2(\Omega)}^2 - \Vert v^{\frac{n}{2}}(t,\cdot) \Vert_{L^1(\Omega)}^2  )$. Hence, 
  \begin{equation*}
        \frac{\der}{\der t}\Vert v^{n}(t,\cdot)\Vert_{L^1(\Omega)} + \Vert v^{n} (t,\cdot) \Vert_{L^1(\Omega)} \leq C_1 n^4 \Vert v^{\frac{n}{2}}(t,\cdot) \Vert_{L^1(\Omega)}^2, 
    \end{equation*}
where the positive constant $C_1$ is independent of $n$. By setting $n_k := 2^k$ and $A_k(t) := \Vert v^{n_k}(t,\cdot) \Vert_{L^1(\Omega)}, $ we obtain, with $C_2 := 16 C_1$, the recurrence inequality 
    \begin{equation*}
        \frac{\der A_k}{\der t} + A_k \leq C_2^k (A_{k-1})^2. 
    \end{equation*}
An application of the iterative lemma \cite[Lemma 3.5]{reg1} (\textit{cf.}~\cite[Lemma 3.2]{KZ}), as per the proof of \cite[Theorem 1.3]{reg1}, then immediately implies the global-in-time estimate 
    \begin{equation*}
        \Vert v(t,\cdot) \Vert_{L^{2^k}(\Omega)} \leq C(C_2,\Vert v_0 \Vert_{L^\infty}) \qquad \text{for all } t, k, 
    \end{equation*}
    and, by taking the $\esssup$ in time and letting $k\to\infty$, the proof is complete. 
\end{proof}

\section{Improved Boundedness for Positive Times}\label{sec:boundedness positive times}

\subsection{Derivative estimates}\label{subsec:second deriv}

The main result of this subsection is as follows. It shall subsequently be used to perform $H^1$ and $H^2$ estimates, and De Giorgi type iterations, on \eqref{eq:main eqn}. While these $H^2$ estimates are relatively strong, they do not immediately yield boundedness in $L^\infty$ by Morrey's embedding because of the time-dependence; this is why we require a further De Giorgi iteration (\textit{cf.}~Proposition \ref{lem:Linfty for f}). 

\begin{prop}\label{lem:nabla rho in L8}
   Let $f$ be a weak solution of \eqref{eq:main eqn} with initial data satisfying the assumptions of Theorem \ref{thm:smooth}, and $\rho$ as per \eqref{eq:rho def}. Then, for a.e.~$t \in (0,T)$, there holds 
    \begin{equation*}
\partial_t \rho , \nabla^2 \rho \in L^4(\Omega_{t,T}), \qquad \nabla \rho \in L^8(\Omega_{t,T}), \qquad \rho \in C^{0,\frac{1}{4}}(\Omega_{t,T}). 
    \end{equation*}
\end{prop}

The rest of this subsection is spent proving this proposition. For completeness, we explicitly state the estimate satisfied by the relevant derivatives of $\rho$. It is obtained from \eqref{eq:rho eq} precisely as per the proof of Theorem 3.1 in \cite{RosHeat}. Indeed we have, for some $\epsilon>0$,
\begin{equation*}
\int_{\Omega_{t(1+\epsilon),T}} \Big( |\de_t \rho|^q+|\nabla^2 \rho|^q\Big) \d x\d s
\leq
C_{q,\epsilon} \bigg(
\int_{\Omega_{t,T}}  |\rho|^q \d x\d s
+
\int_{\Omega_{t,T}}  |\dv((1-\rho) \p)|^q \d x\d s
\bigg). 
\end{equation*}

Using the lower bound on $1\!-\!\rho$ from Propositions \ref{prop:lower bound 1-rho} and \ref{lem:unif lower bound for stronger initial data}, we obtain the following consequence of Lemma \ref{lem:H1 for tensors}. 
\begin{lemma}\label{cor:H1 p P tensors away from initial time}
Let $f$ be a weak solution of \eqref{eq:main eqn} with initial data satisfying the assumptions of Theorem \ref{thm:smooth}. Then, for all $n\in\mathbb{N}$, with $\pi^n$ as per \eqref{eq:tensor def}, and for a.e.~$t \in (0,T)$, there holds 
    \begin{equation*}
        \nabla \pi^n \in L^2(\Omega_{t,T}). 
    \end{equation*}
    In particular, with $\p$ and $\P$ as per \eqref{eq:polarisation def} and \eqref{eq:P matrix def}, we have $\nabla \p , \nabla \P \in L^2(\Omega_{t,T})$ for a.e.~$t\in(0,T)$. Furthermore, if the initial data is assumed to satisfy $\esssup_\Omega \rho_0 < 1$, then the inclusion holds up to $t=0$. 
\end{lemma}

The following lemma is a $H^2$-type result for $\rho$. 

\begin{lemma}\label{cor:nabla rho in L4}
   Let $f$ be a weak solution of \eqref{eq:main eqn} with initial data satisfying the assumptions of Theorem \ref{thm:smooth}, and $\rho$ as per \eqref{eq:rho def}. Then, for a.e.~$t \in (0,T)$, there holds 
    \begin{equation*}
      \partial_t \rho , \nabla^2 \rho \in L^2(\Omega_{t,T}), \quad  \nabla \rho \in L^4(\Omega_{t,T}). 
    \end{equation*}
    Furthermore, there holds 
  \begin{equation}\label{eq:H2 est for rho in statement}
        \Vert \nabla \rho \Vert_{L^\infty(t,T;L^2(\Omega))}^2 \!+\! \Vert \Delta \rho \Vert^2_{L^2(\Omega_{t,T})} \!\leq\! \Vert \nabla \rho(t,\cdot) \Vert_{L^2(\Omega)}^2 \!+\! 2\Big(\Vert \nabla \rho \Vert^2_{L^2(\Omega_{T})} \!+\! \Vert \sqrt{1\!-\!\rho}\nabla \p \Vert^2_{L^2(\Omega_{T})} \Big). 
    \end{equation}
    If the initial data is assumed to satisfy $\rho_0 \in H^1_\per(\Omega)$, then the inclusions hold up to $t=0$. 
\end{lemma}

The proof relies on the Calder\'on--Zygmund estimate for the heat equation and the interpolation  of Lemma \ref{lem:dibenedetto classic}.

\begin{proof}
    Using Lemma \ref{lem:H1 p} (or Lemma \ref{cor:H1 p P tensors away from initial time}), equation \eqref{eq:rho eq} reads $\partial_t \rho - \Delta \rho = -\dv((1-\rho)\p) \in L^2(\Omega_{t,T})$, whence an application of the Calder\'on--Zygmund estimate for the heat equation (\textit{e.g.}~\cite[Theorem 3.1]{RosHeat}) yields $\partial_t \rho , \nabla^2 \rho \in L^2(\Omega_{t,T})$. In turn, we may rigorously perform the standard $H^2$-type estimate on the equation \eqref{eq:rho eq} by testing against $\Delta \rho$. We get 
    \begin{equation*}
        \frac{1}{2}\frac{\der}{\der t}\int_\Omega |\nabla \rho|^2 \d x + \int_\Omega |\Delta \rho|^2 \d x \leq \int_\Omega |\dv((1-\rho)\p)|^2 \d x, 
    \end{equation*}
    and hence, by integrating in time and using Young's inequality, we deduce \eqref{eq:H2 est for rho in statement}.
    
   Lemma \ref{lem:H1 p} and Definition \ref{def:function space} imply $\nabla \rho \in L^\infty(t,T;L^2(\Omega)) \cap L^2(t,T;H^1(\Omega))$ a.e.~$t \in (0,T)$, whence Lemma \ref{lem:dibenedetto classic} implies $\nabla \rho \in L^4(\Omega_{t,T})$. If $\rho_0 \in H^1_\per(\Omega)$, then we may take the limit as $t \to 0$ in the previous inequality and inclusions hold up to $t=0$. 
\end{proof}

Our next focus is the $H^2$-type estimate for $\p$. This result will yield higher integrability of $\nabla \p$, which will subsequently be used to improve the integrability of $\partial_t \rho, \nabla^2 \rho$ via \eqref{eq:rho eq}. 

\begin{lemma}[$H^2$-type estimate for $\p$]\label{lem:H2 for p}
Let $f$ be a weak solution of \eqref{eq:main eqn} with initial data satisfying the assumptions of Theorem \ref{thm:smooth}, and $\p$ as per \eqref{eq:polarisation def}. Then, for a.e.~$t \in (0,T)$, there holds 
    \begin{equation*}
      \nabla \p \in L^\infty(t,T;L^2(\Omega)), \qquad  \nabla^2 \p \in L^2(\Omega_{t,T}), \qquad \nabla \p \in L^4(\Omega_{t,T}). 
    \end{equation*}
    Furthermore, there exists $C_t$ a positive constant depending on $t$ such that 
    \begin{equation}\label{eq:est H2 p formal at the end}
    \begin{aligned}
         \Vert \nabla \p \Vert^2_{L^\infty(t,T;L^2(\Omega))} + \Vert \Delta \p \Vert^2_{L^2(\Omega_{t,T})} \leq C_t\Big(\Vert \nabla \p(t,\cdot) \Vert^2_{L^2(\Omega)} + \Vert \Delta \rho\Vert_{L^2(\Omega_{t,T})}^2 + 4|\Omega_T| \Big). 
    \end{aligned}
    \end{equation}
\end{lemma}
The underlying idea of the proof is to test the equation \eqref{eq:polarisation eq} with $\Delta \p$ and to exploit a cancellation that occurs when writing the equation in non-divergence form. Indeed, by formally expanding the diffusions in the non-divergence form, \textit{i.e.}, 
    \begin{equation}\label{eq:non div form p eqn}
        \partial_t \p + \dv((1-\rho)\P) = (1-\rho)\Delta \p + \p \Delta \rho - \p, 
    \end{equation}
 we see that the terms of the form $\nabla \p \otimes \nabla \rho$ have cancelled. By testing with $\Delta \p$, we get 
    \begin{equation}\label{eq:formal estimate H2 p}
    \begin{aligned}
         \frac{1}{2}\frac{\der}{\der t}\!\int_\Omega\! |\nabla \p |^2 \d x \!+\!\! \int_\Omega \! (1\!-\!\rho) & |\Delta \p |^2 \d x \!=\!\! \!\int_\Omega \! \p\! \cdot\! \Delta \p \Delta \rho \d x \!- \!\!\!\int_\Omega \!(1\!-\!\rho) \nabla \p \!:\! \P \d x \!- \!\!\!\int_\Omega \!\p \!\cdot \!\Delta \p \d x, 
    \end{aligned}
    \end{equation}
    whence, using also Young's inequality, for a.e.~$t \in (0,T)$, we recover \eqref{eq:est H2 p formal at the end}, where $C_t$ is a positive constant obtained from the lower bound on $1-\rho$ (\textit{cf.}~Proposition \ref{prop:lower bound 1-rho}), and the second term on the right-hand side is bounded by virtue of Lemma \ref{cor:nabla rho in L4}. The periodic Calder\'on--Zygmund Lemma \ref{lem:CZ periodic} then yields the required boundedness of the second derivatives. However, this formal argument is not rigorous since \emph{a priori} we do not know that $\Delta \p$ belongs to $L^2(\Omega_{t,T})$; indeed, we only know $\Delta \p \in L^2(t,T;(H^1_\per(\Omega))')$ a.e.~$t \in (0,T)$ from Lemma \ref{cor:H1 p P tensors away from initial time}. We justify the formal estimate \eqref{eq:formal estimate H2 p} rigorously in \S \ref{sec:proof of H2 via galerkin}. 

      For clarity, before giving the proof of Lemma \ref{lem:H2 for p} (\textit{cf.}~\S \ref{sec:proof of H2 via galerkin}), we prove Proposition \ref{lem:nabla rho in L8}. 

\begin{proof}[Proof of Proposition \ref{lem:nabla rho in L8}]
    It follows from Lemma \ref{lem:H2 for p} and Lemma \ref{cor:nabla rho in L4} that equation \eqref{eq:rho eq} reads 
   $\partial_t \rho - \Delta \rho = -\dv((1-\rho)\p) \in L^4(\Omega_{t,T})$ {a.e.~}$t\in(0,T)$. The Calder\'on--Zygmund estimate for the heat equation (\textit{cf.}~\textit{e.g.}~\cite[Theorem 3.1]{RosHeat} or Lemma \ref{lem:Lp schauder}) thereby yields $\partial_t \rho , \nabla^2 \rho \in L^4(\Omega_{t,T})$ a.e.~$t\in(0,T)$. 

   In turn, the Gagliardo--Nirenberg inequality yields 
\begin{equation}\label{eq:GN grad rho L8}
    \Vert \nabla\rho(t,\cdot) \Vert_{L^8(\Omega)} \leq C \Vert   \nabla^2\rho(t,\cdot) \Vert_{L^4(\Omega)}^{\frac{1}{2}} \Vert \rho(t,\cdot) \Vert_{L^\infty(\Omega)}^{\frac{1}{2}} + C \Vert \rho(t,\cdot) \Vert_{L^1(\Omega)}, 
\end{equation}
whence, by raising to the power $8$ and integrating in time, we get 
\begin{equation*}
    \Vert \nabla\rho \Vert_{L^8(\Omega_{t,T})}^8 \leq C \Vert   \nabla^2\rho \Vert_{L^4(\Omega_{t,T})}^{4} \Vert \rho \Vert_{L^\infty(\Omega_{t,T})}^{4} + C. 
\end{equation*}
   Furthermore, we have that $\rho \in W^{1,4}(\Omega_{t,T})$, whence an application of Morrey's embedding implies the H\"older continuity $\rho \in C^{0,\frac{1}{4}}(\Omega_{t,T})$. The proof is complete. 
\end{proof}

\subsection{Proof of Lemma \ref{lem:H2 for p}}\label{sec:proof of H2 via galerkin}

Our goal is to implement a Galerkin approximation in order to justify the formal estimate \eqref{eq:formal estimate H2 p}. Our reason for doing so is that attempts to obtain \eqref{eq:formal estimate H2 p} by difference quotients or by mollification destroys the cancellation of the terms $\nabla \p \nabla \rho$ when expanding the diffusion $\dv((1-\rho)\nabla \p + \p \otimes \nabla \rho) = (1-\rho)\Delta \p + \p \Delta \rho$. Our strategy is therefore to obtain the $H^2$ estimate at the level of each Galerkin approximant and then pass to the limit. As each approximant is smooth with respect to $x$, the diffusion term may be rewritten in non-divergence form. For this strategy to succeed, we require an improvement on the integrability of $\nabla \rho$, which we get from a parabolic Gehring's Lemma.

\begin{lemma}[Gain of integrability for $\nabla \p$ by Gehring's Lemma]\label{cor:gehring grad p}
Let $f$ be a weak solution of \eqref{eq:main eqn} with initial data satisfying the assumptions of Theorem \ref{thm:smooth}, and $\rho,\p$ as per \eqref{eq:rho def}--\eqref{eq:polarisation def}. Then, there exists $\delta \in (0,1)$ such that, for a.e.~$t \in (0,T)$, there holds 
    \begin{equation*}
        \nabla \p \in L^{2+\delta}(\Omega_{t,T}), \quad \partial_t \rho , \nabla^2 \rho \in L^{2+\delta}(\Omega_{t,T}), \quad \nabla \rho \in L^{2+\delta}(t,T;L^\infty(\Omega)). 
    \end{equation*}
\end{lemma}
\begin{proof}
\noindent 1. \textit{Gehring's Lemma for $\nabla \p$}: The system of equations \eqref{eq:polarisation eq} for $\p$ is diagonal. Indeed, in coordinate form, this system reads as \eqref{eq:polarisation eq coord form}, \textit{i.e.}, 
    \begin{equation}\label{eq:polarisation eq coord form galerkin}
    \partial_t p_i  - \dv((1-\rho)\nabla p_i) = - p_i - \dv \bF_i, 
\end{equation}
where, using the integrability of $\nabla \rho$ from Lemma \ref{cor:nabla rho in L4}, 
\begin{equation*}
    \bF_i = \big((1-\rho)P_{ij}-  p_i \partial_j \rho \big)_{j=1,2} \in L^4(\Omega_{t,T}) \qquad \text{a.e.~}t\in(0,T) \quad \text{for all } q \in [1,\infty), 
\end{equation*}
and we recall $p_i \in L^\infty(\Omega_T)$. The above is a scalar linear equation for $p_i \in L^2(t,T;H^1_\per(\Omega))$ in divergence form with known coefficients, and with matrix $a = (1-\rho)\mathrm{Id}_2 \in L^\infty(\Omega_T)$ satisfying 
\begin{equation*}
    c(t_0)|w|^2 \leq a(t,x) w\cdot w \leq |w|^2 \qquad \forall w \in \mathbb{R}^2, \quad \text{a.e.~}0<t_0 \leq t \leq T. 
\end{equation*}
In turn, an application of \cite[Theorem 8.1]{AuscherGehring}, \textit{i.e.}~Gehring's Lemma in the parabolic setting, implies that there exists $\delta \in (0,1)$ such that $\nabla p_i \in L^{2+\delta}_{\loc}(\Omega_{t,T})$. Furthermore, using a covering argument in the style of \cite[\S 3.1.3]{reg1} and estimate (8.1) of \cite[Theorem 8.1]{AuscherGehring}, we get 
    \begin{equation*}
        \Vert \nabla p_i \Vert_{L^{2+\delta}(\Omega_{t/2,T})} \leq C \Big( \Vert p_i \Vert_{L^2(\Omega_{T})} + \Vert \bF_i \Vert_{L^{2+\delta}(\Omega_{t,T})} + \Vert p_i \Vert_{L^\infty(\Omega_T)} \Big), 
    \end{equation*}
   where $C=C(t,T,\delta,\Omega)>0$. Using the boundedness of $\rho,\p,\P$, and Jensen's inequality, 
        \begin{equation*}
        \Vert \nabla \p \Vert_{L^{2+\delta}(\Omega_{t/2,T})} \leq C \big( 1+ \Vert \nabla \rho \Vert_{L^{4}(\Omega_{t,T})}\big) \implies \nabla \p \in L^{2+\delta}(\Omega_{t,T}) \quad \text{a.e.~} t \in (0,T). 
    \end{equation*}

    \smallskip 

    \noindent 2. \textit{Interpolation}: By returning to the equation \eqref{eq:rho eq} for $\rho$ and applying the classical Calder\'on--Zygmund estimate for the heat equation, we deduce $\partial_t \rho, \nabla^2 \rho \in L^{2+\delta}(\Omega_{t,T})$ a.e.~$t\in(0,T)$. Lemma \ref{lem:dibenedetto classic} yields $\nabla \rho \in L^{4+2\delta}(\Omega_{t,T})$, while the Sobolev Embedding gives 
    \begin{equation*}
        \Vert \nabla \rho(t,\cdot) \Vert_{L^\infty(\Omega)} \leq C_S \Vert \nabla \rho(t,\cdot) \Vert_{W^{1,2+\delta}(\Omega)}, 
    \end{equation*}
    and the result follows from integrating in time and using $\nabla^2 \rho \in L^{2+\delta}(\Omega_{t,T})$. 
\end{proof}

For the purposes of what follows, we fix a Lebesgue point (common to $f,\rho,\p$) $t_0 \in (0,T)$ arbitrarily. For $\delta$ as given by Lemma \ref{cor:gehring grad p}, we define the exponent 
\begin{equation}\label{eq:qdelta def}
q_\delta := 2 + \frac{4}{\delta}, \qquad \frac{1}{2} = \frac{1}{q_\delta} + \frac{1}{2+\delta}. 
\end{equation}
Recall the coordinate-wise scalar equation \eqref{eq:polarisation eq coord form galerkin}. In what follows, we omit the $i$-subscript for clarity of presentation. \eqref{eq:polarisation eq coord form galerkin} is a linear equation for $p$ with $\rho,\mathbf{F}$ given coefficients, and, in the sequel, we write a Galerkin approximation for this equation. We define 
\begin{equation}\label{eq:ansatz p n}
    p^n(t,x) := \sum_{j=1}^n \alpha^n_{j}(t) \varphi_{j}(x), 
\end{equation}
where $\{\varphi_{j}\}_{j}$ is the smooth orthonormal basis of $L^2_\per(\Omega)$ provided by the eigenfunctions of the Laplacian on $\Omega$, \textit{i.e.}, with $0 < \lambda_1 < \lambda_2 < \dots$, 
\begin{equation}\label{eq:relation for eig}
    -\Delta \varphi_j = \lambda_j^2 \varphi_j \qquad j \in \mathbb{N}, 
\end{equation}
and the coefficients $\{\alpha^n_{j}\}_{j=1,\dots,n}$ are to be determined, subject to the initial condition 
\begin{equation}\label{eq:initial data galerkin}
    \alpha^n_{j}(t_0) = \langle p(t_0,\cdot), \varphi_{j} \rangle \qquad \text{for all } j=1,\dots,n \text{ and } n\in\mathbb{N}, 
\end{equation}
where $\langle \cdot , \cdot \rangle$ denotes the $L^2_\per(\Omega)$ inner product; $\langle \varphi , \psi \rangle = \int_\Omega \varphi \psi \d x$. By construction, 
\begin{equation}\label{eq:conv of initial data for galerkin ii}
    \lim_n \Vert p(t_0,\cdot) - p^n(t_0,\cdot) \Vert_{L^2(\Omega)} = 0. 
\end{equation}
Furthermore, since $p(t_0,\cdot) \in H^1_\per(\Omega)$, integration by parts and \eqref{eq:relation for eig} imply 
\begin{equation*}
   \begin{aligned}
       \langle p(t_0,\cdot) ,\! \varphi_j \rangle \! =\!\! -\lambda_j^{-2}\!\!\int_\Omega\!\! p(t_0,x) \Delta \varphi_j(x) \! \d x \!=\! \lambda_j^{-2}\!\!\int_\Omega\!\! \nabla p(t_0,x) \! \cdot \! \nabla \varphi_j(x) \! \d x \!=\! \lambda_j^{-1}\!\langle p (t_0,\cdot) ,\! \lambda_j^{-1}\!\varphi_j \rangle_{H^1}, 
   \end{aligned} 
\end{equation*}
where $\langle \cdot, \cdot \rangle_{H^1}$ denotes the inner product on $H^1_\per(\Omega)$. Recall from the standard spectral theory of the Laplacian that $\{ \lambda_j^{-1}\varphi_j \}_j$ is an orthonormal basis of $H^1_\per(\Omega)$, whence 
\begin{equation*}
    p^n(t_0,\cdot) = \sum_{j=1}^n \langle p(t_0,\cdot) , \lambda_j^{-1} \varphi_j \rangle_{H^1} \lambda_j^{-1}\varphi_j \longrightarrow p(t_0,\cdot) \qquad \text{strongly in } H^1_\per(\Omega), 
\end{equation*}
\textit{i.e.}, there holds 
\begin{equation}\label{eq:conv of initial data for galerkin}
    \lim_n \Vert p(t_0,\cdot) - p^n(t_0,\cdot) \Vert_{H^1(\Omega)} = 0. 
\end{equation}
We emphasise that, for each fixed $n\in\mathbb{N}$, there holds $p^n(t,\cdot) \in C^\infty_\per(\Omega). $

The following result shows that there exists coefficients $\{\alpha^n_j\}_{j=1,\dots,n}$ such that $p^n$ is an appropriate approximation of the solution of equation \eqref{eq:polarisation eq coord form galerkin}. For clarity, we define 
\begin{equation}\label{eq:tilde P def}
\tilde{\P}_i = (P_{ij})_{j=1,2}, 
\end{equation}
and, in what follows, we shall omit the $i$-subscript for brevity. 

\begin{prop}[Galerkin approximation for $p$]\label{prop:galerkin}
Let $f$ be a weak solution of \eqref{eq:main eqn} with initial data satisfying Definition \ref{def:reg initial data}, $\rho,\p$ as per \eqref{eq:rho def}--\eqref{eq:polarisation def}, and $t_0 \in (0,T)$ be an arbitrary Lebesgue point of $f,\rho,\p$. Then, for all $n\in\mathbb{N}$, there exist functions $ p^n \in AC(t_0,T;C^\infty_\per(\Omega))$ satisfying 
    \begin{equation}\label{eq:galerkin approx 1}
    \left\lbrace \begin{aligned} 
        &\partial_t p^n - \dv((1-\rho)\nabla p^n + p^n \nabla \rho) + p^n = - \dv((1-\rho)\tilde{\P}), \\ 
        &p^n(t_0,\cdot) = \sum_{j=1}^n \langle p(t_0,\cdot),\varphi_j \rangle \varphi_j, 
    \end{aligned}\right. 
\end{equation}
in the weak sense dual to $\mathrm{span}\{\varphi_j:j=1,\dots,n\}$ a.e.~$t \in (0,T)$, with the initial data achieved in the strong $L^2$-sense. Furthermore, for a.e.~$t\in(0,T)$, there holds$:$ 
    \begin{equation}\label{eq:strong conv galerkin p}
    p^n \to p \qquad \text{strongly in } L^\infty(t,T;L^2(\Omega)) \cap L^4(\Omega_{t,T}) \text{ and weakly in } L^2(t,T;H^1(\Omega)). 
\end{equation}
Moreover, for a.e.~$t\in(0,T)$, we have the uniform bounds$:$ 
\begin{equation}\label{eq:unif bounds for galerkin}
    \sup_n \Vert p^n \Vert_{L^\infty(t,T;L^2(\Omega))} + \sup_n \Vert \nabla p^n \Vert_{L^2(\Omega_{t,T})} + \sup_n \Vert \nabla^2 p^n \Vert_{L^2(\Omega_{t,T})} < \infty. 
\end{equation}
\end{prop}

The regularity of $\rho$ and $p^n$ from the ansatz \eqref{eq:ansatz p n} imply that \eqref{eq:galerkin approx 1} is equivalent to 
\begin{equation}\label{eq:galerkin approx 2}
     \partial_t p^n - (1-\rho)\Delta p^n - p^n \Delta \rho + p^n = - \dv((1-\rho)\tilde{\P}); 
\end{equation}
this observation is what enables us to perform the $H^2$-type estimate on the Galerkin approximations $p^n$, by using an adaptation of the formal argument presented in \eqref{eq:formal estimate H2 p}.

\begin{proof}
\noindent 1. \textit{Existence of Galerkin approximation}: Equation \eqref{eq:galerkin approx 1} is a linear equation for $p^n$ with known coefficients $\rho,\mathbf{F}$. We show that there exist coefficients $\{\alpha^n_j\}$ for the ansatz \eqref{eq:ansatz p n} such that \eqref{eq:galerkin approx 1} is satisfied in the weak sense dual to $\mathrm{span}\{\varphi_j\}_{j=1,\dots,n}$. We test the equation against $\varphi_{j}$ and, using the orthonormal property of the basis, we get 
\begin{equation*}
    \begin{aligned}
        \frac{\der \alpha^n_{j}}{\der t}(t)  + \sum_{k=1}^n  \alpha^n_{k}(t) \bigg( \int_\Omega \Big( (1-\rho) \nabla \varphi_k + \varphi_k \nabla \rho \Big) \cdot \nabla\varphi_{j} \d x \bigg) + \alpha^n_j(t) = \int_\Omega (1-\rho)\tilde{\P}\cdot \nabla \varphi_j \d x, 
    \end{aligned}
\end{equation*}
\textit{i.e.}, by defining the vector of coefficients $\boldsymbol{\alpha}^n := (\alpha^n_{j})_{j}$, the initial data corresponding to \eqref{eq:initial data galerkin} $\boldsymbol{\alpha}^n_0 = (\langle p(t_0,\cdot), \varphi_{j} \rangle )_j $, $\mathbf{g}^n := (\int_\Omega (1-\rho) \mathbf{F}\cdot \nabla \varphi_j \d x)_{j}$, and the matrix $A = (A_{jk})_{jk}$ with 
\begin{equation*}
    A_{{jk}}(t) := \delta_{jk} + \int_\Omega \Big( (1-\rho) \nabla \varphi_k + \varphi_k \nabla \rho \Big) \cdot \nabla\varphi_{j} \d x , 
\end{equation*}
we have the linear system of ODEs: 
\begin{equation}\label{eq:system ode linear}
   \left\lbrace\begin{aligned} &\frac{\der \boldsymbol{\alpha}^n}{\der t} + A\boldsymbol{\alpha}^n = \mathbf{g}^n \qquad \text{for } t \in (t_0,T), \\ 
   &\boldsymbol{\alpha}^n(t_0) = \boldsymbol{\alpha}^n_0. 
   \end{aligned}\right. 
\end{equation}
Since all coefficients are locally integrable in time, the standard theory for linear systems of ODEs (\textit{cf.}~Carath\'eodory's Theorem) implies the existence and uniqueness of a solution $\boldsymbol{\alpha} \in AC([t_0,T])$ with the initial data achieved pointwise. It follows that there exists $p^n \in AC([t_0,T];C^\infty_\per(\Omega))$ a solution of \eqref{eq:galerkin approx 1} in the sense dual to $\mathrm{span}\{\varphi_{j}: j=1,\dots,n\}$.

\smallskip 

\noindent 2. \textit{Uniform estimates in $H^1$}: Since $p^n \in \mathrm{span}\{\varphi_j:j=1,\dots,n\}$, we may insert it into the weak formulation \eqref{eq:galerkin approx 1}. We get the standard $H^1$-type estimate 
\begin{equation}\label{eq:unif est for pn galerkin H1}
   \begin{aligned}
        \Vert p^n &\Vert_{L^\infty(t_0,T;L^2(\Omega))}^2 + \Vert p^n \Vert^2_{L^2(\Omega_{t_0,T})} + \Vert \nabla p^n \Vert^2_{L^2(\Omega_{t_0,T})} \\ 
        &\leq C \Big( \Vert p^n(t_0,\cdot) \Vert^2_{L^2(\Omega)} + \Vert \tilde{\P} \Vert^2_{L^2(\Omega_{t_0,T})} + \Vert \nabla \rho \Vert^2_{L^2(\Omega_{t_0,T})} + \Vert p \Vert^2_{L^2(\Omega_{t_0,T})} \Big) \\ 
        &\leq C \Big( 1 + \Vert \tilde{\P}\Vert^2_{L^2(\Omega_{t_0,T})} + \Vert \nabla \rho \Vert^2_{L^2(\Omega_{t_0,T})} \Big)
   \end{aligned}
\end{equation}
for some positive constant $C$ depending on $t_0$ but not on $n$, where we obtain the final line using \eqref{eq:conv of initial data for galerkin ii} and the boundedness $\Vert p(t_0,\cdot) \Vert_{L^\infty(\Omega)} \leq 1$, which imply $\sup_n \Vert p^n(t_0,\cdot) \Vert_{L^2(\Omega)} \leq C$. 

It follows from the previous estimate that $\{p^n \}_n$ is uniformly bounded in $L^2(t_0,T;H^1(\Omega))$ and $L^\infty(t_0,T;L^2(\Omega))$. We deduce from Alaoglu's Theorem that $\{p^n\}_n$ converges to some limit function $\bar{p}$ weakly in $L^2(t_0,T;H^1(\Omega))$, weakly-* in $L^\infty(t_0,T;L^2(\Omega))$. Using the weak (resp.~weak-*) lower semicontinuity of the norm, there holds 
\begin{equation}\label{eq:H1 est pi}
   \begin{aligned}
       \!\! \Vert \bar{p} \Vert_{L^\infty(t_0,T;L^2(\Omega))}^2 \! +\! \Vert \bar{p} \Vert^2_{L^2(\Omega_{t_0,T})} \!\!+\!\Vert \nabla \bar{p} \Vert^2_{L^2(\Omega_{t_0,T})} \! \leq \! C \Big( \!1\! +\! \Vert \tilde{\P} \Vert^2_{L^2(\Omega_{t_0,T})} \!\!+\! \Vert \nabla \rho \Vert^2_{L^2(\Omega_{t_0,T})} \! \Big). 
   \end{aligned}
\end{equation}
Since the equation \eqref{eq:galerkin approx 1} is linear in $p^n$, this weak convergence is sufficient to pass to the limit in the equation. It follows that $\bar{p}$ is a weak solution of 
\begin{equation}\label{eq:limit for galerkin}
    \left\lbrace\begin{aligned}
        & \partial_t \bar{p} - \dv((1-\rho)\nabla \bar{p} + \bar{p} \nabla \rho) +\bar{p} = - \dv((1-\rho)\tilde{\P}) , \\ 
        & \bar{p}(t_0,\cdot) = p(t_0,\cdot), 
    \end{aligned}\right.
\end{equation}
in duality with test functions belonging to $L^2(t_0,T;H^1_\per(\Omega))$. Furthermore, we see directly from \eqref{eq:limit for galerkin} that $\partial_t \bar{p} \in L^2(t_0,T;(H^1_\per(\Omega))')$, whence we have the embedding $\bar{p} \in C([t_0,T];L^2(\Omega))$ and hence the initial data is achieved in the sense 
\begin{equation}\label{eq:L2 conv pi to p}
    \lim_{t \to t_0^+} \Vert \bar{p}(t,\cdot) - p(t_0,\cdot) \Vert_{L^2(\Omega)} = 0. 
\end{equation}

\smallskip 

\noindent 3. \textit{Uniqueness of limits}: We show $\bar{p} = p$. Define the difference $w := p - \bar{p}$, and observe: 
\begin{equation}\label{eq:difference of galerkin}
    \left\lbrace\begin{aligned}
        & \partial_t w - \dv((1-\rho)w + w \nabla \rho) + w = 0, \\ 
        & w(t_0,\cdot) = 0, 
    \end{aligned}\right.
\end{equation}
in the weak sense dual to $L^2(t_0,T;H^1(\Omega))$, where the initial data is achieved in the strong $L^2$-sense, as per \eqref{eq:L2 conv pi to p}. We test the above with $w \in L^2(t_0,T;H^1_\per(\Omega))$, which is an admissible test function, and using Young's inequality we obtain 
\begin{equation*}
    \begin{aligned}
        \frac{\der}{\der t}\int_\Omega w^2 \d x + \int_\Omega |\nabla w|^2 \d x + \int_\Omega w^2 \d x \leq C \int_\Omega |w|^2 |\nabla \rho|^2 \d x \qquad \text{for a.e.~}t \in (t_0,T), 
    \end{aligned}
\end{equation*}
where the constant $C$ depends on $t_0$ but not on $n$ and $t$. Integrating over $[t_0,t]$, 
\begin{equation}\label{eq:needed for sup wn argument sec 6}
    \begin{aligned}
        \Vert w(t,\cdot)\Vert^2_{L^2(\Omega)} \!+\!\! \int_{t_0}^t \!\int_\Omega |\nabla w|^2 \d x \d s \leq {\Vert w(t_0,\cdot)\Vert^2_{L^2(\Omega)}} \!+\! C \int_{t_0}^t \!\! \! \Vert \nabla \rho(s,\cdot) \Vert^2_{L^\infty(\Omega)} \Vert w(s,\cdot)\Vert^2_{L^2(\Omega)} \d s. 
    \end{aligned}
\end{equation}
Note that $w \in L^\infty(t_0,T;L^2(\Omega))$ by \eqref{eq:H1 est pi}, so that the final integral is well-defined, as $\nabla \rho \in L^{2+\delta}(t_0,T;L^\infty(\Omega))$ from Lemma \ref{cor:gehring grad p}. Thus, we apply Gr\"onwall's Lemma, and get 
\begin{equation*}
    \begin{aligned}
        \Vert w \Vert^2_{L^\infty(t_0,T;L^2(\Omega))} \leq \Vert w(t_0,\cdot)\Vert^2_{L^2(\Omega)}\exp\bigg( C \int_{t_0}^T \Vert \nabla \rho(s,\cdot) \Vert^2_{L^\infty(\Omega)} \d s \bigg) = 0, 
    \end{aligned}
\end{equation*}
by \eqref{eq:L2 conv pi to p}. We have therefore shown that $\bar{p} = p$ a.e.~$\Omega_{t_0,T}$, as required. 

We also show the strong convergence $p^n \to p$ in $L^\infty(t_0,T;L^2(\Omega))$. To see this, we write the equation for $w^n := p^n - p$ as per \eqref{eq:difference of galerkin}:
\begin{equation*}
        \left\lbrace\begin{aligned}
        & \partial_t w^n - \dv((1-\rho)w^n + w^n \nabla \rho) + w^n = 0, \\ 
        & w^n(t_0,\cdot) = w^n(t_0,\cdot), 
    \end{aligned}\right.
\end{equation*}
and an argument identical to the previous one implies 
\begin{equation*}
    \begin{aligned}
        \Vert w^n \Vert^2_{L^\infty(t_0,T;L^2(\Omega))} \leq \Vert w^n(t_0,\cdot)\Vert^2_{L^2(\Omega)}\exp\bigg( C \int_{t_0}^T \Vert \nabla \rho(s,\cdot) \Vert^2_{L^\infty(\Omega)} \d s \bigg) \to 0 \quad \text{as } n \to \infty, 
    \end{aligned}
\end{equation*}
by \eqref{eq:conv of initial data for galerkin ii}, as required. The strong convergence in $L^4(\Omega_{t_0,T})$ follows from the Gagliardo--Nirenberg inequality: 
\begin{equation*}
    \Vert w^n(t,\cdot) \Vert_{L^4(\Omega)} \leq C \Vert \nabla w^n(t,\cdot) \Vert_{L^2(\Omega)}^{\frac{1}{2}} \Vert w^n(t,\cdot) \Vert^{\frac{1}{2}}_{L^2(\Omega)} + C \Vert w^n(t,\cdot) \Vert_{L^1(\Omega)}, 
\end{equation*}
and thus, using also the $n$-independent estimate on $\Vert \nabla w^n \Vert_{L^2(\Omega_{t_0,T})}$ from \eqref{eq:needed for sup wn argument sec 6}, 
\begin{equation*}
    \begin{aligned}
        \Vert w^n \Vert_{L^4(\Omega_{t_0,T})}^4 &\leq C \big(\sup_n \Vert \nabla w^n \Vert_{L^2(\Omega_{t_0,T})}^{2}\big) \Vert w^n \Vert^{2}_{L^\infty(t_0,T;L^2(\Omega))} + C \int_{t_0}^T \Vert w^n(t,\cdot) \Vert_{L^1(\Omega)}^4 \d t \\ 
        &\leq C\Big(\Vert w^n \Vert^{2}_{L^\infty(t_0,T;L^2(\Omega))} + \Vert w^n \Vert^{4}_{L^\infty(t_0,T;L^2(\Omega))}\Big) \to 0, 
    \end{aligned}
\end{equation*}
where the final constant $C$ depends on $T,\Omega$ but not on $n$, and we used Jensen's inequality to obtain the second line. We get $\lim_{n\to\infty} \Vert w^n \Vert_{L^4(\Omega_{t_0,T})}=0$, and \eqref{eq:strong conv galerkin p}--\eqref{eq:unif bounds for galerkin} follow. 
\end{proof}

We are now ready to give the proof of Lemma \ref{lem:H2 for p}. 

\begin{proof}[Proof of Lemma \ref{lem:H2 for p}]
We return to \eqref{eq:galerkin approx 2} and obtain the desired estimates on the second derivatives. Recall that \eqref{eq:galerkin approx 1} only holds in duality with $\mathrm{span}\{\varphi_j:j=1,\dots,n\}$. By testing the equation \eqref{eq:galerkin approx 1} against the eigenfunction $\varphi_j$, using the relation \eqref{eq:relation for eig}, and summing over $j=1,\dots,n$, we are able to introduce $\Delta p^n \in \mathrm{span}\{\varphi_j\}_{j=1,\dots,n}$ as a test function into \eqref{eq:galerkin approx 1}. Using the regularity $p^n \in AC(t_0,T;C^\infty_\per(\Omega))$, by using the non-divergence form \eqref{eq:galerkin approx 2} of the equation, we obtain, with $g := - \dv((1-\rho)\tilde{\P}) \in L^2(\Omega_{t_0,T})$, 
\begin{equation*}
    \begin{aligned}
        \frac{1}{2}\frac{\der}{\der t}\int_\Omega |\nabla p^n|^2  \d x \!+\! \int_\Omega (1-\rho) |\Delta p^n|^2 \d x = -\!\int_\Omega p^n \Delta p^n \Delta \rho \d x \! +\!\int_\Omega p^n \Delta p^n \d x \!+\! \int_\Omega g \Delta p^n \d x, 
    \end{aligned}
\end{equation*}
whence, by using the Young inequality and the boundedness \eqref{eq:unif bounds for galerkin}, integrating in time over the interval $[t_0,t]$ yields 
\begin{equation}\label{eq:first step H2 p bound}
    \begin{aligned}
       \Vert \nabla p^n(t)\Vert_{L^2(\Omega)}^2 + \Vert \Delta p^n & \Vert_{L^2(\Omega_{t_0,t})}^2 \\ 
       \leq & \,  C\!\int_{t_0}^t\Vert p^n(s) \Vert_{L^{q_\delta}(\Omega)}^2 \Vert \Delta \rho(s) \Vert_{L^{2+\delta}(\Omega)}^{2} \d s \\ 
       &+ \underbrace{C \Big(\Vert \nabla p^n(t_0,\cdot) \Vert_{L^2(\Omega)}^2 \!+\! \Vert p^n \Vert^2_{L^2(\Omega_{t_0,T})} \!+\! \Vert g\Vert^2_{L^2(\Omega_{t_0,T})}\Big)}_{\leq C'}, 
    \end{aligned}
\end{equation}
where $q_\delta$ is given in \eqref{eq:qdelta def} and the positive constants $C,C'$ depend on $t_0$ but not on $n,t$; recall that the convergence \eqref{eq:conv of initial data for galerkin} of the initial data implies $\sup_n \Vert \nabla p^n(t_0,\cdot) \Vert_{L^2(\Omega)} < \infty$. Meanwhile, observe that for a.e.~$s \in (t_0,T)$, using Sobolev's inequality, 
\begin{equation*}
    \Vert p^n(s) \Vert_{L^{q_\delta}(\Omega)} \leq C_\delta \Big( \Vert p^n(s) \Vert_{L^2(\Omega)} + \Vert \nabla p^n(s) \Vert_{L^2(\Omega)} \Big) \leq C_\delta \Big( 1 + \Vert \nabla p^n(s) \Vert_{L^2(\Omega)} \Big), 
\end{equation*}
where $C_\delta$ depends only on $\delta,t_0$, and we used the uniform-in-$n$ estimate \eqref{eq:unif bounds for galerkin} to bound the contribution from $\Vert p^n(s) \Vert_{L^2(\Omega)}$. By returning to \eqref{eq:first step H2 p bound}, we deduce, using also the integrability $\Vert \Delta \rho(s) \Vert^2_{L^{2+\delta}(\Omega)} \in L^1(t_0,T)$ from Lemma \ref{cor:gehring grad p}, 
\begin{equation}\label{eq:second step H2 p bound}
    \begin{aligned}
       \Vert \nabla p^n(t)\Vert_{L^2(\Omega)}^2 + \Vert \Delta p^n \Vert_{L^2(\Omega_{t_0,t})}^2 \leq  C\bigg( 1 + \int_{t_0}^t\Vert \nabla p^n(s) \Vert_{L^{2}(\Omega)}^2 \Vert \Delta \rho(s) \Vert_{L^{2+\delta}(\Omega)}^{2} \d s \bigg), 
    \end{aligned}
\end{equation}
where the positive constant $C$ depends on $t_0,\delta$ but is independent of $n,t$. We emphasise that the integrand above is integrable on account of $p^n \in AC(t_0,T;C^\infty_{\per}(\Omega))$. Using again that $\Vert \Delta \rho(s) \Vert^2_{L^{2+\delta}(\Omega)} \in L^1(t_0,T)$, an application of Gr\"onwall's Lemma implies 
\begin{equation*}
    \begin{aligned}
       \Vert \nabla p^n\Vert_{L^\infty(t_0,T;L^2(\Omega))}^2 \leq  C \exp \bigg( \int_{t_0}^T \Vert \Delta \rho(s) \Vert_{L^{2+\delta}(\Omega)}^{2} \d s \bigg) \leq C \exp \Big({T^{\frac{\delta}{2+\delta}}\Vert \Delta \rho \Vert^2_{L^{2+\delta}(\Omega_{t_0,T})}} \Big). 
    \end{aligned}
\end{equation*}
Thus, by returning to \eqref{eq:second step H2 p bound}, we get 
\begin{equation*}
    \begin{aligned}
       \Vert \nabla p^n\Vert_{L^\infty(t_0,T;L^2(\Omega))}^2 & + \Vert \Delta p^n\Vert_{L^2(\Omega_{t_0,T})}^2 \leq C, 
    \end{aligned}
\end{equation*}
for some positive constant depending on $\delta,t_0$, but independent of $n$. Letting $n\to\infty$ using the weak (weak-*) lower semicontinuity of the norms, we recover 
\begin{equation*}
    \nabla p \in L^\infty(t_0,T;L^2(\Omega)), \qquad \Delta p \in L^2(\Omega_{t_0,T}) \qquad \text{a.e.~}t_0 \in (0,T). 
\end{equation*}
Since the previous estimates hold for all coordinates $p_i$ ($i=1,2$), the result now follows by using the periodic Calder\'on--Zygmund inequality of Lemma \ref{lem:CZ periodic}, and the interpolation  of Lemma \ref{lem:dibenedetto classic} over $\Omega$ to obtain $\nabla p \in L^4(\Omega_{t_0,T})$. 

Finally, the estimate \eqref{eq:est H2 p formal at the end} is verified by testing the equation \eqref{eq:non div form p eqn} and arguing as per \eqref{eq:formal estimate H2 p}. This argument is now rigourous, since the earlier part of the proof showed $\Delta \p \in L^2(\Omega_{t,T})$; thus, it is an admissible test function.
\end{proof}

\subsection{De Giorgi type iterations for $f$ and boundedness of spatial derivatives}

The main result of this subsection is as follows, and the rest of the subsection is devoted to its proof. 

\begin{prop}[$L^\infty$ boundedness of $f$]\label{lem:Linfty for f}
    Let $f$ be a weak solution of \eqref{eq:main eqn} with initial data satisfying the assumptions of Theorem \ref{thm:smooth}. Then, for a.e.~$t \in (0,T)$, there holds $$f \in L^\infty(\Upsilon_{t,T}).$$ 
\end{prop}

Having obtained all the necessary boundedness on $\rho$ and $\p$, we now focus on $f$. We begin by performing the classical $H^1$-type estimate for $f$. 

\begin{lemma}[$H^1$-type estimate for $f$]\label{lem:H1 for f}
    Let $f$ be a weak solution of \eqref{eq:main eqn} with initial data satisfying the assumptions of Theorem \ref{thm:smooth}. Then, a.e.~$t \in (0,T)$, there holds $f \in L^\infty(t,T;L^2(\Upsilon))$, $\nabla_{\bxi} f \in L^2(\Upsilon_{t,T})$, $\partial_t f \in L^2(t,T;(H^1_\per(\Upsilon))')$.
\end{lemma}
\begin{proof}
 For clarity of presentation, we only present the following formal argument, which can be rendered rigorous by adapting the proof of Lemma \ref{lem:H1 p}. Test \eqref{eq:main eqn} with $f$: 
    \begin{equation*}
       \begin{aligned} \frac{1}{2}\frac{\der}{\der t}\int_\Upsilon\! |f|^2 \d \bxi\!+\! \int_\Upsilon\! \Big( (1\!-\!\rho) |\nabla f|^2 \!+\! |\partial_\theta f|^2 \Big) \d \bxi \!=\! \int_\Upsilon \! f\nabla f\! \cdot\! \nabla \rho \d \bxi \!+\! \int_\Upsilon (1\!-\!\rho)f \nabla f \!\cdot\! \e (\theta) \d \bxi, 
       \end{aligned}
    \end{equation*}
    and hence, using Young's inequality and the lower bound on $1-\rho$ from Proposition \ref{prop:lower bound 1-rho}, 
       \begin{equation*}
       \begin{aligned} \Vert f \Vert_{L^\infty(t,T;L^2(\Upsilon))}^2 + \Vert & \nabla f \Vert_{L^2(\Upsilon_{t,T})}^2 + \Vert\partial_\theta f\Vert_{L^2(\Upsilon_{t,T})}^2 \\ 
       &\leq C \Big(\Vert f(t,\cdot) \Vert_{L^2(\Upsilon)}^2 + \Vert f \Vert_{L^3(\Upsilon_{t,T})}^2 \Vert \nabla \rho \Vert_{L^6(\Upsilon_{t,T})}^2 + \Vert f \Vert_{L^2(\Upsilon_{t,T})}^2 \Big), 
       \end{aligned}
    \end{equation*}
   for a.e.~$t \in (0,T)$, for some positive constant $C$ depending only on $t,T,\Upsilon$, and we note that the second term on the right-hand side is finite due to Proposition \ref{lem:nabla rho in L8}. 
\end{proof}

Armed with the previous result, we are ready to prove Lemma \ref{lem:Linfty for f}, by means of De Giorgi's iteration strategy. The proof is similar to that of Proposition \ref{prop:lower bound 1-rho}.

\begin{proof}[Proof of Proposition \ref{lem:Linfty for f}]
    By Lemma \ref{lem:H1 for f}, the equation \eqref{eq:main eqn} may be interpreted in the weak sense of \eqref{eq:weak sense eq} dual to $L^2(t,T;H^1_\per(\Upsilon))$ for a.e.~$t\in(0,T)$, and may be rewritten as 
    \begin{equation}\label{eq:f eqn in big div form matrix A}
        \partial_t f + \dv_{\bxi} (Uf) = \dv_{\bxi}(A\nabla_{\bxi}f), 
    \end{equation}
    where the coefficients of the equation are given by 
    \begin{equation}\label{eq:A def matrix}
        A = \left( \begin{matrix}
            (1-\rho)\mathrm{Id}_2 & 0 \\ 
            0 & 1 
        \end{matrix} \right), \quad U = \left( \begin{matrix}
        -\nabla \rho + (1-\rho) \e(\theta) \\ 0
        \end{matrix} \right). 
    \end{equation}
    Fix $t_0 \in (0,T)$ a Lebesgue point of $f$ and suppose without loss of generality $t_0 < 1$. Lemma \ref{cor:nabla rho in L4} implies $U \in L^8(\Upsilon_{t_0,T})$. Meanwhile, $A \in L^\infty(\Upsilon_T)$ satisfies 
    \begin{equation}\label{eq:coercivity for A}
       c(t_0)|w|^2 \leq A w \cdot w \leq |w|^2 \qquad \forall w \in \mathbb{R}^3, 
    \end{equation}
 \textit{i.e.}, usual coercivity assumptions, by virtue of the lower bound on $1-\rho$ of Proposition \ref{prop:lower bound 1-rho}.

    Equation \eqref{eq:f eqn in big div form matrix A} is a linear equation in $f$ with given coefficients $A,U$. Define 
    \begin{equation}\label{eq:g def at start of moser proof}
        g(t,\bxi) := \frac{f(t,\bxi)}{L}, 
    \end{equation}
    where $L>0$ is to be determined. Observe that $g$ also satisfies equation \eqref{eq:f eqn in big div form matrix A} in the sense dual to $L^2(t,T;H^1_\per(\Upsilon))$. Until the final step of the proof, we work with $g$ instead of $f$. 

    \smallskip 

  \noindent  1. \textit{Caccioppoli inequality}: Let $k > 0$ and recall the Stampacchia truncations $\mathscr{T}_k g := (g-k)_+$. Note $\mathscr{T}_k g \in C([t_0/2,T];L^2_\per(\mathbb{R}^3)) \cap L^2(t_0/2,T;H^1_\per(\mathbb{R}^3))$ is an admissible test function to insert into the weak formulation of the equation for $g$, namely, \eqref{eq:f eqn in big div form matrix A}. We get 
    \begin{equation*}
        \begin{aligned}
            \frac{1}{2}\frac{\der}{\der t}\int_{\Upsilon} |\mathscr{T}_k g|^2 \d \bxi \!+\!\int_{\Upsilon} (A\nabla_{\bxi}\mathscr{T}_k g) \cdot \nabla_{\bxi} \mathscr{T}_k g \d \bxi = \int_{\Upsilon} \mathscr{T}_k g (U \cdot \nabla_{\bxi} \mathscr{T}_k g) \d \bxi + k\int_{\Upsilon} U \cdot \nabla_{\bxi} \mathscr{T}_k g \d \bxi. 
        \end{aligned}
    \end{equation*}
The assumptions on $A$ and $U$ and Young's inequality imply, for some positive constant $M$ depending only on $t_0,T,\Upsilon$, and independent of $L$, for a.e.~$t \in (t_0/2,T)$ 
\begin{equation*}
            \begin{aligned}
            \frac{\der}{\der t}\int_{\Upsilon} |\mathscr{T}_k g|^2 \d \bxi +\int_{\Upsilon} |\nabla_{\bxi} \mathscr{T}_k g |^2 \d \bxi\leq & \, M(1+k^2) \int_{\Upsilon}  (1 + |U|^2) (|\mathscr{T}_k g|^2 + \mathds{1}_{\{\mathscr{T}_k g>0\}}) \d \bxi,  
        \end{aligned}
\end{equation*}
and we may freely assume $M>1$. Integrating in time, we obtain that there exists $M>0$ independent of $U,k,f,L$ but depending on $t_0$ such that, for a.e.~$t_0/2<s<t<T$, 
\begin{equation}\label{eq:caccioppoli}
            \begin{aligned}
            \int_{\Upsilon} |\mathscr{T}_k g|^2 \d \bxi\bigg|_{t} - \int_{\Upsilon} & |\mathscr{T}_k g|^2 \d \bxi\bigg|_{s} + \int_s^t \int_{\Upsilon} |\nabla_{\bxi} \mathscr{T}_k g |^2 \d \bxi \d \tau \\ 
            &\leq M(1+k^2) \int_s^t \int_{\Upsilon} (1 + |U|^2) (|\mathscr{T}_k g|^2 + \mathds{1}_{\{\mathscr{T}_k g>0\}}) \d \bxi \d \tau, 
        \end{aligned}
\end{equation}
and thus, since $U \in L^8(\Upsilon_{t_0,T})$, we are in a setting analogous to that of \cite[\S 3.2]{reg1}; we remark that $U \in L^q(\Upsilon_{t_0,T})$ for $q > 5$ is sufficient to make the method work. 

    \smallskip 

    \noindent 2. \textit{Setup for iterations}: Recall the increasing sequence of times $\{T_n\}_n$ defined in \eqref{eq:Tn def} and the increasing sequence of cut-off constants $\{\kappa_n\}_n$ defined in \eqref{eq:kappa n def}. Define shorthand $g_n := \mathscr{T}_{\kappa_n} g$, and $\mathscr{G}_n := \esssup_{t \in [T_n,T]}  \int_\Upsilon |g_n|^2 \d \bxi  + \int_{T_n}^T \int_\Upsilon |\nabla_{\bxi}g_n|^2 \d \bxi \d t, $ noting that 
\begin{equation}\label{eq:G0 def}
\mathscr{G}_0 = \frac{1}{L^{2}} (\Vert f \Vert_{L^\infty(t_0/2,T;L^2(\Upsilon))}^2 + \Vert \nabla_{\bxi} f \Vert_{L^2(\Upsilon_{t_0/2,T})})^2. 
\end{equation}

We will prove, with $\epsilon = 1-\frac{2}{8}-\frac{3}{5} \in (0,1)$ and $C_*$ independent of $n$, the recursion estimate 
\begin{equation}\label{eq:recursion estimate moser}
    \mathscr{G}_{n} \leq C_*^n \mathscr{G}_{n-1}^{1+\epsilon} \qquad \forall n \in \mathbb{N} \cup \{0\}.
\end{equation}

Substitute $g_{n}$ into the inequality \eqref{eq:caccioppoli} and, let $T_{n-1} \leq s \leq T_{n} \leq t$, and then integrate in $s$ over $[T_{n},T_{n+1}]$. By arguing as we did in \eqref{eq:some of the way there moser for w}, by defining $C_1 :=8(t_0^{-1} + M)$, we obtain 
    \begin{equation}\label{eq:some of the way there moser}
            \begin{aligned}
         \mathscr{G}_n \leq C_1^n \int_{T_{n-1}}^T \int_{\Upsilon} \mathds{1}_{\{g_n > 0\}} (1 + |U|^2) |g_{n-1}|^2 \d \bxi \d \tau. 
        \end{aligned}
\end{equation}

   \smallskip 

    \noindent 3. \textit{Recursive estimate}: Sobolev inequality's implies, for some positive constant $C_S$ independent of $t,n$, $\Vert g_n(t,\cdot) \Vert_{L^6(\Upsilon)}^2 \leq C_S^2 \Vert g_n(t,\cdot) \Vert_{H^1(\Upsilon)}^2 \leq C_S \mathscr{G}_n$. Using H\"older's inequality, 
    \begin{equation}\label{eq:interp 1}
        \Vert g_n \Vert^2_{L^2(\Upsilon_{T_n,T})} \leq |\Upsilon|^{\frac{2}{3}}\Vert g_n \Vert^2_{L^2(T_n,T;L^6(\Upsilon))} \leq  C_S  |\Upsilon|^{\frac{2}{3}} \mathscr{G}_n. 
    \end{equation}
Meanwhile, the Interpolation Lemma \ref{lem:dibenedetto classic} implies 
\begin{equation}\label{eq:interp 2}
    \Vert g_n^2 \Vert_{L^{\frac{5}{3}}(\Upsilon_{T_n,T})} \leq \Vert g_n \Vert^2_{L^{\frac{10}{3}}(\Upsilon_{T_n,T})} \leq C_I^2(1+t_0)^2 \mathscr{G}_n. 
\end{equation}
Returning to \eqref{eq:some of the way there moser} and using \eqref{eq:interp 2}, H\"older, Minkowski, and Jensen inequalities, 
\begin{equation*}
    \begin{aligned}
        \mathscr{G}_n \leq 4 C_1^n |\{g_n>0\} \cap [T_{n-1},T]|^\epsilon (1+ \Vert |U|^2 \Vert_{L^{4}(\Upsilon_{t_0/2,T})})  \Vert g_{n-1}^2 \Vert_{L^{\frac{5}{3}}(\Upsilon_{T_{n-1},T})}, 
    \end{aligned}
\end{equation*}
and $\epsilon = 1-\frac{2}{8}-\frac{3}{5} \in (0,1)$. In turn, defining $C_2 := 4C_1 (1+ \Vert U \Vert_{L^8(\Upsilon_{t_0/2,T})})^2 (1+C_I)^2(1+t_0)^2$, 
\begin{equation}\label{eq:moser almost there}
    \mathscr{G}_n \leq C_2^n |\{g_n >0 \} \cap [T_{n-1},T]|^\varepsilon \mathscr{G}_{n-1}. 
\end{equation}
Analogously to \eqref{eq:to get from fn-1 to fn for w} we have $\{g_n \!>\! 0\} \!\subset\! \{ g_{n-1} \!>\! 2^{-n}\}$. By Markov's inequality and \eqref{eq:interp 1}, 
\begin{equation*}
   \begin{aligned}
       |\{g_n \!>\!0 \} \!\cap\! [T_{n-1},T]| \leq  |\{g_{n-1} \mathds{1}_{[T_{n-1},T]} \!>\! 2^{-n} \}| \leq \frac{\Vert g_{n-1}\Vert_{L^2(\Upsilon_{T_{n-1},T})}^2}{ 2^{-2n}} \leq C_S  |\Upsilon|^{\frac{2}{3}}  2^{2n}\mathscr{G}_{n-1}. 
   \end{aligned}
\end{equation*}
Returning to \eqref{eq:moser almost there} and letting $C_* := (1+C_S  |\Upsilon|^{\frac{2}{3}}) C_2$, we have proved \eqref{eq:recursion estimate moser}. 

\smallskip 

\noindent 4. \textit{Conclusion}: Estimate \eqref{eq:recursion estimate moser} and the iteration lemma \cite[\S 1, Lemma 4.1]{DiBenedetto} yield: if 
\begin{equation}\label{eq:initialisation moser}
    \mathscr{G}_0 \leq C_*^{-1/\epsilon^2}, 
\end{equation}
then there holds $\lim_{n\to\infty}\mathscr{G}_n = 0$. In turn, using \eqref{eq:G0 def}, by choosing, \textit{e.g.}, 
\begin{equation}\label{eq:L choice at end}
    L = 2 C_*^{1/\epsilon^2} \Big(\Vert f \Vert_{L^\infty(t_0/2,T;L^2(\Upsilon))}^2 + \Vert \nabla_{\bxi} f \Vert_{L^2(\Upsilon_{t_0/2,T})} \Big)^2, 
\end{equation}
we see that \eqref{eq:initialisation moser} is satisfied. Using the Monotone Convergence Theorem and \eqref{eq:interp 1}, 
\begin{equation*}
    \int_{t_0}^T \int_\Upsilon (g-1)_+ \d \bxi \d t = \lim_{n\to\infty} \Vert g_n \Vert^2_{L^2(\Upsilon_{T_n,T})} \leq C_S |\Upsilon|^{\frac{2}{3}} \lim_{n\to\infty}\mathscr{G}_n = 0, 
\end{equation*}
whence $0 \leq g \leq 1$ a.e.~$\Upsilon_{t_0,T}$. By \eqref{eq:g def at start of moser proof}, $$0 \leq f \leq L \quad \text{a.e.~}\Upsilon_{t_0,T},$$ where $L(t_0,\Vert U \Vert_{L^8(\Upsilon_{t_0/2,T})} , T , \Upsilon)$ is prescribed by \eqref{eq:L choice at end}. The proof is complete. 
\end{proof}

\section{Proof of Theorem \ref{thm:smooth}: Smoothness}\label{sec:higher reg}

This section is devoted to the higher order regularity of the equation \eqref{eq:main eqn}. We shall bootstrap boundedness results and apply time derivatives to \eqref{eq:main eqn} to obtain the boundedness of all derivatives of $f$, thereby yielding smoothness by Morrey's embedding. 

\subsection{Regularity bootstrap for first time derivative}
The main result of this subsection is as follows.

\begin{prop}[Higher Integrability for $\partial_t f$ and $\nabla_{\bxi}^2 f$]\label{lem:all space derivs}
   Let $f$ be a weak solution of \eqref{eq:main eqn} with initial data satisfying Definition \ref{def:reg initial data}. Then, a.e.~$t \in (0,T)$ and all $q \in [1,\infty)$,  $$\partial_t f, \nabla_{\bxi}^2 f \in L^q(\Upsilon_{t,T}).$$ 
\end{prop}

We prove this result in several steps. To begin with, we write the $H^2$-estimate for $f$. 

\begin{lemma}[$H^2$-type estimate for $f$]\label{lem:H2 for f}
Let $f$ be a weak solution of \eqref{eq:main eqn} with initial data satisfying Definition \ref{def:reg initial data}. Then, for a.e.~$t \in (0,T)$, there holds $\nabla_{\bxi}f \in L^\infty(t,T;L^2(\Upsilon))$, $\partial_t f, \nabla^2_{\bxi} f \in L^2(\Upsilon_{t,T})$, $\nabla_{\bxi}f \in L^4(\Upsilon_{t,T})$. 
\end{lemma}

The following formal argument illustrates why we expect this regularity to hold. As per the proof of Lemma \ref{lem:H2 for p}, the key idea is to exploit a cancellation that occurs when writing the equation in non-divergence form: by formally rewriting \eqref{eq:main eqn} in non-divergence form, we obtain \eqref{eq:main eqn non div form}, and notice that the terms $\nabla f \!\cdot\!\nabla \rho$ have cancelled. Testing with $\Delta f$, 
    \begin{equation*}
        \begin{aligned}
            \frac{1}{2}\frac{\der}{\der t}\int_\Upsilon |\nabla f|^2 \d \bxi + \int_\Upsilon (1-\rho)|\Delta f|^2 & \d \bxi + \int_\Upsilon |\nabla \partial_\theta f|^2 \d \bxi \\ 
            &= \int_\Upsilon \Delta f \dv((1-\rho)f \e(\theta)) \d \bxi -\int_\Upsilon f \Delta f \Delta \rho \d \bxi, 
        \end{aligned}
    \end{equation*}
    whence, arbitrarily fixing a common Lebesgue point $t_0 \in (0,T)$ of $f,\rho,\p$ and their derivatives, using the lower bound on $1-\rho$ provided by Proposition \ref{prop:lower bound 1-rho}, and integrating in time over the interval $[t_0,T]$, we get 
         \begin{equation*}
        \begin{aligned}
           \Vert \nabla f\Vert_{L^\infty(t_0,T;L^2(\Upsilon))}^2 +& \, \Vert  \Delta f\Vert_{L^2(\Upsilon_{t_0,T})}^2 + \Vert \nabla \partial_\theta f \Vert_{L^2(\Upsilon_{t_0,T})}^2 \\ 
            \leq C \Big( \Vert \nabla f (t_0,\cdot) &\Vert^2_{L^2(\Upsilon)} + \Vert \dv((1-\rho)f \e(\theta)) \Vert_{L^2(\Upsilon_{t_0,T})}^2 + \Vert f \Vert^2_{L^\infty(\Upsilon_{t_0,T})}\Vert \Delta \rho \Vert^2_{L^2(\Upsilon_{t_0,T})} \Big). 
        \end{aligned}
    \end{equation*}
    for some positive constant $C$ depending on $t_0$, where we used the result of Lemma \ref{lem:H1 for f}, and we recall $\Vert \Delta \rho \Vert_{L^2(\Upsilon_{t_0,T})}< \infty$ by virtue of Proposition \ref{lem:nabla rho in L8}. It follows that, for a.e.~$t \in (0,T)$, $\nabla f \in L^\infty(t,T;L^2(\Upsilon))$, $\Delta f \in L^2(\Upsilon_{t,T})$. Similarly, by testing with $\partial^2_\theta f$, 
    \begin{equation}\label{eq:formal h2 partial theta squared}
        \begin{aligned}
            \frac{1}{2}\frac{\der}{\der t}\int_\Upsilon |\partial_\theta & f|^2 \d \bxi + \int_\Upsilon |\partial^2_\theta f|^2 \d \bxi \\ 
            &= \int_\Upsilon (1-\rho)\Delta f \partial^2_\theta f \d \bxi + \int_\Upsilon f \Delta \rho \partial^2_\theta f \d \bxi - \int_\Upsilon \dv((1-\rho) f \e(\theta)) \partial^2_\theta f \d \bxi, 
        \end{aligned}
    \end{equation}
    whence Young's inequality and integrating yields $\partial^2_\theta f \in L^2(\Upsilon_{t,T})$, $\partial_\theta f \in L^\infty(t,T;L^2(\Upsilon))$.

    This argument is only formal, since \emph{a priori} we do not know that $\Delta f$ and $\partial^2_\theta f$ belong to $L^2$. In turn, we adapt the Galerkin approximation technique employed in the proof of Lemma \ref{lem:H2 for p} (\textit{cf.}~\S \ref{sec:proof of H2 via galerkin}) to this new setting. The only changes are that  we now have $\nabla^2 \rho \in L^4(\Omega_{t,T})$ by Proposition \ref{lem:nabla rho in L8}. 

\begin{proof}[Proof of Lemma \ref{lem:H2 for f}]

As per the approach in \S \ref{sec:proof of H2 via galerkin}, we recall $\{\varphi_j\}_j$ the orthonormal basis of $L^2_\per(\Omega)$ of eigenfunctions of the Laplacian, and define the Galerkin approximations 
\begin{equation*}
    f^n(t,x,\theta) := \sum^n_{j=1} \beta^n_j(t,\theta) \varphi_j(x), 
\end{equation*}
where the coefficients $\{\beta^n_j\}_{j=1,\dots,n}$ are determined from the linear system of PDEs 
\begin{equation}\label{eq:coeff f galerkin}
    \partial_t \beta^n_j - \partial^2_\theta \beta^n_j + \sum_{k=1}^n \beta^n_k\bigg( \int_\Omega \Big( (1\!-\!\rho) \nabla \varphi_k + \varphi_k\nabla\rho\Big) \!\cdot\! \nabla \varphi_j  \d x \bigg) = \int_\Omega  (1-\rho) f \e(\theta) \cdot \nabla \varphi_j \d x, 
\end{equation}
with corresponding initial data $\beta^n_j(0,\theta) = \langle f(t_0,\cdot,\theta) , \varphi_j \rangle, $ chosen such that $f^n$ solves 
\begin{equation}\label{eq:galerkin f}
    \left\lbrace\begin{aligned}
        &\partial_t f^n + \dv((1-\rho)f \e(\theta)) = \dv((1-\rho)\nabla f^n + f^n \nabla \rho ) + \partial^2_\theta f^n, \\ 
        &f^n(t_0,\cdot,\theta) = \sum_{j=1}^n \langle f(t_0,\cdot,\theta) , \varphi_j \rangle \varphi_j 
    \end{aligned}\right.
\end{equation}
in the weak sense dual to $\mathrm{span}\{\varphi_1,\dots,\varphi_n\}$. Notice that $f$, and not $f^n$, is on the right-hand side of \eqref{eq:coeff f galerkin}, for simplicity. The standard theory for such linear systems of uniformly parabolic PDEs implies the existence of $f^n \in L^2(t_0,T;H^1_\per(0,2\pi);C^\infty_\per(\Upsilon))$ solving the above. Alternatively, one can construct a solution of this problem by doing a further Galerkin approximation in the angle variable; \textit{cf.}~\textit{e.g.}~\cite[Appendix A]{bbes analysis}. 

Furthermore, by the orthonormality of the basis functions and standard results on bases, 
\begin{equation*}
    \Vert f^n(t_0,\cdot,\theta) \Vert^2_{L^2(\Omega)} = \sum_{j=1}^n \langle f(t_0,\cdot,\theta) , \varphi_j \rangle^2 \longrightarrow \Vert f(t_0,\cdot,\theta) \Vert^2_{L^2(\Omega)} \quad \text{as } n \to\infty, 
\end{equation*}
where the limit holds pointwise for all $\theta$. Additionally, by integrating with respect to $\theta$, $\Vert f^n(t_0,\cdot) \Vert^2_{L^2(\Upsilon)} = \int_0^{2\pi} \sum_{j=1}^n \langle f(t_0,\cdot,\theta) , \varphi_j \rangle^2 \d \theta,$ whence, since the integrands form an increasing sequence of nonnegative quantities, the Monotone Convergence Theorem implies 
\begin{equation*}
   \lim_{n\to\infty}\Vert f^n(t_0,\cdot) \Vert^2_{L^2(\Upsilon)} = \int_0^{2\pi} \lim_{n\to\infty} \sum_{j=1}^n \langle f(t_0,\cdot,\theta) , \varphi_j \rangle^2 \d \theta = \int_0^{2\pi} \Vert f(t_0,\cdot,\theta) \Vert^2_{L^2(\Omega)} \d \theta. 
\end{equation*}
An analogous argument, using similar manipulations to those used to obtain \eqref{eq:conv of initial data for galerkin}, proves the same convergence result for the space-angle derivatives. We have shown: 
\begin{equation}\label{eq:conv initial data galerkin f}
    \lim_{n\to\infty}\Vert f^n(t_0,\cdot) \Vert^2_{L^2(\Upsilon)} = \Vert f(t_0,\cdot) \Vert^2_{L^2(\Upsilon)}, \, \lim_{n\to\infty}\Vert \nabla_{\bxi} f^n(t_0,\cdot) \Vert^2_{L^2(\Upsilon)} = \Vert \nabla_{\bxi} f(t_0,\cdot) \Vert^2_{L^2(\Upsilon)}. 
\end{equation}

    \smallskip 

 \noindent 1. \textit{Consistency of the Galerkin approximations}: Since $f^n \in \mathrm{span}\{\varphi_1,\dots,\varphi_n\}$, we may insert it as a test function in the previous weak formulation. We get 
 \begin{equation*}
     \frac{1}{2}\frac{\der}{\der t}\!\!\int_\Upsilon\!\! |f^n|^2 \d \bxi +\!\! \int_\Upsilon\!\! \Big( (1-\rho)|\nabla f^n|^2 + |\partial_\theta f^n|^2 \Big) \d \bxi \!=\!\! \int_\Upsilon (1-\rho)f \e(\theta) \cdot \nabla f^n \d \bxi - \int_\Upsilon f^n \nabla f^n \cdot \nabla \rho \d \bxi. 
 \end{equation*}
 Recall $\nabla \rho \in L^{2+\delta}(t_0,T;L^\infty(\Omega))$ by Lemma \ref{cor:gehring grad p}. Using the lower bound of Proposition \ref{prop:lower bound 1-rho}, Young's inequality, and integrating in time over the interval $[t_0,t]$ for $0<t_0<t<T$, 
   \begin{equation}\label{eq:galerkin f h1 est}
     \begin{aligned}
            \Vert f^n(t,\cdot)\Vert_{L^2(\Upsilon)}^2 + \Vert \nabla_{\bxi} f^n \Vert_{L^2(\Upsilon_{t_0,t})}^2 \leq C \Big(  \Vert & f^n(t_0,\cdot)\Vert_{L^2(\Upsilon)}^2 + \Vert f  \Vert_{L^2(\Upsilon_{t_0,T})}^2 \\ 
            &+ 
 \int_{t_0}^t  \Vert\nabla \rho(s,\cdot)\Vert^2_{L^\infty(\Omega)} \Vert f^n(s,\cdot)\Vert_{L^2(\Upsilon)}^2 \d s \Big), 
     \end{aligned}
 \end{equation}
where $C$ is independent of $n$ and we emphasise that the integral on the right-hand side is well-defined by virtue of Lemma \ref{cor:gehring grad p}, $f^n \in L^\infty(t_0,T;L^2_\per(0,2\pi);C^\infty_\per(\Upsilon))$, and $\beta^n_j \in L^\infty(t_0,T;L^2_\per(0,2\pi))$ for each $n\in\mathbb{N}$; see also \eqref{eq:coeff f galerkin}. Gr\"onwall's Lemma then yields 
  \begin{equation*}
     \begin{aligned}
            \Vert f^n(t,\cdot)\Vert_{L^2(\Upsilon)}^2 \leq C \Big(  \Vert f^n(t_0,\cdot)\Vert_{L^2(\Upsilon)}^2 + \Vert f \Vert_{L^2(\Upsilon_{t_0,T})}^2 \Big) \exp\Big( C \int_{t_0}^t\Vert\nabla & \rho(s,\cdot)\Vert^2_{L^\infty(\Omega)} \Big). 
     \end{aligned}
 \end{equation*}
The convergence of the initial data \eqref{eq:conv initial data galerkin f} implies that $\sup_n \Vert f^n(t_0,\cdot) \Vert_{L^2(\Upsilon)} < \infty$, whence we deduce that $\sup_n \Vert f^n \Vert_{L^\infty(t_0,T;L^2(\Upsilon))} < \infty$,
 and, by returning to \eqref{eq:galerkin f h1 est}, we see that $\{\nabla_{\bxi}f^n \}_n$ is uniformly bounded in $L^2(\Upsilon_{t_0,T})$: 
 \begin{equation}\label{eq:10/3 unif galerkin f}
    \sup_n \Vert f^n \Vert_{L^\infty(t_0,T;L^2(\Upsilon))} + \sup_n \Vert \nabla_{\bxi}f^n \Vert_{L^2(\Upsilon_{t_0,T})} < \infty. 
 \end{equation}
 Using Alaoglu's Theorem, we obtain the weak convergence of the sequence $\{f^n\}_n$ to some limit $\bar{f} \in L^2(t_0,T;H^1_\per(\Upsilon)) \cap L^\infty(t_0,T;L^2(\Upsilon))$. An identical argument to the one presented in Step 3 of the proof of Proposition \ref{prop:galerkin} implies $\bar{f}=f$; the initial data is achieved in $L^2$. 

 \smallskip 

 \noindent 2. \textit{$H^2$-type estimate}: Since $f^n \in L^2(t_0,T;H^1_\per(0,2\pi);C^\infty_\per(\Upsilon))$, we may expand the spatial diffusions in non-divergence form in \eqref{eq:galerkin f}: $\partial_t f^n + \dv((1-\rho)f \e(\theta)) = (1-\rho)\Delta f^n + f^n \Delta \rho + \partial^2_\theta f^n.$ Relation \eqref{eq:relation for eig} implies $\Delta f^n$ is an admissible test function in \eqref{eq:galerkin f}. We get 
 \begin{equation*}
     \begin{aligned}
         \frac{1}{2}\frac{\der}{\der t}\int_\Upsilon |\nabla f^n |^2 \d \bxi +\! \int_\Upsilon (1-\rho)|\Delta f^n|^2 & \d \bxi + \!\int_\Upsilon |\nabla \partial_\theta f^n |^2 \d \bxi \\
         &= \!\int_\Upsilon \Delta f^n \dv((1-\rho)f \e(\theta)) \d \bxi - \int_\Upsilon f_n \Delta f_n \Delta \rho \d \bxi. 
     \end{aligned}
 \end{equation*}
Using the lower bound on $1-\rho$ from Proposition \ref{prop:lower bound 1-rho}, Young's inequality, the uniform estimate \eqref{eq:10/3 unif galerkin f} for $\{f^n\}_n$, and the integrability $\Delta \rho \in L^4(\Omega_{t_0,T})$  from \eqref{lem:nabla rho in L8}, we get by integrating in time  over the interval $[t_0,t]$ 
 \begin{equation}\label{eq:H2 for f step i}
     \begin{aligned}
        \Vert\nabla f^n(t) \Vert_{L^2(\Upsilon)}^2 +\! \Vert\Delta & f^n\Vert_{L^2(\Upsilon_{t_0,t})}^2 + \Vert\nabla \partial_\theta f^n \Vert_{L^2(\Upsilon_{t_0,t})}^2 \\
         \leq & \, C \int_{t_0}^t \int_0^{2\pi} \Vert f_n(s,\cdot,\theta) \Vert_{L^{4}(\Omega)}^2 \Vert \Delta \rho(s,\cdot) \Vert_{L^4(\Omega)}^2 \d\theta \d s \\ 
         &+ C\underbrace{\Big(\! \Vert\nabla f^n(t_0)\Vert_{L^2(\Upsilon)}^2 \!+\! \Vert \dv((1-\rho)f \e(\theta)) \Vert^2_{L^2(\Upsilon_{t_0,T})} \!\Big)}_{\leq C'}, 
     \end{aligned}
 \end{equation} 
 where the positive constants $C,C'$ depend on $t_0$ but are independent of $n,t$; recall that the convergence of the initial data \eqref{eq:conv initial data galerkin f} implies $\sup_n \Vert \nabla f^n(t_0,\cdot) \Vert_{L^2(\Upsilon)}<\infty$. Meanwhile, using the Sobolev embedding in only the $x$ variable: 
 \begin{equation*}
     \Vert f^n(s,\cdot,\theta) \Vert_{L^q(\Omega)} \leq C_q \Big( \Vert f^n(s,\cdot,\theta) \Vert_{L^2(\Omega)} + \Vert \nabla f^n(s,\cdot,\theta) \Vert_{L^2(\Omega)} \Big), 
 \end{equation*}
 for all $q \in [1,\infty)$, where the positive constant $C$ is independent of $s,\theta,n$; in particular, we choose $q=4$ in the above. By returning to \eqref{eq:H2 for f step i}, we obtain 
 \begin{equation*}
     \begin{aligned}
        \Vert\nabla f^n(t) \Vert_{L^2(\Upsilon)}^2 +\! \Vert\Delta & f^n\Vert_{L^2(\Upsilon_{t_0,t})}^2 + \Vert\nabla \partial_\theta f^n \Vert_{L^2(\Upsilon_{t_0,t})}^2 \\ 
        \leq & \, C + C \int_{t_0}^t \Vert \Delta \rho(s,\cdot) \Vert_{L^4(\Omega)}^2 \bigg( \int_0^{2\pi} \Vert f^n(s,\cdot,\theta) \Vert_{L^2(\Omega)}^2  \d\theta \bigg) \d s \\ 
        &\, \, \, \, \, \,  +  C \int_{t_0}^t \Vert \Delta \rho(s,\cdot) \Vert_{L^4(\Omega)}^2 \bigg( \int_0^{2\pi} \Vert \nabla f^n(s,\cdot,\theta) \Vert_{L^2(\Omega)}^2 \d\theta \bigg) \d s \\ 
        \leq & \, C + C \Vert f^n \Vert_{L^\infty(t_0,T;L^2(\Upsilon))}^2 \Vert \Delta \rho \Vert_{L^4(\Omega_{t_0,T})}^2 \\ 
        & \, \, \, \, \, \, + C \int_{t_0}^t  \Vert \Delta \rho(s,\cdot) \Vert_{L^4(\Omega)}^2 \Vert \nabla f^n(s,\cdot) \Vert_{L^2(\Upsilon)}^2 \d s, 
     \end{aligned}
 \end{equation*} 
 \textit{i.e.}, using the uniform-in-$n$ boundedness of $\Vert f^n \Vert_{L^\infty(t_0,T;L^2(\Upsilon))}$ provided by \eqref{eq:10/3 unif galerkin f}, there holds 
  \begin{equation}\label{eq:H2 for f step ii}
     \begin{aligned}
        \Vert\nabla f^n(t) \Vert_{L^2(\Upsilon)}^2 +\! \Vert\Delta  f^n\Vert_{L^2(\Upsilon_{t_0,t})}^2 &+ \Vert\nabla \partial_\theta f^n \Vert_{L^2(\Upsilon_{t_0,t})}^2 \\ 
        \leq & \, C \bigg( 1 +  \int_{t_0}^t  \Vert \Delta \rho(s,\cdot) \Vert_{L^4(\Omega)}^2 \Vert \nabla f^n(s,\cdot) \Vert_{L^2(\Upsilon)}^2 \d s \bigg), 
     \end{aligned}
 \end{equation} 
 where the positive constant $C$ depends on $t_0$ but is independent of $n,t$. The integrand on the right-hand side is integrable on account of $f^n \in L^\infty(t_0,T;L^2_\per(0,2\pi);C^\infty_\per(\Upsilon))$ and $\Vert \Delta \rho(s) \Vert_{L^4(\Omega)}^4 \in L^1(t_0,T)$ from Proposition \ref{lem:nabla rho in L8}. An application of the Gr\"onwall Lemma then implies 
 \begin{equation*}
     \Vert\nabla f^n \Vert_{L^\infty(t_0,T;L^2(\Upsilon))}^2 \leq C \exp \bigg( \int_{t_0}^T  \Vert \Delta \rho(s,\cdot) \Vert_{L^4(\Omega)}^2 \d s \bigg) \leq C \exp \Big ( T^{\frac{1}{2}} \Vert \Delta \rho \Vert_{L^4(\Omega_{t_0,T})}^2 \Big). 
 \end{equation*} 
 Returning to \eqref{eq:H2 for f step ii}, the right-hand side of the previous estimate implies
  \begin{equation*}
     \begin{aligned}
        \sup_n \Vert\nabla f^n \Vert_{L^\infty(t_0,T;L^2(\Upsilon))}^2 +\! \sup_n \Vert\Delta f^n\Vert_{L^2(\Upsilon_{t_0,t})}^2 + \sup_n \Vert\nabla \partial_\theta f^n \Vert_{L^2(\Upsilon_{t_0,t})}^2 < \infty, 
     \end{aligned}
 \end{equation*}
 and the lower semicontinuity of the norms imply $\nabla f \in L^\infty(t_0,T;L^2(\Upsilon))$, $\Delta f \in L^2(\Upsilon_{t_0,T}).$

 Similarly, we test \eqref{eq:galerkin f} with $\partial^2_\theta f^n$. Analogously to the formal computation \eqref{eq:formal h2 partial theta squared}, we get 
     \begin{equation*}
        \begin{aligned}
            \frac{1}{2}\frac{\der}{\der t}\int_\Upsilon |\partial_\theta & f^n|^2 \d \bxi + \int_\Upsilon |\partial^2_\theta f^n|^2 \d \bxi \\ 
            &= \int_\Upsilon (1-\rho)\Delta f^n \partial^2_\theta f^n \d \bxi + \int_\Upsilon f^n \Delta \rho \partial^2_\theta f^n \d \bxi - \int_\Upsilon \dv((1-\rho) f \e(\theta)) \partial^2_\theta f^n \d \bxi, 
        \end{aligned}
    \end{equation*}
    and Young's inequality yields $\sup_n \Vert\partial_\theta f^n \Vert_{L^\infty(t_0,T;L^2(\Upsilon))}^2 \!+\! \sup_n \Vert \partial^2_\theta f^n\Vert_{L^2(\Upsilon_{t_0,t})}^2  < \infty$. We deduce $\partial_\theta f \in L^\infty(t_0,T;L^2(\Upsilon))$ and $\partial^2_\theta f \in L^2(\Upsilon_{t_0,T})$.

    \smallskip 

    \noindent 3. \textit{Conclusion}: It follows from Steps 1 and 2 that there holds $\nabla_{\bxi} f \in L^\infty(t,T;L^2(\Upsilon))$ and $\Delta_{\bxi} f \in L^2(\Upsilon_{t,T})$ for a.e.~$t \in (0,T)$. An application of the periodic Calder\'on--Zygmund inequality of Lemma \ref{lem:CZ periodic} thereby yields $\nabla^2_{\bxi} f \in L^2(\Upsilon_{t,T})$, and it is then immediate from the equation \eqref{eq:main eqn non div form} that $\partial_t f \in L^2(\Upsilon_{t,T})$ a.e.~$t\in (0,T)$. Using the Gagliardo--Nirenberg inequality on the bounded domain $\Upsilon$, we get, for a.e.~$t \in (0,T)$, 
\begin{equation*}
    \Vert \nabla_{\bxi} f(t,\cdot) \Vert_{L^4(\Upsilon)} \leq C \Big( \Vert \nabla_{\bxi}^2 f(t,\cdot) \Vert_{L^2(\Upsilon)}^{\frac{1}{2}} \Vert f(t,\cdot) \Vert_{L^\infty(\Upsilon)}^{\frac{1}{2}} + \Vert f(t,\cdot) \Vert_{L^1(\Upsilon)} \Big), 
\end{equation*}
whence, by raising the whole line to the power $4$ and then integrating in time and using the boundedness of $\rho$ for the final term on the right-hand side, there holds 
\begin{equation}\label{eq:L4 grad f using GN}
    \Vert \nabla_{\bxi} f \Vert_{L^4((t_0,T)\times\Upsilon)}^4 \leq C \Big( \Vert \nabla_{\bxi}^2 f \Vert_{L^2((t_0,T)\times\Upsilon)}^{2} \Vert f \Vert_{L^\infty((t_0,T)\times\Upsilon)}^{2} + T|\Omega|^4 \Big)  < \infty, 
\end{equation}
and the assertion of the lemma follows. 
\end{proof}

\begin{proof}[Proof of Proposition \ref{lem:all space derivs}]

  The result of Lemma \ref{lem:H2 for f} implies that we may rigorously expand the diffusive term of equation \eqref{eq:main eqn} in the non-divergence form \eqref{eq:main eqn non div form}, which
    holds in the weak sense dual to test functions belonging to $L^2(\Upsilon_{t,T})$. Furthermore, it follows from Lemma \ref{lem:H2 for f} and Propositions \ref{lem:nabla rho in L8} and \ref{lem:Linfty for f} that the right-hand side of \eqref{eq:main eqn} belongs to $L^4(\Upsilon_{t,T})$ a.e.~$t \in (0,T)$. Hence, recalling the definition of the matrix $A$ from \eqref{eq:A def matrix}, 
    \begin{equation}\label{eq:avant lieberman L4}
        \partial_t f - A:\nabla^2_{\bxi} f \in L^4(\Upsilon_{t,T}) \qquad \text{a.e.~}t \in (0,T), 
    \end{equation}
    and recall that $A$ satisfies the coercivity assumptions \eqref{eq:coercivity for A}. Furthermore, since $\rho$ is H\"older continuous (\textit{cf.}~Proposition \ref{lem:nabla rho in L8}) the matrix $A$ is endowed with the same modulus of continuity in time and space. Applying the Schauder-type result of Lemma \ref{lem:Lp schauder}, we obtain 
\begin{equation*}
    \partial_t f , \nabla^2_{\bxi} f \in L^4(\Upsilon_{t,T}) \qquad \text{a.e.~}t\in(0,T), 
\end{equation*}
whence, using Proposition \ref{lem:Linfty for f}, the Gagliardo--Nirenberg inequality yields, for a.e.~$t \in (0,T)$, 
\begin{equation}\label{eq:grad f in L8}
  \begin{aligned}  \Vert \nabla_{\bxi} f(t,\cdot) \Vert_{L^8(\Upsilon)}^8 \leq C \Big( \Vert \nabla_{\bxi}^2 f(t,\cdot) \Vert_{L^4(\Upsilon)}^{4} \Vert f(t,\cdot) \Vert_{L^\infty(\Upsilon)}^{4} + {\Vert f(t,\cdot) \Vert_{L^1(\Upsilon)}}^4 \Big), 
  \end{aligned}
\end{equation}
and hence, since $\Vert f \Vert_{L^1} \leq \Vert \rho \Vert_{L^1}$, by integrating in time we recover $\nabla_{\bxi}f \in L^8(\Upsilon_{t,T})$. 

\smallskip

\noindent 2. \textit{Induction}: By iterating the previous step we show that, for all $m\in\mathbb{N}$, 
\begin{equation}\label{eq:for induction proof just space}
    \partial_t f, \nabla^2_{\bxi} f \in L^{2^m}(\Upsilon_{t,T}), \quad   \nabla_{\bxi} f \in L^{2^{m+1}}(\Upsilon_{t,T}) , \qquad \text{a.e.~}t\in(0,T). 
\end{equation}
We prove this by induction; Step 1 of the proof showed that the result is true for $m=1,2$. Suppose that the result holds for $m\in\mathbb{N}$, and consider the case $m+1$. Eq.~\eqref{eq:main eqn non div form} implies 
\begin{equation*}
    \partial_t f - A:\nabla^2_{\bxi} f \in L^{2^{m+1}}(\Upsilon_{t,T}) \qquad \text{a.e.~}t\in(0,T). 
\end{equation*}
In turn, the Schauder-type result Lemma \ref{lem:Lp schauder} implies $\partial_t f , \nabla^2_{\bxi} f \in L^{2^{m+1}}(\Upsilon_{t,T})$, and by applying the Gagliardo--Nirenberg inequality as per \eqref{eq:grad f in L8}, we get 
\begin{equation*}
    \Vert \nabla_{\bxi} f \Vert_{L^{2^{m+2}}(\Upsilon_{t,T})} \leq C \Big( \Vert \nabla_{\bxi}^2 f \Vert_{L^{2^{m+1}}(\Upsilon_{t,T})}^{2^{m+1}} \Vert f \Vert_{L^\infty(\Upsilon_{t,T})}^{2^{m+1}} + 1 \Big). 
\end{equation*}
This shows that \eqref{eq:for induction proof just space} holds for the case $m+1$. By induction, \eqref{eq:for induction proof just space} holds for all $m\in\mathbb{N}$. The result follows from \eqref{eq:for induction proof just space} by interpolation using H\"older's inequality over $\Upsilon_{t,T}$. 
\end{proof}

\subsection{Bootstrap for higher-order time derivatives and Proof of Theorem \ref{thm:smooth}}

We take successive time derivatives in \eqref{eq:main eqn} and obtain adequate estimates to show that all derivatives of $f$ are continuous away from the initial time, thereby proving Theorem \ref{thm:smooth}. The main result of this section is the following, for which we introduce the notations, for all $n\in\mathbb{N}$, 
\begin{equation*}
    f^{(n)} := \partial_t^n f, \qquad \dot{f}^{(n)} = f^{(n+1)}. 
\end{equation*}

\begin{prop}[Higher integrability for $f^{(n)}$]\label{prop:bounded fn all time deriv}
  Let $f$ be a weak solution of \eqref{eq:main eqn} with initial data satisfying Definition \ref{def:reg initial data}. Then, for all $q \in [1,\infty),n\in\mathbb{N}\cup\{0\}$, a.e.~$t \in (0,T)$ 
   \begin{equation}\label{eq:integ all time deriv}
      \partial_t f^{(n)} , \nabla^2_{\bxi} f^{(n)} \in L^q(\Upsilon_{t,T}), 
   \end{equation}
   and $f^{(n)}$ satisfies, in the weak sense, the equation 
       \begin{equation}\label{eq:general nth time deriv eqn}
    \begin{aligned}
    \de_t \ff{n} - \dv\Big[ & (1-\rho)\nabla\ff{n} + \ff{n}\nabla \rho \Big]  - \partial^2_\theta \ff{n} \\ 
    = -\dv & \Big[
    \big( (1- \rho) \ff{n} -\rr{n} f
    \big) \,\e(\theta)
    \Big] + \dv\Big[  -\rr{n}\nabla f + f\nabla \rr{n}  \Big]  + \sum_{k=1}^{n-1} \dv E^n_k,  
\end{aligned}
\end{equation}
where, for all $n\in\mathbb{N}$, the terms $\{E^n_k\}_{k=1}^{n-1}$ are given by 
\begin{equation}\label{eq:Enk def}
    E^n_k := \binom{n}{k} \Big[ \ff{n-k}\nabla\rr{k} - \rr{n-k}\nabla\ff{k}   - \rr{n-k} \ff{k} \e(\theta) \Big]. 
\end{equation}
\end{prop}

Before giving the proof of Proposition \ref{prop:bounded fn all time deriv}, we apply it to prove Theorem \ref{thm:smooth}.

\begin{proof}[Proof of Theorem \ref{thm:smooth}]
    Using the integrability \eqref{eq:integ all time deriv} and Morrey's embedding, we obtain that $f^{(n)}$ is continuous in time and space-angle for all $n\in\mathbb{N}$. Furthermore, using the equation \eqref{eq:general nth time deriv eqn}, the lower bound on $1-\rho$ provided by Proposition \ref{prop:lower bound 1-rho}, and the Dominated Convergence Theorem, we apply successive space-angle derivatives in \eqref{eq:general nth time deriv eqn} and deduce $\nabla_{\bxi}^{(m)} f^{(n)} \in W^{1,q}(\Upsilon_{t,T})$ {a.e.~}$t\in(0,T)$, for all $q \in [1,\infty)$, $m,n\in\mathbb{N}$, whence another application of Morrey's embedding implies that all such quantities are continuous. 
\end{proof}

The proof of Proposition \ref{prop:bounded fn all time deriv} is divided into several steps, which constitute the following subsections. To begin with, we introduce the notations for the difference quotients with respect to the time variable. For clarity, we first cover the case $n=1$ in detail. 

\subsubsection{Notations for difference quotients in time}\label{subsec:time diff quotients} Fix $\delta>0$ arbitrarily and extend $f$ smoothly to zero outside of $[0,T]$ to the larger time-interval $(-\delta,T+\delta)$; this extends the weak formulation of \eqref{eq:main eqn} to the larger time-interval and preserves for all $q \in [1,\infty)$, 
\begin{equation}\label{eq:diff quot basic with C}
    \Vert f \Vert_{L^q((-\delta,T+\delta)\times\Upsilon)} \leq C_q \Vert f \Vert_{L^q(\Upsilon_T)}, \qquad \Vert \dot{f} \Vert_{L^q((-\delta,T+\delta)\times\Upsilon)} \leq C_q \Vert \dot{f} \Vert_{L^q(\Upsilon_T)}. 
\end{equation}
We define, for $0<|h|<\delta$ and a.e.~$(t,\bxi) \in (0,T)\times\Upsilon$, the difference quotients in time: 
\begin{equation*}
    D_h f(t,\bxi) := \frac{f(t+h,\bxi) - f(t,\bxi)}{h}, 
\end{equation*}
and $D_h \rho(t,x) := [\rho(t+h,x) - \rho(t,x)]/{h}$, $D_h \p(t,x) := [\p(t+h,x) - \p(t,x)]/{h}$. We have $|D_h \rho(t,\x)| , |D_h \p(t,\x)| \leq \int_0^{2\pi}|D_h f(t,\x,\theta)| \d \theta$ and by Jensen inequality, for all $q \in [1,\infty)$, 
\begin{equation}\label{eq:diff quotient rho bound f}
    \Vert D_h \rho(t,\cdot) \Vert_{L^q(\Omega)} ,  \Vert D_h \p(t,\cdot) \Vert_{L^q(\Omega)} \leq 2\pi \Vert D_h f(t,\cdot) \Vert_{L^q(\Upsilon)}. 
\end{equation} 
From here onwards, we employ the notation $\tau_h f = f(t+h,\cdot)$ for the translates; the estimate on $\rho$ implies $0 \leq \tau_h \rho \leq 1$ a.e.~and $1 \geq 1- \tau_h \rho \geq c(t)$ a.e.~in $\Omega_{t,T}$ from Proposition \ref{prop:lower bound 1-rho}. 

\subsubsection{Deriving the equation for $\dot{f}$}\label{sec:deriving eqn for dot f}

To begin with, we show how to introduce a time derivative into the equation \eqref{eq:main eqn} by means of difference quotients. In order to do this, we must first show integrability on $\nabla \dot{\rho}$, which will be obtained by introducing a time derivative in the equation \eqref{eq:rho eq} for $\rho$. This is the content of the next lemma. 

\begin{lemma}[Equation for $\dot{\rho}$]\label{lem:eqn for dot rho}
    Let $f$ be a weak solution of \eqref{eq:main eqn} with initial data satisfying the assumptions of Theorem \ref{thm:smooth}, and $\rho$ as per \eqref{eq:rho def}. Then, for a.e.~$t \in (0,T)$, there holds $\dot{\rho} \in L^\infty(t,T;L^2(\Omega))$, $\nabla \dot{\rho} \in L^2(\Omega_{t,T})$, and, in the weak sense dual to $L^2(t,T;H^1_\per(\Omega))$, 
    \begin{equation}\label{eq:eqn for dot rho}
        \partial_t \dot{\rho} + \dv\big[ (1-\rho)\dot{\p} - \dot{\rho} \p \big] = \Delta \dot{\rho}. 
    \end{equation}
\end{lemma}
\begin{proof}
    The main idea is as follows: provided \eqref{eq:eqn for dot rho} holds, the integrability obtained in Proposition \ref{lem:all space derivs} implies that the classical $H^1$-type estimate for $\dot{\rho}$ (obtained formally by testing \eqref{eq:eqn for dot rho} against $\dot{\rho}$ itself) closes. However, this approach is not rigorous, since we do not know that \eqref{eq:eqn for dot rho} holds, whence we employ the method of difference quotients.

    \smallskip 

    \noindent 1. \textit{$H^1$-type estimate for $D_h \rho$}: We recall the notations introduced in \ref{subsec:time diff quotients}. Direct calculation shows that the equation for $D_h \rho$ reads: 
    \begin{equation}\label{eq:diff quot eq for Dh rho}
        \partial_t D_h \rho + \dv\big[ (1-\tau_h\rho)D_h \p - \p D_h \rho \big] = \Delta D_h \rho, 
    \end{equation}
   in the weak sense dual to $L^2(0,T;H^1_\per(\Omega))$. Observe that $D_h \rho \in L^2(t,T;H^1_\per(\Omega))$ a.e.~$t \in (0,T)$, whence we may insert it into the weak formulation of the previous equation: 
    \begin{equation*}
        \begin{aligned}
            \frac{1}{2}\frac{\der}{\der t}\int_\Omega |D_h \rho |^2 \d x + \int_\Omega |\nabla D_h \rho |^2 \d x = \int_\Omega (1-\tau_h \rho) D_h \p \cdot \nabla D_h \rho \d x - \int_\Omega D_h \rho \p \cdot \nabla D_h \rho \d x. 
        \end{aligned}
    \end{equation*}
    Applying Young's inequality, we get 
        \begin{equation*}
        \begin{aligned}
            \frac{1}{2}\frac{\der}{\der t}\int_\Omega |D_h \rho |^2 \d x + \frac{1}{2}\int_\Omega |\nabla D_h \rho |^2 \d x \leq  \int_\Omega |D_h \p|^2 \d x + \int_\Omega |D_h \rho|^2 \d x \leq C\Vert \partial_t f(t,\cdot) \Vert^2_{L^2(\Omega)}, 
        \end{aligned}
    \end{equation*}
  for some positive constant $C$ independent of $h$, where we used \eqref{eq:diff quot basic with C} and the boundedness of Proposition \ref{lem:all space derivs}. By integrating in time, we get, for a.e.~$t\in(0,T)$, 
          \begin{equation*}
        \begin{aligned}
          \Vert D_h \rho \Vert^2_{L^\infty(t,T;L^2(\Omega))} + \Vert \nabla D_h \rho \Vert^2_{L^2(\Omega_{t,T})}&\leq  C\Big( \Vert D_h \rho (t,\cdot) \Vert^2_{L^2(\Omega)} + \Vert \partial_t f \Vert^2_{L^2(\Omega_{t,T})} \Big) \\ 
          &\leq  C\Big( \Vert \partial_t \rho (t,\cdot) \Vert^2_{L^2(\Omega)} + \Vert \partial_t f \Vert^2_{L^2(\Omega_{t,T})} \Big), 
        \end{aligned}
    \end{equation*}
    where we used \eqref{eq:diff quot basic with C} again. By letting $h \to 0$, we obtain the required boundedness of $\dot{\rho}$ and $\nabla \dot{\rho}$; note that $\partial_t \rho(t,\cdot) \in L^2(\Omega)$ a.e.~$t$ by Proposition \ref{lem:nabla rho in L8}. 

    \smallskip 

    \noindent 2. \textit{Equation for $\dot{\rho}$}: Now that the boundedness of $\nabla\dot{\rho}$ has been established, by testing \eqref{eq:diff quot eq for Dh rho} with a test function belonging to $L^2(t,T;H^1_\per(\Omega))$ and using the Dominated Convergence Theorem to pass to the limit $h \to 0$, noting that $D_h \rho \to \partial_t \rho $ and $D_h \p \to \partial_t \p$ a.e.~in $\Omega_{t,T}$, we obtain the weak formulation of \eqref{eq:eqn for dot f} as desired. The proof is complete. 
\end{proof}

We are now in a position to derive the equation for $\dot{f}$ rigourously. 

\begin{lemma}[Equation for $\dot{f}$]\label{lem:eqn for dot f}
   Let $f$ be a weak solution of \eqref{eq:main eqn} with initial data satisfying the assumptions of Theorem \ref{thm:smooth}. Then, for a.e.~$t \in (0,T)$, there holds $\dot{f} \in L^\infty(t,T;L^2(\Upsilon))$ and $\nabla_{\bxi} \dot{f} \in L^2(\Upsilon_{t,T})$, and, in the weak sense dual to $L^2(t,T;H^1_\per(\Upsilon))$, 
    \begin{equation}\label{eq:eqn for dot f}
        \partial_t \dot{f} + \dv\big[\big((1-\rho)\dot{f} - \dot{\rho} f \big)\e(\theta) \big] = \dv\big[ (1-\rho)\nabla \dot{f} - \dot{\rho} \nabla f + \dot{f} \nabla \rho + f \nabla\dot{\rho}  \big] + \partial^2_\theta \dot{f}. 
    \end{equation}
\end{lemma}
\begin{proof}
The underlying strategy of proof rests on observing that, assuming that \eqref{eq:eqn for dot f} holds, we are able to close the $H^1$-type estimate on $\dot{f}$ by formally testing \eqref{eq:eqn for dot f} with $\dot{f}$ itself; this is because of the integrability obtained in Proposition \ref{lem:all space derivs} and Lemma \ref{lem:eqn for dot rho}. This argument is merely formal since we have not rigorously shown that \eqref{eq:eqn for dot f} holds. We therefore proceed to a detailed proof employing the method of difference quotients.

\smallskip 

\noindent 1. \textit{$H^1$-type estimate for $D_h f$}: By direct computation, we get, for a.e.~$t \in (0,T)$, 
\begin{equation}\label{eq:time diff quot eqn i}
 \begin{aligned} \partial_t D_h f +   \dv\big[ & \big(D_h f (1-\tau_h \rho) - f D_h \rho \big)\e(\theta) \big] \\ 
 &= \dv\big[ (1-\tau_h \rho) \nabla D_h f - D_h \rho \nabla f + \tau_h f \nabla D_h \rho + D_h f \nabla \rho \big] + \partial^2_\theta D_h f, 
 \end{aligned}
\end{equation}
in the sense dual to $L^2(t,T;H^1_\per(\Upsilon))$ a.e.~$t \in (0,T)$. Since $D_h f \in L^2(0,T;H^1_\per(\Upsilon))$, it is an admissible test function in the weak formulation of \eqref{eq:main eqn}. We get, for a.e.~$0<s<t <T$, 
\begin{equation*}
    \begin{aligned}
        \frac{1}{2}\frac{\der}{\der t}\int_\Upsilon |D_h f|^2 \d \bxi +\int_\Upsilon & \Big( (1-\tau_h \rho) |\nabla D_h f|^2 + |\partial_\theta D_h f|^2 \Big) \d \bxi   \\ 
        =& \,   \int_\Upsilon (D_h f (1-\tau_h \rho) - f D_h \rho) \e(\theta) \cdot \nabla D_h f \d \bxi + \int_\Upsilon D_h \rho \nabla f \cdot \nabla D_h f \d \bxi \\ 
            &  - \int_\Upsilon  \tau_h f \nabla D_h \rho \cdot \nabla D_h f \d \bxi   + \int_\Upsilon D_h f \nabla \rho \cdot \nabla D_h f \d \bxi           \\ 
        \leq& \, \frac{1}{2}\int_\Upsilon |\nabla D_h f|^2 \d \bxi + \frac{1}{2}\big( 1 + \Vert f \Vert_{L^\infty(\Upsilon_{s,T})}^2 \big) \int_\Upsilon |D_h f|^2 \d \bxi, 
    \end{aligned}
\end{equation*}
whence, using the lower bound of Proposition \ref{prop:lower bound 1-rho} and properties of difference quotients, 
\begin{equation*}
    \begin{aligned}
        \frac{\der}{\der t} \Vert D_h f(t,&\cdot) \Vert_{L^2(\Upsilon)}^2 + \Vert \nabla_{\bxi} D_h f(t,\cdot) \Vert_{L^2(\Upsilon)}^2   \\ 
        \leq& \,C\Big( \Vert \dot{f}(t,\cdot) \Vert_{L^2(\Upsilon)}^2 \! + \Vert f \Vert_{L^\infty(\Upsilon_{s,T})}^2 \Vert \dot{\rho}(t,\cdot)\Vert_{L^2(\Upsilon)}^2 \! + \Vert \dot{\rho}(t,\cdot) \Vert^2_{L^4(\Upsilon)} \Vert \nabla f(t,\cdot) \Vert^2_{L^4(\Upsilon)} \\ 
        & \,\,\,\,\,\,\,\,\,\,\,\,\,\,\,\,\,\,\,\,\,\, \,\,\,\,\,\,\,\,\,\,\,\,\,\,\,\,\, + \Vert f \Vert_{L^\infty(\Upsilon_{s,T})}^2 \Vert \nabla \dot{\rho}(t,\cdot) \Vert^2_{L^2(\Upsilon)} + \Vert \dot{f}(t,\cdot) \Vert^2_{L^4(\Upsilon)}\Vert \nabla \rho(t,\cdot) \Vert^2_{L^4(\Upsilon)} \Big), 
    \end{aligned}
\end{equation*}
where $C$ is independent of $h$. Using Young's inequality and then integrating in time, 
\begin{equation*}
    \begin{aligned}
        \Vert D_h f& \Vert_{L^\infty(s,T;L^2(\Upsilon))}^2 + \Vert \nabla_{\bxi} D_h f \Vert_{L^2(\Upsilon_{s,T})}^2   \\ 
        \leq& \,C\Big( \Vert \dot{f}(s,\cdot) \Vert_{L^2(\Upsilon)}^2 + \Vert \dot{f} \Vert_{L^2(\Upsilon_{s,T})}^2  + \Vert f \Vert_{L^\infty(\Upsilon_{s,T})}^2 \Vert \dot{\rho}\Vert_{L^2(\Upsilon_{s,T})}^2  +  \Vert \dot{\rho} \Vert^4_{L^4(\Upsilon_{s,T})} \\ 
        & \,\,\,\,\,\,\,\,\,\,\,\,  +\Vert \nabla f \Vert^4_{L^4(\Upsilon_{s,T})} + \Vert f \Vert_{L^\infty(\Upsilon_{s,T})}^2 \Vert \nabla \dot{\rho}\Vert^2_{L^2(\Upsilon_{s,T})} + \Vert \dot{f}\Vert^4_{L^4(\Upsilon_{s,T})} + \Vert \nabla \rho\Vert^4_{L^4(\Upsilon_{s,T})} \Big), 
    \end{aligned}
\end{equation*}
and notice that the right-hand side is bounded independently of $h$ by Lemma \ref{lem:eqn for dot rho} and Proposition \ref{lem:all space derivs}. By letting $h \to 0$, we obtain $\dot{f} \in L^\infty(s,T;L^2(\Upsilon))$ and $\nabla_{\bxi} \dot{f} \in L^2(\Upsilon_{s,T})$.

\smallskip 

\noindent 2. \textit{Equation for $\dot{f}$}: As per the proof of Lemma \ref{lem:eqn for dot rho}, by testing \eqref{eq:time diff quot eqn i} against a test function belonging to $L^2(t,T;H^1_\per(\Upsilon))$, using the Dominated Convergence Theorem and the integrability from Step 1, we let $h\to 0$ in formulation \eqref{eq:time diff quot eqn i} and get \eqref{eq:eqn for dot f}. 
\end{proof}

Using the previous result, by testing \eqref{eq:eqn for dot f} with $\e(\theta)$, which is admissible, we obtain the equation for $\dot{\p} \in L^\infty(t,T;L^2_\per(\Omega)) \cap L^2(t,T;H^1_\per(\Omega))$ a.e.~$t\in(0,T)$: 
\begin{equation}\label{eq:dot p eqn}
    \partial_t \dot{\p} + \dv\big[ (1-\rho)\dot{\P} - \dot{\rho} \P  \big] = \dv \big[  (1-\rho) \nabla \dot{\p} - \dot{\rho} \nabla \p + \dot{\p}\otimes \nabla \rho + \p \otimes \nabla \dot{\rho}   \big] - \dot{\p}, 
\end{equation}
where $\dot{\P} \in L^\infty(t,T;L^2_\per(\Omega)) \cap L^2(t,T;H^1_\per(\Omega))$, since $\dot{f}$ belongs to this same space.

\subsubsection{Improved integrability for $\nabla\dot{\rho}$}

In order to obtain $\dot{f}$ bounded in $L^\infty$ (away from the initial time) by the method of De Giorgi presented in the proof of Proposition \ref{lem:Linfty for f}, it is necessary to first obtain sufficient estimates for $\nabla \dot{\rho}$. For this reason, we begin by performing the $H^2$-type estimate for $\dot{\rho}$. Ultimately, through a series of steps analogous to those presented in \S \ref{subsec:angle averaged quantities}, we show that $\nabla \dot{\rho}$ is sufficiently integrable to perform the De Giorgi type iterations on $\dot{f}$.

\begin{lemma}[$H^2$-type estimate for $\dot{\rho}$]\label{lem:H2 dot rho}
   Let $f$ be a weak solution of \eqref{eq:main eqn} with initial data satisfying the assumptions of Theorem \ref{thm:smooth}, and $\rho$ as per \eqref{eq:rho def}. Then, for a.e.~$t \in (0,T)$, there holds 
    \begin{equation*}
        \nabla \dot{\rho} \in L^\infty(t,T;L^2(\Omega)), \quad \nabla^2 \dot{\rho} \in L^2(\Omega_{t,T}), \quad \nabla \dot{\rho} \in L^4(\Omega_{t,T}). 
    \end{equation*}
\end{lemma}
\begin{proof}
    Armed with the boundedness $\nabla_{\bxi} \dot{f} \in L^2(\Upsilon_{t,T})$ a.e.~$t \in (0,T)$ from Lemma \ref{lem:eqn for dot f}, we return to the equation \eqref{eq:eqn for dot rho} and observe that $\partial_t \dot{\rho} - \Delta \dot{\rho} \in L^2(\Omega_{t,T})$, whence the classical Calder\'on--Zygmund estimate for the heat equation implies $\partial_t \dot{\rho} , \nabla^2 \dot{\rho} \in L^2(\Omega_{t,T})$. As such, we may test \eqref{eq:eqn for dot rho} with $\Delta \dot{\rho}$, and, using also Young's inequality, we obtain for a.e.~$t\in(0,T)$ 
    \begin{equation*}
      \begin{aligned}  \frac{1}{2}\frac{\der}{\der t} \Vert\nabla \dot{\rho}(t,\cdot)\Vert_{L^2(\Omega)}^2 + \frac{1}{2}\Vert \Delta \dot{\rho}(t,\cdot)\Vert_{L^2(\Omega)}^2 \leq & \, \Vert \nabla \dot{\p}(t,\cdot)\Vert_{L^2(\Omega)}^2 + \Vert \nabla \rho(t,\cdot)\Vert^2_{L^4(\Omega)}\Vert \dot{\p}(t,\cdot) \Vert^2_{L^4(\Omega)} \\ 
      &+ \Vert\nabla \dot{\rho}(t,\cdot)\Vert_{L^2(\Omega)}^2 + \Vert \dot{\rho}(t,\cdot) \Vert^2_{L^4(\Omega)}\Vert \nabla \p(t,\cdot)\Vert^2_{L^4(\Omega)}; 
      \end{aligned}
    \end{equation*}
    the right-hand side is bounded by virtue of Proposition \ref{lem:all space derivs}. Applying Young's inequality again and then integrating with respect to the time variable yields $\nabla \dot{\rho} \in L^\infty(t,T;L^2(\Omega))$ a.e.~$t \in (0,T)$. Finally, an application of Lemma \ref{lem:dibenedetto classic} implies the boundedness of $\nabla \dot{\rho}$. 
\end{proof}

We upgrade the integrability of $\partial_t \dot{\rho}$ and $\nabla^2\dot{\rho}$ by performing the $H^2$-type estimate for $\dot{\p}$. 

\begin{lemma}[$H^2$-type estimate for $\dot{\p}$]\label{lem:H2 dot p}
    Let $f$ be a weak solution of \eqref{eq:main eqn} with initial data satisfying the assumptions of Theorem \ref{thm:smooth}, and $\rho,\p$ as per \eqref{eq:rho def}--\eqref{eq:polarisation def}. Then, for a.e.~$t \in (0,T)$ there holds 
    \begin{equation*}
        \nabla \dot{\p} \in {L^\infty(t,T;L^2(\Omega))} , \quad \nabla^2 \dot{\p} \in {L^2(\Omega_{t,T})} , \quad \nabla \dot{\p}, \partial_t \dot{\rho},\nabla^2\dot{\rho} \in  L^4(\Omega_{t,T}), \quad \nabla \dot{\rho} \in L^8(\Omega_{t,T}). 
    \end{equation*}
\end{lemma}

\begin{proof}
    We rewrite the equation \eqref{eq:dot p eqn} in coordinate form with $\dot{\p} = (\dot{p}_i)_i$, in the form 
    \begin{equation*}
    \partial_t \dot{p} - \dv \big[ {(1-\rho) \nabla \dot{p} +  \dot{p} \nabla \rho} \big] = \underbrace{- \dv\big[  (1-\rho)\dot{\tilde{\P}} - \dot{\rho} \tilde{\P}   \big]}_{\in L^2} + \underbrace{\dv \big[  - \dot{\rho} \nabla p + p \nabla \dot{\rho}    \big]}_{\in L^2} - \dot{p}, 
\end{equation*}
where we used the integrability from Lemmas \ref{lem:eqn for dot f} and \ref{lem:H2 dot rho} and recall the definition of $\tilde{\P}$ from \eqref{eq:tilde P def}; recall from Lemma \ref{lem:eqn for dot f} that $\nabla_{\bxi} \dot{f} \in L^2(\Upsilon_{t,T})$, whence $\nabla \dot{\P} \in L^2(\Omega_{t,T})$. As per the proof of Lemma \ref{lem:H2 for p}, we construct Galerkin approximations $\dot{p}^n$ that solve the linear problems: 
\begin{equation}
    \left\lbrace \begin{aligned}
        &\partial_t \dot{p}^n \!-\! \dv\! \big[ (1\!-\!\rho) \nabla \dot{p}^n \!+\! \dot{p}^n \nabla \rho   \big] \! +\! \dot{p}^n \!=\! - \dv\!\big[ (1\!-\!\rho)\dot{\tilde{\P}} - \dot{\rho} \tilde{\P}  \big]\! +\! \dv \!\big[ \!-\! \dot{\rho} \nabla p \!+\! p \nabla \dot{\rho}   \big], \\ 
        &\dot{p}^n(t_0,\cdot) = \sum_{j=1}^n \langle \dot{p}(t_0,\cdot),\varphi_j \rangle \varphi_j, 
    \end{aligned}  \right. 
\end{equation}
in the weak sense dual to $\span\{\varphi_1,\dots,\varphi_n\}$ the chosen basis functions (\textit{cf.}~\S \ref{sec:proof of H2 via galerkin}). By following the proof of Lemma \ref{lem:H2 for p} to the letter, we obtain the existence of $\dot{p}^n$ solving the above, uniformly bounded in $L^\infty(t,T;L^2(\Omega)) \cap L^2(t,T;H^1(\Omega))$, and converging weakly in $L^2(t,T;H^1(\Omega))$ to $\dot{p}$. As a result, the interpolation  Lemma \ref{lem:dibenedetto classic} implies that, for a.e.~$t \in (0,T)$, 
\begin{equation}\label{eq:dot p n galerkin unif bounded Lq}
    \sup_n \Vert \dot{p}^n \Vert_{L^4(t,T)} + \sup_n \Vert \dot{p}^n \Vert_{L^\infty(t,T;^2(\Omega))} + \sup_n \Vert \nabla \dot{p}^n \Vert_{L^2(\Omega_{t,T})} < \infty.
\end{equation}

We perform the second derivative estimate, by testing with $\Delta \dot{p}^n \in \span\{\varphi_1,\dots,\varphi_n\}$ and using the non-divergence form of the diffusive term, and obtain 
\begin{equation*}
   \begin{aligned} 
        \frac{1}{2}\frac{\der}{\der t}
        &
        \Vert\nabla \dot{p}^n(t,\cdot)\Vert_{L^2(\Omega)}^2 
        + \Vert \Delta \dot{p}^n(t,\cdot) \Vert_{L^2(\Omega)}^2 +\Vert \nabla \dot{p}^n(t,\cdot) \Vert_{L^2(\Omega)}^2 \\ 
        \leq &
        C \Big( \Vert \dot{p}^n(t,\cdot) \Vert_{L^2(\Omega)}^2 + \Vert \dv[ (1-\rho)\dot{\tilde{\P}} - \dot{\rho} \tilde{\P} ]\Vert^2_{L^2(\Omega)} 
        \\  &
        + \Vert \dv [ - \dot{\rho} \nabla p + p \nabla \dot{\rho}   ]\Vert^2_{L^2(\Omega)} 
        + \Vert \dot{p}^n \Vert_{L^4(\Omega)}^4 + \Vert \Delta \rho \Vert^4_{L^4(\Omega)} \Big), 
   \end{aligned}
\end{equation*}
using the lower bound of Proposition \ref{prop:lower bound 1-rho}. The third term on the right is bounded since 
\begin{equation*}
    \begin{aligned}
        \Vert \dv [ - \dot{\rho} \nabla p + p \nabla \dot{\rho}   ]\Vert^2_{L^2(\Omega)} \leq 2 \Big( \Vert & \nabla \dot{\rho} \Vert^2_{L^4(\Omega)} \Vert \nabla p \Vert^2_{L^4(\Omega)} + \Vert \dot{\rho} \Vert^2_{L^4(\Omega)} \Vert \Delta p \Vert^2_{L^4(\Omega)} \\ 
        &+ \Vert \nabla p \Vert^2_{L^4(\Omega)}\Vert \nabla \dot{\rho} \Vert^2_{L^4(\Omega)} + \Vert p \Vert^2_{L^\infty(\Omega_T)}\Vert \Delta \dot{\rho} \Vert^2_{L^2(\Omega)} \Big). 
    \end{aligned}
\end{equation*}
By integrating in time, using \eqref{eq:dot p n galerkin unif bounded Lq}, Proposition \ref{lem:all space derivs}, and Lemma \ref{lem:H2 dot rho}, we obtain 
\begin{equation*}
    \Vert \nabla \dot{p}^n \Vert_{L^\infty(t,T;L^2(\Omega))} + \Vert \Delta \dot{p}^n \Vert_{L^2(\Omega_{t,T})} \leq C 
\end{equation*}
for some positive constant $C$ depending on $t$, but independent of $n$. It follows from the above and the periodic Calder\'on--Zygmund inequality of Lemma \ref{lem:CZ periodic} that we have proved the first assertion of the lemma. The boundedness  $\nabla \dot{p} \in L^4(\Omega_{t,T})$ follows from the interpolation  of Lemma \ref{lem:dibenedetto classic}. It follows from \eqref{eq:eqn for dot rho} that $\partial_t \dot{\rho} - \Delta \dot{\rho} \in L^4(\Omega_{t,T})$. The classical Calder\'on--Zygmund estimate for the heat equation implies $\partial_t \dot{\rho},\nabla^2 \dot{\rho} \in L^4(\Omega_{t,T})$  a.e.~$t \in (0,T)$. 

It then follows that $\dot{\rho} \in W^{1,4}(\Omega_{t,T})$, whence Morrey's embedding yields $\dot{\rho} \in L^\infty(\Omega_{t,T})$. We then apply the Gagliardo--Nirenberg inequality as per \eqref{eq:GN grad rho L8} to get 
\begin{equation}\label{eq:GN grad rho dot L8}
    \Vert \nabla\dot{\rho}(t,\cdot) \Vert_{L^8(\Omega)} \leq C \Vert   \nabla^2\dot{\rho}(t,\cdot) \Vert_{L^4(\Omega)}^{\frac{1}{2}} \Vert \dot{\rho}(t,\cdot) \Vert_{L^\infty(\Omega)}^{\frac{1}{2}} + C \Vert \dot{\rho}(t,\cdot) \Vert_{L^\infty(\Omega)}, 
\end{equation}
and raising to the power $8$ and integrating in time yields $\nabla \dot{\rho} \in L^8(\Omega_{t,T})$.
\end{proof}

We now proceed to the $H^2$-type estimate for $\dot{f}$. 
\begin{lemma}[$H^2$-type estimate for $\dot{f}$]\label{lem:H2 dot f}
   Let $f$ be a weak solution of \eqref{eq:main eqn} with initial data satisfying the assumptions of Theorem \ref{thm:smooth}. Then, for a.e.~$t \in (0,T)$, there holds $\nabla_{\bxi} \dot{f} \in L^\infty(t,T;L^2(\Upsilon))$, $\partial_t \dot{f} , \nabla_{\bxi}^2 \dot{f} \in L^2(\Upsilon_{t,T})$, $\nabla_{\bxi} \dot{f} \in L^{\frac{10}{3}}(\Upsilon_{t,T})$. 
   \end{lemma}
\begin{proof}
The proof follows the same strategy as that of Lemma \ref{lem:H2 for f}. We construct Galerkin approximations $\dot{f}^n$, each solving the linear approximate problems 
\begin{equation*}
   \left\lbrace \begin{aligned}
        &\partial_t \dot{f}^n \!-\! \dv[(1\!-\!\rho)\nabla \dot{f}^n \!+\! \dot{f}^n \nabla \rho] \!-\! \partial^2_\theta \dot{f}^n \! = 
        \! \overbrace{-\!\dv[\big((1\!-\!\rho)\dot{f} \!-\! \dot{\rho} f \big)\e(\theta) ]}^{\in L^2} \!+\! \overbrace{\dv [  \!-\! \dot{\rho} \nabla f \!+\! f \nabla\dot{\rho}  ]}^{\in L^4}, 
        \\ &\dot{f}^n(t_0,\cdot,\theta) = \sum_{j=1}^n \langle \dot{f}^n(t_0,\cdot,\theta) , \varphi_j \rangle \varphi_j. 
    \end{aligned}\right. 
\end{equation*}
We also have $(1\!-\!\rho)\dot{f} \!-\! \dot{\rho} f \in L^q$ for all $q\in[1,\infty)$ from Proposition \ref{lem:all space derivs}. Arguing as per the proof of Lemma \ref{lem:H2 for f}, there exists $\dot{f}^n$ sovling the above weakly in duality with $\span\{\varphi_1,\dots,\varphi_n\}$, and uniformly bounded in $L^\infty(t,T;L^2(\Upsilon)) \cap L^2(t,T;H^1(\Upsilon))$ and therefore also in $L^{\frac{10}{3}}(\Upsilon_{t,T})$ using Lemma \ref{lem:dibenedetto classic}: 
\begin{equation}\label{eq: dot fn galerkin in L10 3}
    \sup_n \Vert \dot{f}^n \Vert_{L^{\frac{10}{3}}(\Upsilon_{t,T})} < \infty. 
\end{equation} 
Furthermore, $\{\dot{f}^n\}_n$ converges weakly to $\dot{f}$ in the aforementioned function spaces; the convergence to an unknown function is guaranteed by Alaoglu's Theorem, while the uniqueness of the limit is justified by an analogous argument to Step 3 of the proof of Proposition \ref{lem:H2 for p}. 

We perform the $H^2$-type estimate. Using that $\dot{f}^n$ is smooth with respect to the space variable $x$, we may expand in the diffusive term on the right-hand side in non-divergence form. We then test against $\Delta \dot{f}^n \in \span\{\varphi_1,\dots,\varphi_n\}$, \textit{cf.}~the eigenfunction relation \eqref{eq:relation for eig}, and obtain, using Young's inequality and the lower bound of Proposition \ref{prop:lower bound 1-rho}, 
\begin{equation*}
  \begin{aligned}  \frac{\der}{\der t}&\Vert \nabla \dot{f}^n(t,\cdot) \Vert^2_{L^2(\Upsilon)} + \Vert \Delta \dot{f}^n(t,\cdot) \Vert^2_{L^2(\Upsilon)} + \Vert \nabla \partial_\theta \dot{f}^n(t,\cdot) \Vert^2_{L^2(\Upsilon)} \\ 
  \leq & C \Big(\! \Vert \!\dv[((1\!-\!\rho)\dot{f} \!-\! \dot{\rho} f )\e(\theta) ]  \Vert^2_{L^2(\Upsilon)}  \!+\! \Vert  \! \dv [  f \nabla\dot{\rho} \!-\! \dot{\rho} \nabla f  ]\Vert^2_{L^2(\Upsilon)} \!+\! \Vert \dot{f}^n(t) \Vert^{\frac{10}{3}}_{L^{\frac{10}{3}}(\Upsilon)} \!+\! \Vert \Delta \rho \Vert^5_{L^5(\Upsilon)}  \!\Big), 
  \end{aligned}
\end{equation*}
where we recall that $\Delta \rho \in L^q(\Omega_{t,T})$ from Proposition \ref{lem:all space derivs}. By integrating in time, we obtain 
\begin{equation*}
    \begin{aligned}
        \Vert \nabla \dot{f}^n   \Vert^2_{L^\infty(t,T;L^2(\Upsilon))} +  \Vert \Delta \dot{f}^n   \Vert^2_{L^2(\Upsilon_{t,T})} + \Vert \nabla \partial_\theta \dot{f}^n \Vert^2_{L^2(\Upsilon_{t,T})} \leq C , 
    \end{aligned}
\end{equation*}
for $C$ depending on $t$ but not on $n$. The weak (weak-*) lower semicontinuity of the norms implies that $\dot{f}$ inherits the above bounds. An analogous estimate, computed by testing against $\partial^2_\theta \dot{f}^n$, yields the boundedness $\partial_\theta \dot{f} \in L^\infty(t,T;L^2(\Upsilon))$ and $\partial^2_\theta \dot{f} \in L^2(\Upsilon_{t,T})$. Then, $\nabla_{\bxi} \dot{f} \in L^\infty(t,T;L^2(\Upsilon))$ and $\nabla^2_{\bxi} \dot{f} \in L^2(\Upsilon_{t,T})$ follows from applying the periodic Calder\'on--Zygmund Lemma \ref{lem:CZ periodic}. Finally, the inclusion of Lemma \ref{lem:dibenedetto classic} implies $\nabla_{\bxi} \dot{f} \in L^{\frac{10}{3}}(\Upsilon_{t,T})$. 
\end{proof}

\subsubsection{De Giorgi type iterations for $\dot{f}$ and boundedness of derivatives}\label{sec:moser dot f}

The result of Lemma \ref{lem:H2 dot f} enables us to perform a De Giorgi iteration similar to that of Proposition \ref{lem:Linfty for f}, as follows. 

\begin{prop}[$L^\infty$ boundedness of $\dot{f}$]\label{prop:Linfty for dot f}
   Let $f$ be a weak solution of \eqref{eq:main eqn} with initial data satisfying Definition \ref{def:reg initial data}. Then, for a.e.~$t \in (0,T)$, there holds $\dot{f} \in L^\infty(\Upsilon_{t,T})$. 
   \end{prop}
\begin{proof}
    We adopt the same notations as in the proof of Proposition \ref{lem:Linfty for f}. Let $V:=(-\dot{\rho} f \e(\theta) + \dot{\rho} \nabla f - f \nabla \dot{\rho},0)^\intercal  \in L^8(\Upsilon_{t,T})$ from Lemma \ref{lem:H2 dot p} and Proposition \ref{lem:all space derivs}. Rewrite \eqref{eq:eqn for dot f} as 
    \begin{equation}\label{eq:eqn for dot f in moser}
        \partial_t \dot{f} + \dv_{\bxi}(U \dot{f} + V) = \dv_{\bxi}(A \nabla_{\bxi}\dot{f}), 
    \end{equation}
    where $U$ and $A$ were defined in the proof of Proposition \ref{lem:Linfty for f}; observe that we now know $U \in L^q(\Upsilon_{t,T})$ for all $q \in [1,\infty)$ by virtue of Proposition \ref{lem:all space derivs}. In turn, our setting is again similar to that of \cite[\S 3.2]{reg1}. As per the proof of Proposition \ref{lem:Linfty for f}, we define $\dot{g}(t,\bxi) := \dot{f}(t,\bxi)/{L}$, where $L>0$ is to be determined. Observe then that $\dot{g}$ also satisfies the equation \eqref{eq:eqn for dot f in moser} in the weak sense dual to $L^2(t,T;H^1_\per(\Upsilon))$. Again we fix $t_0 \in (0,T)$ chosen as a Lebesgue point of $f,\rho,\p$ and their first time-derivatives.

   Let $k > 0$ and recall the Stampacchia truncation $\mathscr{T}_k \dot{g} := (\dot{g}-k)_+$, which is an admissible test function for the weak formulation of the equation for $\dot{g}$. We get, for a.e.~$t \in (t_0/2,T)$, 
    \begin{equation*}
        \begin{aligned}
            \frac{1}{2}\frac{\der}{\der t}\int_{\Upsilon} \!|\mathscr{T}_k \dot{g}|^2 \d \bxi \!+\!\int_{\Upsilon}\! (A\nabla_{\bxi}\mathscr{T}_k \dot{g}) \cdot \nabla_{\bxi} \mathscr{T}_k \dot{g} \d \bxi\! =\!\!\! \int_{\Upsilon}\! \!\mathscr{T}_k g (U \!\cdot \!\nabla_{\bxi} \mathscr{T}_k \dot{g}) \d \bxi \!+\! \int_{\Upsilon}\! (kU+V) \!\cdot \!\nabla_{\bxi} \mathscr{T}_k \dot{g} \d \bxi. 
        \end{aligned}
    \end{equation*}
Using the assumptions on $A$, $U$, and $V$, and Young's inequality, we obtain, for a positive constant $M$ depending only on $t_0,T,\Upsilon$, and independent of $L$, for a.e.~$t \in (t_0/2,T)$ 
\begin{equation*}
            \begin{aligned}
            \frac{\der}{\der t}\int_{\Upsilon} |\mathscr{T}_k \dot{g}|^2 \d \bxi +\int_{\Upsilon} |\nabla_{\bxi} \mathscr{T}_k \dot{g} |^2 \d \bxi\leq & \, M(1+k^2) \int_{\Upsilon}  (1 \!+\! |U|^2 \!+\! |V|^2) (|\mathscr{T}_k \dot{g}|^2 \!+\! \mathds{1}_{\{\mathscr{T}_k \dot{g}>0\}}) \d \bxi. 
        \end{aligned}
\end{equation*}
By integrating with respect to $t$, it follows that there exists a positive constant $M$ independent of $U,V,k,f,L$ but depending on $t_0$ such that, for a.e.~$t_0/2<s<t<T$, 
\begin{equation}\label{eq:caccioppoli dot}
            \begin{aligned}
            \int_{\Upsilon} |\mathscr{T}_k \dot{g}|^2 \d \bxi\bigg|_{t} - \int_{\Upsilon} |& \mathscr{T}_k \dot{g}|^2 \d \bxi\bigg|_{s} + \int_s^t \int_{\Upsilon} |\nabla_{\bxi} \mathscr{T}_k \dot{g} |^2 \d \bxi \d \tau \\ 
            &\leq M(1+k^2) \int_s^t \int_{\Upsilon} (1 + |W|^2) (|\mathscr{T}_k \dot{g}|^2 + \mathds{1}_{\{\mathscr{T}_k \dot{g} >0\}}) \d \bxi \d \tau, 
        \end{aligned}
\end{equation}
where $|W| := |U| + |V| \in L^8(\Upsilon_{t_0,T})$. The rest of the proof is identical to that of Proposition \ref{lem:Linfty for f}, and so we omit it for succinctness. 
\end{proof}

We prove Proposition \ref{prop:bounded fn all time deriv} in the case $n=1$, arguing as per the proof of Proposition \ref{lem:all space derivs}.

\begin{cor}\label{cor:arbitrary integ dot f}
Let $f$ be a weak solution of \eqref{eq:main eqn} with initial data satisfying Definition \ref{def:reg initial data}. Then, for all $q \in [1,\infty)$ and a.e.~$t \in (0,T)$, there holds $\partial_t \dot{f} , \nabla^2_{\bxi} \dot{f} \in L^q(\Upsilon_{t,T})$. 
\end{cor}
\begin{proof}
Since Lemma \ref{lem:H2 dot f} showed that $\nabla^2 \dot{f} \in L^2(\Upsilon_{t,T})$ a.e.~$t\in(0,T)$, the equation \eqref{eq:eqn for dot f} may be rewritten in non-divergence form. Furthermore, arguing by the Gagliardo--Nirenberg inequality as per \eqref{eq:L4 grad f using GN}, there holds 
\begin{equation}\label{eq:L4 grad f dot using GN}
    \Vert \nabla_{\bxi} \dot{f}(t,\cdot) \Vert_{L^4(\Upsilon)} \leq C \Big( \Vert \nabla_{\bxi}^2 \dot{f}(t,\cdot) \Vert_{L^2(\Upsilon)}^{\frac{1}{2}} \Vert \dot{f}(t,\cdot) \Vert_{L^\infty(\Upsilon)}^{\frac{1}{2}} + \Vert \dot{f}(t,\cdot) \Vert_{L^\infty(\Upsilon)} \Big), 
\end{equation}
and integrating in time yields $\nabla_{\bxi} f \in L^4(\Upsilon_{t,T})$, where we used $\dot{f} \in L^\infty(\Upsilon_{t,T})$ from Proposition \ref{prop:Linfty for dot f}. It then follows from the boundedness $\nabla \dot{\rho} \in L^8(\Omega_{t,T})$ (\textit{cf.}~Lemma \ref{lem:H2 dot p}), as well as Proposition \ref{lem:all space derivs} and Lemma \ref{lem:H2 dot f}, that 
\begin{equation*}
    \partial_t \dot{f} - A:\nabla^2_{\bxi} \dot{f} \in L^{4}(\Upsilon_{t,T}) \qquad \text{a.e.~}t\in(0,T), 
\end{equation*}
and we deduce from the Schauder-type result Lemma \ref{lem:Lp schauder} that $\partial_t \dot{f} , \nabla^2_{\bxi} \dot{f} \in L^{4}(\Upsilon_{t,T})$.

In what follows, we prove that, for all $m\in\mathbb{N}$, 
\begin{equation}\label{eq:for induction proof just space dot f}
    \partial_t \dot{f}, \nabla^2_{\bxi} \dot{f} \in L^{2^m}(\Upsilon_{t,T}) \quad \text{a.e.~}t\in(0,T). 
\end{equation}
We prove this by induction; for now we have showed that the result is true for $m=1$. Suppose that the result is true up to $m\in\mathbb{N}$, and consider the case $m+1$. The Gagliardo--Nirenberg inequality yields, using also the boundedness in $L^\infty$ from Proposition \ref{prop:Linfty for dot f}, 
\begin{equation*}
    \Vert \nabla_{\bxi} \dot{f} \Vert_{L^{2^{m+1}}(\Upsilon_{t,T})} \leq C \Big( \Vert \nabla_{\bxi}^2 \dot{f} \Vert_{L^{2^{m}}(\Upsilon_{t,T})}^{2^{m}} \Vert \dot{f} \Vert_{L^\infty(\Upsilon_{t,T})}^{2^{m}} + \Vert \dot{f} \Vert_{L^\infty(\Upsilon_{t,T})}^{2^{m+1}} \Big). 
\end{equation*}
In turn, the equation \eqref{eq:eqn for dot rho} implies $\partial_t \dot{\rho} - \Delta \dot{\rho} \in L^{2^{m+1}}(\Omega_{t,T})$ and so the Calder\'on--Zygmund estimate implies $\partial_t \dot{\rho},\nabla^2 \dot{\rho} \in L^{2^{m+1}}(\Omega_{t,T})$. Then, using also Proposition \ref{lem:all space derivs}, the equation \eqref{eq:eqn for dot f} implies 
\begin{equation*}
    \partial_t \dot{f} - A:\nabla^2_{\bxi} \dot{f} \in L^{2^{m+1}}(\Upsilon_{t,T}) \qquad \text{a.e.~}t\in(0,T). 
\end{equation*}
We deduce from the Schauder-type result Lemma \ref{lem:Lp schauder} that $\partial_t \dot{f} , \nabla^2_{\bxi} \dot{f} \in L^{2^{m+1}}(\Upsilon_{t,T})$. By induction, we have proved \eqref{eq:for induction proof just space dot f} for all $m\in\mathbb{N}$. The result now follows from \eqref{eq:for induction proof just space dot f} by interpolation using H\"older's inequality on the bounded domain $\Upsilon_{t,T}$.
\end{proof}

\subsubsection{Induction for all time derivatives and proof of Proposition \ref{prop:bounded fn all time deriv}}

By taking repeated time derivatives in the equation \eqref{eq:main eqn}, we formally expect the equation \eqref{eq:general nth time deriv eqn}, \textit{i.e.}, 
    \begin{equation*}
    \begin{aligned}
    \de_t \ff{n} - \dv\Big[ & (1-\rho)\nabla\ff{n} + \ff{n}\nabla \rho \Big]  - \partial^2_\theta \ff{n} \\ 
    = -\dv & \Big[
    \big( (1- \rho) \ff{n} -\rr{n} f
    \big) \,\e(\theta)
    \Big] + \dv\Big[  -\rr{n}\nabla f + f\nabla \rr{n}  \Big]  + \sum_{k=1}^{n-1} \dv E^n_k,  
\end{aligned}
\end{equation*}
where the terms $E_k^n$ were defined in \eqref{eq:Enk def}, and, by integrating with respect to $\theta$, 
\begin{equation}\label{eq:general nth rho n}
        \begin{aligned}
    \de_t \rr{n} -  \Delta \rr{n} = - \dv 
    \big( (1- \rho) \p^{(n)} -\rr{n} \p\big) + \sum_{k=1}^{n-1} \dv G^n_k,  
\end{aligned}
\end{equation}
to be interpreted in the weak sense, where, for $k=1,\dots,n-1$, 
\begin{equation}\label{eq:def Gnk}
    G^n_k := -\binom{n}{k} \rr{n-k} \p^{(k)}. 
\end{equation}
Similarly, by multiplying \eqref{eq:general nth time deriv eqn} with $\e(\theta)$ (which is an admissible test function) and integrating with respect to the angle variable $\theta$, we obtain the evolution equation for $\p^{(n)}$: 
    \begin{equation}\label{eq:eqn for p nth time deriv}
    \begin{aligned}
    \de_t \p^{(n)} - & \dv\Big[(1-\rho)\nabla\p^{(n)}  + \p^{(n)}\otimes \nabla \rho  \Big]  + \p^{(n)}
\\ 
=& \, \dv\Big[  -\rr{n}\nabla \p  + \p \otimes \nabla \rr{n}  \Big] - \dv \Big[
 (1- \rho) \P^{(n)} -\rr{n} \P
    \Big] + \sum_{k=1}^{n-1} \dv H^n_k,  
\end{aligned}
\end{equation}
\begin{equation}\label{eq:Hnk def}
    H^n_k := \binom{n}{k} \Big[ \p^{(n-k)} \otimes \nabla \rho^{(k)} - \rho^{(n-k)} \nabla \p^{(k)} - \rho^{(n-k)} \P^{(k)}   \Big]. 
\end{equation}

This was proved for $n=1$ in \S \ref{sec:deriving eqn for dot f}. The next lemma is used to prove Proposition \ref{prop:bounded fn all time deriv}.

\begin{lemma}\label{lem:big induction proof}
Let $f$ be a weak solution of \eqref{eq:main eqn} with initial data satisfying the assumptions of Theorem \ref{thm:smooth}, and let $n \in \mathbb{N}$. Assume that, for all $j \in \{0,\dots,n\}$ and all $q\in [1,\infty)$, 
    \begin{equation}\label{eq:de giorgi up to n}
     f^{(j)}  \in L^\infty(\Upsilon_{t,T}), \quad \partial_t f^{(j)} , \nabla_{\bxi}^2 f^{(j)} \in {L^q(\Upsilon_{t,T})} \qquad \text{a.e.~}t\in(0,T). 
    \end{equation}
    Assume also that $f^{(n)}$ satisfies \eqref{eq:general nth time deriv eqn} in the weak sense. Then, for all $q\in[1,\infty)$, 
   \begin{equation}\label{eq:de giorgi n+1}
     f^{(n+1)} \in {L^\infty(\Upsilon_{t,T})} , \quad \partial_t f^{(n+1)} , \nabla_{\bxi}^2 f^{(n+1)} \in {L^q(\Upsilon_{t,T})} \qquad \text{a.e.~}t\in(0,T), 
    \end{equation}
and $f^{(n+1)}$ is a weak solution of    
\begin{equation}\label{eq:general n+1th time deriv eqn}
    \hspace{-8pt}
    \begin{aligned}
    &\de_t \ff{n+1} - \dv\Big[ (1-\rho)\nabla\ff{n+1} + \ff{n+1}\nabla \rho \Big]  - \partial^2_\theta \ff{n+1} \\ 
    &= \hspace{-1.5pt}\!-\hspace{-1.5pt} \dv \! \Big[
    \big( (1- \rho) \ff{n+1}\! -\!\rr{n+1} f
    \big)\e(\theta)
    \Big]\! \hspace{-2pt}+\! \dv\!\Big[\! \! -\!\rr{n+1}\nabla f \!+\! f\nabla \rr{n+1}  \Big] \! \hspace{-2pt}+ \!\sum_{k=1}^{n} \hspace{-1.5pt} \dv \hspace{-1.5pt} E^{n+1}_k.
\end{aligned}
\end{equation}
\end{lemma}

The proof is  analogous to the content of \S \ref{sec:deriving eqn for dot f}--\S \ref{sec:moser dot f} and follows the same sequence of steps; for succinctness, we delay this to Appendix \ref{app: proof of big induction}. Finally, we prove Proposition \ref{prop:bounded fn all time deriv}. 

\begin{proof}[Proof of Proposition \ref{prop:bounded fn all time deriv}]
    The result of Proposition \ref{prop:bounded fn all time deriv} 
    was shown in detail in  Propositions \ref{lem:Linfty for f} and \ref{lem:all space derivs} for $n=0$, and in Proposition \ref{prop:Linfty for dot f} and Corollary \ref{cor:arbitrary integ dot f} for $n=1$. For the general case $n\in\mathbb{N}$ we argue by induction. Suppose that the result is true for $n\in\mathbb{N}$. Then, the conclusion of Lemma \ref{lem:big induction proof} implies that the result is true for the case $n+1$. In turn, by induction, the result is true for all $n\in\mathbb{N}$. 
\end{proof}

\section{Proof of Theorem \ref{thm:uniqueness}: Uniqueness}\label{sec:uniqueness}

Our strategy is first to prove uniqueness on a small time interval $[0,t_*]$ for $t_* \in (0,T)$, \textit{cf.}~Lemma \ref{lem:uniqueness small times}, and to then use Theorem \ref{thm:smooth} to extend this result to the whole interval $[0,T]$. 

In this section, let $f_1, f_2 \in \mathcal{X}$ be weak solutions with the same admissible initial data $f_0$ satisfying the assumptions of Theorem \ref{thm:uniqueness}. Define $\rho_j = \int_0^{2\pi} f_j \d \theta$ ($j=1,2$), as well as the differences $\bar{f} := f_1-f_2$, $\bar{\rho} := \rho_1 - \rho_2$, and $\bar{\p} := \p_1 - \p_2$. For $j=1,2$, 
$ 0 \leq \rho_j , |\p_j| \leq 1$ a.e., 
\begin{equation}\label{eq:rho bar eqn}
    \partial_t \bar{\rho} + \dv\big((1-\rho_1)\bar{\p} - \bar{\rho}\p_2\big) = \Delta \bar{\rho} 
\end{equation}
in the weak sense dual to $L^2(0,T;H^1(\Omega))$, as well as 
\begin{equation}\label{eq:f bar eqn}
    \partial_t \bar{f} + \dv\big[\big( (1-\rho_1)\bar{f} - \bar{\rho}f_2 \big)\e(\theta) \big] = \dv\big[ (1-\rho_1)\nabla\bar{f} + \bar{f} \nabla \rho_1 - \bar{\rho}\nabla f_2 + f_2 \nabla \bar{\rho}   \big] + \partial^2_\theta \bar{f}, 
\end{equation}
\begin{equation}\label{eq:p bar eqn}
    \partial_t \bar{\p} + \dv \big( (1-\rho_1)\bar{\P} - \bar{\rho}\P_2 \big)  = \dv\big[ (1-\rho_1)\nabla\bar{\p} + \bar{\p} \otimes \nabla \rho_1 - \bar{\rho}\nabla \p_2 + \p_2 \otimes \nabla \bar{\rho}   \big] - \bar{\p}. 
\end{equation}

Our first result is the following $L^\infty$ estimate for $\bar{\rho}$ in terms of $\bar{f}$.

\begin{lemma}[$L^\infty$ bound for $\bar{\rho}$]\label{lem:bar rho quantitative Linfty bound}
Suppose the assumptions of Theorem \ref{thm:uniqueness} hold. Then, there exists $t_* \in (0,T)$ such that $\Vert \bar{\rho} \Vert_{L^\infty(\Omega_t)} \leq \Vert \bar{f} \Vert_{L^2(\Omega_t)}$ for a.e.~$t \in (0,t_*]$. 
\end{lemma}

To prove Lemma \ref{lem:bar rho quantitative Linfty bound}, we require gradient estimates for the space-periodic heat kernel $\Phi$, 
\begin{equation}\label{eq:Phi def}
    \Phi(t,x) := \sum_{n\in\mathbb{Z}^2} \Psi(t,x+2\pi n), \qquad \Psi(t,x) := \frac{1}{4\pi t}e^{-\frac{|x|^2}{4t}}; 
\end{equation}
$\Psi$ being the usual heat kernel on $\mathbb{R}^2$. The proof is an exercise (\textit{cf.}~Appendix \ref{app:proff periodic heat kernel}). 

\begin{lemma}[Gradient estimate for the periodic heat kernel]\label{lem:periodic heat kernel integ gradient}
 Let $\Phi$ be the space-periodic heat kernel of \eqref{eq:Phi def}. Then, for all $q \in [1,\frac{4}{3})$, there holds $\nabla \Phi \in L^q(\Omega_t) $ for all $t \in (0,\infty)$. Furthermore, there exists a positive constant $C_q=C_q(q,T)$ such that 
 \begin{equation*}
     \Vert \nabla \Phi \Vert_{L^q(\Omega_t)} \leq C_q t^{\frac{4-3q}{2q}} \qquad \text{for all } t \in (0,T). 
 \end{equation*}
\end{lemma}

\begin{proof}[Proof of Lemma \ref{lem:bar rho quantitative Linfty bound}]
\noindent 1. \textit{$H^1$ estimate for $\bar{\rho}$}: We test \eqref{eq:rho bar eqn} with $\bar{\rho}$. By Young's inequality, we get $\frac{\der}{\der t}\int_\Omega |\bar{\rho}|^2 \d x + \int_\Omega |\nabla \bar{\rho}|^2 \d x \leq 2\int_\Omega |\bar{\p}|^2 \d x + 2 \int_\Omega |\bar{\rho}|^2 \d x$. Using Jensen's inequality, 
\begin{equation}\label{eq:H1 bar rho}
    \Vert \bar{\rho} \Vert_{L^\infty(0,t;L^2(\Omega))}^2 + \Vert \nabla \bar{\rho} \Vert_{L^2(\Omega_t)}^2 \leq C \Vert \bar{f} \Vert_{L^2(\Upsilon_t)}^2 \qquad \text{for a.e.~}t\in(0,T), 
\end{equation}
where the positive constant $C$ depends only on $|\Omega|$. Using the inclusions of Lemma \ref{lem:dibenedetto classic}, we obtain that there exists a positive constant $C=C(\Omega,T)$ such that 
\begin{equation}\label{eq:bar rho all Lq bounds}
    \Vert \bar{\rho} \Vert_{L^4(\Omega_t)} \leq C \Vert \bar{f} \Vert_{L^2(\Upsilon_t)} \qquad \text{for a.e.~} t \in (0,T). 
\end{equation}

\smallskip 

\noindent 2. \textit{$H^1$ estimate for $\bar{\p}$}: We obtain an analogous estimate to \eqref{eq:bar rho all Lq bounds} for $\bar{\p}$. We perform the $H^1$-estimate for $\bar{\p}$. By testing \eqref{eq:p bar eqn} with $\bar{\p}$, using Young's inequality, the $L^\infty$ bound for $1-\rho_1,\p_2,\P_2$, and the lower bound on $1-\rho_1$ from Lemma \ref{lem:unif lower bound for stronger initial data}, there exists $C=C(\esssup_\Omega \rho_0,\Omega)$ such that 
\begin{equation*}
    \begin{aligned}
       \Vert \bar{\p} \Vert^2_{L^\infty(0,t;L^2(\Omega))} + \Vert \nabla \bar{\p} \Vert^2_{L^2(\Omega_t)} \leq C \Big( & \Vert \bar{\P} \Vert_{L^2(\Omega_t)}^2 + \Vert \bar{\rho} \Vert_{L^2(\Omega_t)}^2 +  \Vert \nabla \bar{\rho} \Vert_{L^2(\Omega_t)}^2 \\ 
       &\,\,\, + \Vert \bar{\p} \Vert_{L^4(\Omega_t)}^2 \Vert \nabla \rho_1 \Vert_{L^4(\Omega_t)}^2 +  \Vert \bar{\rho} \Vert_{L^\infty(\Omega_t)}^2 \Vert \nabla \p_2 \Vert_{L^2(\Omega_t)}^2 \Big). 
    \end{aligned}
\end{equation*}
Using Jensen's inequality, the estimates \eqref{eq:H1 bar rho} and \eqref{eq:bar rho all Lq bounds}, and the interpolation Lemma \ref{lem:dibenedetto classic}, 
\begin{equation*}
    \begin{aligned}
     \Vert \bar{\p} \Vert^2_{L^4(\Omega_t)} +  \Vert \bar{\p} & \Vert^2_{L^\infty(0,t;L^2(\Omega))} + \Vert \nabla \bar{\p} \Vert^2_{L^2(\Omega_t)} \\ 
     &\leq C \Big( \Vert \bar{f} \Vert_{L^2(\Omega_t)}^2 + \Vert \bar{\p} \Vert_{L^4(\Omega_t)}^2 \Vert \nabla \rho_1 \Vert_{L^4(\Omega_t)}^2 + \Vert \bar{\rho} \Vert_{L^\infty(\Omega_t)}^2 \Vert \nabla \p_2 \Vert_{L^2(\Omega_t)}^2 \Big), 
    \end{aligned}
\end{equation*}
with $C=C(\Omega,T)$. Since since $\p_2 \in L^2(0,T;H^1(\Omega))$ by Lemma \ref{cor:H1 p P tensors away from initial time} and $\nabla \rho_1 \in L^4(\Omega_T)$ by Lemma \ref{cor:nabla rho in L4}, there holds $\lim_{t \to 0} \Vert \nabla \p_2 \Vert_{L^2(\Omega_t)}=\lim_{t \to 0} \Vert \nabla \rho_1 \Vert_{L^4(\Omega_t)} = 0$ by the Monotone Convergence Theorem. We absorb the fourth term on the right: there exists $t_*$ such that 
\begin{equation}\label{eq:H1 bar p}
    \begin{aligned}
          \Vert \bar{\p} \Vert^2_{L^\infty(0,t;L^2(\Omega))} + \Vert \nabla \bar{\p} \Vert^2_{L^2(\Omega_t)} \leq C \Big( \Vert \bar{f} \Vert_{L^2(\Omega_t)}^2 + \Vert \bar{\rho} \Vert_{L^\infty(\Omega_t)}^2 \Big) \quad \text{a.e.~}t \in (0,t_*].
    \end{aligned}
\end{equation}
By Lemma \ref{lem:dibenedetto classic}, there exists $C=C(\Omega,T)$ such that 
\begin{equation}\label{eq:bar p all Lq bounds}
    \Vert \bar{\p} \Vert_{L^4(\Omega_t)} \leq C \Big( \Vert \bar{f} \Vert_{L^2(\Upsilon_t)} + \Vert \bar{\rho} \Vert_{L^\infty(\Omega_t)} \Big) \qquad \text{for a.e.~} t \in (0,t_{*}]. 
\end{equation}
Furthermore, the same strategy may be applied for the quantity $\bar{\P}$, and we obtain 
\begin{equation}\label{eq:bar P matrix all Lq bounds}
    \Vert \bar{\P} \Vert_{L^4(\Omega_t)} \leq C \Big( \Vert \bar{f} \Vert_{L^2(\Upsilon_t)} + \Vert \bar{\rho} \Vert_{L^\infty(\Omega_t)} \Big) \qquad \text{for a.e.~} t \in (0,t_{*}]; 
\end{equation}
we skip the details for concision.

\smallskip 

\noindent 3. \textit{Higher estimates for $\bar{\rho}$}: We return to the equation for $\bar{\rho}$, \eqref{eq:rho bar eqn}, and perform the $H^2$ estimate. We obtain 
\begin{equation*}
\begin{aligned}
    \frac{\der}{\der t}\Vert \nabla \bar{\rho}&(t,\cdot) \Vert^2_{L^2(\Omega)} + \Vert \Delta \bar{\rho} \Vert^2_{L^2(\Omega)} \\ 
    &\leq \Vert \dv((1\!-\!\rho_1)\bar{\p}) \Vert_{L^2(\Omega)}^2 + \Vert \dv(\bar{\rho} \p_2 ) \Vert^2_{L^2(\Omega)} \\ 
    &\leq \Vert \nabla \rho_1 \Vert^2_{L^4(\Omega)} \Vert \bar{\p} \Vert^2_{L^4(\Omega)} + \Vert \nabla \bar{\p} \Vert^2_{L^2(\Omega)} + \Vert \nabla \p_2 \Vert^2_{L^2(\Omega)} \Vert \bar{\rho} \Vert^2_{L^\infty(\Omega)} + \Vert \nabla \bar{\rho} \Vert^2_{L^2(\Omega)}. 
    \end{aligned}
\end{equation*}
Integrating in time, using \eqref{eq:H1 bar p}, and $\lim_{t \to 0} \Vert \nabla \p_2 \Vert_{L^2(\Omega_t)}=\lim_{t \to 0} \Vert \nabla \rho_1 \Vert_{L^4(\Omega_t)} = 0$, we obtain that, by shrinking $t_*$ if necessary, 
\begin{equation}\label{eq:H2 bar rho}
    \begin{aligned}
        \Vert \nabla \bar{\rho} \Vert^2_{L^4(\Omega_t)} +   \Vert \nabla \bar{\rho} \Vert^2_{L^\infty(0,t;L^2(\Omega))} + \Vert \Delta \bar{\rho} \Vert^2_{L^2(\Omega_t)} \leq C \Big( \Vert \bar{f} \Vert_{L^2(\Omega_t)}^2 + \Vert \bar{\rho} \Vert_{L^\infty(\Omega_t)}^2 \Big) \quad \text{a.e.~}t \in (0,t_*], 
    \end{aligned}
\end{equation}
where we used the interpolation Lemma \ref{lem:dibenedetto classic} and $C$ is independent of $t$.

\smallskip 

\noindent 4. \textit{Higher estimate for $\bar{p}$}: We repeat this strategy to get a similar estimate for $\bar{\p}$. We use the coordinate form of the equation \eqref{eq:p bar eqn}, \textit{cf.}~\eqref{eq:polarisation eq coord form galerkin}, \textit{i.e.}~with summation convention in the index $j$: 
\begin{equation*}
    \partial_t \bar{p}_i + \partial_j \big( (1\!-\!\rho_1)\bar{P}_{ij} - \bar{\rho} P_{2,ij} \big) = \partial_j \big[ (1\!-\!\rho_1) \partial_j \bar{p}_i + \bar{p}_i \partial_j \rho_1 - \bar{\rho} \partial_j p_{2,i} + p_{2,i}\partial_j \bar{\rho}   \big] - \bar{p}_i. 
\end{equation*}
For clarity  of exposition, we omit the $i$ subscript. We test the above with $4\bar{p}^3$ and get, using the uniform lower bound on $1\!-\!\rho_1$ from Lemma \ref{lem:unif lower bound for stronger initial data}, 
\begin{equation*}
    \begin{aligned}
        \frac{\der}{\der t}\Vert \bar{p}\Vert^4_{L^4(\Omega)} \!+\! \Vert \nabla \bar{p}^2 \Vert^2_{L^2(\Omega)} \!+\! \Vert \bar{p} \Vert^4_{L^4(\Omega)} \!\leq & \, C\! \int_\Omega |\bar{p}| |\nabla \bar{p}^2 | (1\!-\!\rho_1) |\bar{\P}|   \d x \!+\!  C \!\int_\Omega |\bar{p}| |\nabla \bar{p}^2 | |\bar{\rho}| |\P_2| \d x  \\ 
        &+\! C\! \int_\Omega |\nabla \rho_1 | |\nabla \bar{p}^2 | \bar{p}^2 \d x \!+\! C \! \int_\Omega |\bar{\rho}| |\nabla \p_2 | |\bar{p} | |\nabla \bar{p}^2 | \d x \\ 
        &+\! C \!\int_\Omega |\p_2| |\nabla \bar{\rho} | |\bar{p} | |\nabla \bar{p}^2 | \d x. 
    \end{aligned}
\end{equation*}
Using Young's inequality and the boundedness of $\rho_i,\p_i,\P_i$ ($i=1,2$), we get 
\begin{equation*}
    \begin{aligned}
        \frac{\der}{\der t}\Vert \bar{p}^2\Vert^2_{L^2(\Omega)} \!+\! \Vert \nabla \bar{p}^2 \Vert^2_{L^2(\Omega)}  \!\leq & \, C \Big( \Vert \bar{p} \Vert^2_{L^4(\Omega)} \Vert \bar{\P} \Vert^2_{L^4(\Omega)} \! +\!  \Vert \bar{p} \Vert_{L^4}^2 \Vert \bar{\rho}\Vert^2_{L^4(\Omega)} \! + \! \Vert \nabla \rho_1 \Vert^2_{L^{4}(\Omega)} \Vert \bar{p}^2 \Vert_{L^4(\Omega)}^2  \\ 
        &\,\,\,\,\,\,\,\,\,\,\, +\! \Vert \nabla \p_2 \Vert^2_{L^4(\Omega)} \Vert \bar{\rho} \Vert^2_{L^\infty(\Omega)} \Vert \bar{p} \Vert^2_{L^{4}(\Omega)} \!+\! \Vert \nabla \bar{\rho} \Vert^2_{L^4(\Omega)} \Vert \bar{p} \Vert^2_{L^4(\Omega)} \Big). 
    \end{aligned}
\end{equation*}
In turn, by integrating in time and using \eqref{eq:H2 bar rho}, \eqref{eq:bar P matrix all Lq bounds}, and \eqref{eq:bar p all Lq bounds}, we get 
\begin{equation*}
    \begin{aligned}
      \Vert \bar{p}^2 \Vert_{L^4(\Omega_t)}^2 +  \Vert \bar{p}^2\Vert^2_{L^\infty(0,t;L^2(\Omega))} \!+\! \Vert \nabla \bar{p}^2 \Vert^2_{L^2(\Omega_t)}  \!\leq & \, C \Big( \Vert \bar{f} \Vert_{L^2(\Omega_t)}^2 + \Vert \bar{\rho} \Vert_{L^\infty(\Omega_t)}^2 \Big), 
    \end{aligned}
\end{equation*}
where we absorbed the term $\Vert \nabla \rho_1 \Vert^2_{L^{4}(\Omega_t)} \Vert \bar{p}^2 \Vert_{L^4(\Omega_t)}^2$ into the left-hand side using the interpolation inequality of Lemma \ref{lem:dibenedetto classic} applied to $\bar{p}^2$. It therefore follows that 
\begin{equation}\label{eq:bar p L8 bound}
    \begin{aligned}
      \Vert \bar{\p} \Vert_{L^8(\Omega_t)} \!\leq & \, C \Big( \Vert \bar{f} \Vert_{L^2(\Omega_t)} + \Vert \bar{\rho} \Vert_{L^\infty(\Omega_t)} \Big). 
    \end{aligned}
\end{equation}
We note that, by doing the estimate in this manner instead of arguing by the $H^2$ estimate for $\p$, we do not need to impose the condition $\p \in H^1(\Omega)$ for the initial data.

\smallskip 

\noindent 5. \textit{$L^\infty$ estimate by Duhamel}: We return to the equation \eqref{eq:rho bar eqn} for $\bar{\rho}$. Defining $\mathbf{G} := (1-\rho_1)\bar{\p} - \bar{\rho} \p_2 $, this equation reads $\partial_t \bar{\rho} - \Delta \bar{\rho} = -\dv \mathbf{G}$. We interpret this as a linear problem with $\mathbf{G}$ given and satisfying $|\mathbf{G}| \leq 2$, as well as the estimate 
\begin{equation}\label{eq:bold G bound}
    \Vert \mathbf{G} \Vert_{L^8(\Omega_t)} \leq C \Big( \Vert \bar{f} \Vert_{L^2(\Omega_t)} + \Vert \bar{\rho} \Vert_{L^\infty(\Omega_t)} \Big) \qquad \text{for a.e.~} t \in (0,t_{*}], 
\end{equation}
where $C = C(\Omega,T)>0$, which follows from \eqref{eq:bar p L8 bound}. The standard well-posedness theory for this linear problem implies that $\bar{\rho}$ is given via Duhamel's Principle by the formula 
\begin{equation*}
    \bar{\rho}(t,x) = -\int_0^t \int_{\Omega} \Phi(s,y) \dv \mathbf{G}(t-s,x-y) \d y \d s = \int_0^t \int_{\Omega} \nabla \Phi(s,y) \! \cdot \! \mathbf{G}(t-s,x-y) \d y \d s. 
\end{equation*}
Set $q = \frac{8}{7} \in (1,\frac{4}{3})$ such that $q'=8$ and, using \eqref{eq:bold G bound} and H\"older's inequality, we deduce, for $C_q = C_q(q,\Omega,T)$, 
\begin{equation*}
    \begin{aligned}
        \Vert \bar{\rho}\Vert_{L^\infty(\Omega_t)} \leq \Vert \nabla \Phi \Vert_{L^q(\Omega_t)} \Vert \mathbf{G} \Vert_{L^{q'}(\Omega_t)} \leq C_q t^{\frac{4-3q}{2q}} \Big( \Vert \bar{f} \Vert_{L^2(\Omega_t)} + \Vert \bar{\rho} \Vert_{L^\infty(\Omega_t)} \Big) \quad \text{a.e.~}t \in (0,t_*], 
    \end{aligned}
\end{equation*}
Restricting $t_*$ to be smaller if necessary, we absorb the final term and conclude. 
\end{proof}

Next, we perform the $H^2$-type estimate on \eqref{eq:rho bar eqn}, and prove a $L^4$ estimate for weak solutions satisfying the assumptions of Theorem \ref{thm:uniqueness}.

\begin{lemma}[$H^2$ bound for $\bar{\rho}$]\label{lem:H2 bar rho}
    Suppose the assumptions of Theorem \ref{thm:uniqueness} hold. Then, there exists $t_* \in (0,T)$ and $C=C(\Omega,T)>0$ such that, for a.e.~$t \in (0,t_*]$, 
    \begin{equation*}
\begin{aligned}
  \Vert \nabla \bar{\rho} \Vert_{L^4(\Omega_t)} + \Vert \nabla \bar{\rho} \Vert_{L^\infty(0,t;L^2(\Omega))} + \Vert \nabla^2 \bar{\rho}\Vert_{L^2(\Omega_t)}  \leq C  \Vert \bar{f}\Vert_{L^2(\Upsilon_t)}. 
   \end{aligned}
\end{equation*}  
\end{lemma}

\begin{proof}
The result follows immediately from estimate \eqref{eq:H2 bar rho} and Lemma \ref{lem:bar rho quantitative Linfty bound}. 
\end{proof}

\begin{lemma}[$L^4$ bound for $f$]\label{lem:L4 f for uniqueness}
   Suppose the assumptions of Theorem \ref{thm:uniqueness} hold. Then, there exists $C=C(\Omega,T)>0$ such that $ \Vert f \Vert_{L^4(\Upsilon_t)} \leq C$ for a.e.~$t \in (0,T)$.
\end{lemma}

\begin{proof}
    The Gagliardo--Nirenberg inequality implies 
\begin{equation*}
    \Vert f(t,x,\cdot) \Vert_{L^4(0,2\pi)} \leq C \Vert \partial_\theta f(t,x,\cdot) \Vert_{L^2(0,2\pi)}^{\frac{1}{2}} \Vert f(t,x,\cdot) \Vert_{L^1(0,2\pi)}^{\frac{1}{2}} + C \Vert f(t,x,\cdot) \Vert_{L^1(0,2\pi)}, 
\end{equation*}
hence, raising to the power 4 and integrating, and using that $\Vert f(t,x,\cdot) \Vert_{L^1(0,2\pi)} = \rho \leq 1$, 
\begin{equation*}
    \Vert f\Vert_{L^4(\Upsilon_t)} \leq C \Big( 1 + \Vert \partial_\theta f \Vert_{L^2(\Upsilon_t)}^{2} \Big). 
\end{equation*}
It is an exercise to show that, if the initial datum is assumed $f_0 \in L^2(\Upsilon)$ and $\esssup_{\Omega} \rho_0 < 1$ (\textit{i.e.}~the conditions of Theorem \ref{thm:uniqueness}), then the result of Lemma \ref{lem:H1 for f} holds up to $t=0$ (\textit{cf.}~Lemma \ref{lem:H1 p}). As a result, it follows that $f \in L^4(\Upsilon_t)$, as required. 
\end{proof}

We are ready to prove the uniqueness for small times, \textit{i.e.}, on the interval $[0,t_*]$. 

\begin{lemma}[Uniqueness for small times]\label{lem:uniqueness small times}
    Let $f_1$ and $f_2$ weak solutions of \eqref{eq:main eqn} with regular initial data $f_0$ satisfying the assumptions of Definition \ref{def:reg initial data} and Theorem \ref{thm:uniqueness}. Then, there exists a positive time $t_* \in (0,T)$ such that $f_1 = f_2$ a.e.~in $[0,t_*] \times \Upsilon$. 
\end{lemma}
\begin{proof}
Testing the equation \eqref{eq:f bar eqn} against $\bar{f}$, we get 
\begin{equation*}
    \begin{aligned}
        \frac{1}{2}\frac{\der}{\der t}\int_\Upsilon |\bar{f}|^2 \d \bxi + \int_\Upsilon \Big( (1-\rho_1)|\nabla \bar{f}|^2 + |\partial_\theta \bar{f}|^2 \Big) \d \bxi = &\int_\Upsilon (1-\rho_1) \bar{f} \nabla \bar{f} \cdot \e(\theta) \d \bxi \\ 
        &-\!\int_\Upsilon\! f_2 \bar{\rho}\nabla\bar{f}\cdot \e(\theta) \d \bxi \!-\!\int_\Upsilon\! \bar{f} \nabla \bar{f}\! \cdot\! \nabla \rho_1 \d \bxi \\ 
        &+ \int_\Upsilon \bar{\rho}\nabla\bar{f} \cdot \nabla f_2 \d \bxi - \int_\Upsilon f_2 \nabla\bar{\rho} \cdot \nabla \bar{f} \d \bxi. 
    \end{aligned}
\end{equation*}
Note from the result of Lemma \ref{lem:unif lower bound for stronger initial data} that $1-\rho_1 \geq 1-\esssup_{\Omega_T}\rho_1>0$ a.e.~$\Omega_T$, whence, using also Young's inequality and $\Vert\bar{f}(0,\cdot)\Vert_{L^2(\Upsilon)}=0$, for a.e.~$t \in (0,T)$, 
\begin{equation}\label{eq:first H1 bound f bar}
    \begin{aligned}
        \frac{\der}{\der t}\Vert\bar{f} \Vert_{L^2(\Upsilon)}^2\!\! +\!\! \Vert \nabla_{\bxi} \bar{f}\Vert_{L^2(\Upsilon)}^2 \!\leq\! C\Big( \! &\Vert \bar{f} \Vert^2_{L^2(\Upsilon)} \!+ \!\Vert \nabla \bar{f} \Vert_{L^2(\Upsilon)}\big[\Vert f_2 \bar{\rho}\Vert_{L^2(\Upsilon)}  \\ 
        &+\!\Vert \bar{f} \nabla \rho_1 \Vert_{L^2(\Upsilon)} \!+\! \Vert \bar{\rho} \nabla f_2 \Vert_{L^2(\Upsilon)}  \! +\! \Vert f_2 \nabla\bar{\rho} \Vert_{L^2(\Upsilon)}\big] \!\Big), 
    \end{aligned}
\end{equation}
for a.e.~$t \in [0,T]$, where the positive constant $C$ depends only on $\Upsilon,T,\esssup \rho_0$. We estimate the terms on the right. By H\"older's inequality, and Sobolev's inequality, 
\begin{equation*}
    \begin{aligned}
        \Vert \bar{f} \nabla \rho_1 \Vert_{L^2(\Upsilon)} \leq C \Vert \nabla \rho_1 \Vert_{L^\infty(\Omega)} \Vert \bar{f} \Vert_{L^2(\Upsilon)} &\leq C \Vert \rho_1 \Vert_{W^{2,2}(\Omega)} \Vert \bar{f} \Vert_{L^2(\Upsilon)}. 
    \end{aligned}
\end{equation*}
Similarly, 
\begin{equation*}
    \begin{aligned}
       &\Vert f_2 \bar{\rho} \Vert_{L^2(\Upsilon)} \leq \Vert f_2 \Vert_{L^2(\Upsilon)} \Vert \bar{\rho} \Vert_{L^\infty(\Omega)}, \quad \Vert f_2 \nabla \bar{\rho} \Vert_{L^2(\Upsilon)} \leq \Vert f_2 \Vert_{L^{4}(\Upsilon)} \Vert \nabla \bar{\rho} \Vert_{L^4(\Omega)}. 
    \end{aligned}
\end{equation*}
which we shall subsequently estimate using Lemmas \ref{lem:bar rho quantitative Linfty bound} and \ref{lem:H2 bar rho}. In turn, by returning to \eqref{eq:first H1 bound f bar} and using Young's inequality, 
\begin{equation}\label{eq:returning to 8.10}
    \begin{aligned}
        \frac{\der}{\der t}\Vert\bar{f}\Vert_{L^2(\Upsilon)}^2 + \Vert \nabla_{\bxi} \bar{f}\Vert_{L^2(\Upsilon)}^2 \leq C &  (1 + \Vert \rho_1 \Vert_{W^{2,2}(\Omega)}^2 ) \Vert \bar{f} \Vert_{L^2(\Upsilon)}^{2} + C \Vert f_2 \Vert_{W^{1,2}(\Upsilon)}^2 \Vert \bar{\rho} \Vert_{L^\infty(\Omega)}^2  \\ 
        & + C\Vert f_2 \Vert_{L^4(\Omega)}^2 \Vert \nabla\bar{\rho} \Vert_{L^4(\Omega)}^2, 
    \end{aligned}
\end{equation}
hence, dropping the second term on the right-hand side and multiplying by the integrating factor $\exp(-C\int_0^t (1+\Vert \rho_1 (s)\Vert^2_{W^{2,2(\Omega)}} \d s ) = \exp(-C(t+\Vert \rho_1 \Vert^2_{L^2(0,t;W^{2,2}(\Omega))}))$, we get 
\begin{equation*}
    \begin{aligned}
        \frac{\der}{\der t}&\Big( \exp\big({-C(t+\Vert \rho_1 \Vert^2_{L^2(0,t;W^{2,2}(\Omega)))}} \big)\Vert\bar{f}(t)\Vert_{L^2(\Upsilon)}^2 \Big) \\ 
        &\leq C \exp\big({-C(t+\Vert \rho_1 \Vert^2_{L^2(0,t;W^{2,2}(\Omega)))}} \big) \Big(  \Vert f_2 \Vert_{W^{1,2}(\Upsilon)}^2 \Vert \bar{\rho} \Vert_{L^\infty(\Omega)}^2  + C\Vert f_2 \Vert_{L^4(\Omega)}^2 \Vert \nabla\bar{\rho} \Vert_{L^4(\Omega)}^2 \Big); 
    \end{aligned}
\end{equation*}
note that the integrating factor is well-defined on account on the $H^2$ estimate for $\rho$ being justified for a.e.~$t \in [0,t_*]$ as was shown in Lemma \ref{cor:nabla rho in L4} provided the assumptions of Theorem \ref{thm:uniqueness} hold for the initial data. Furthermore, since the quantity inside the exponential is nonnegative, we deduce by integrating in time the inequality: 
\begin{equation*}
    \begin{aligned}
        \exp\big(-C(t\!+\!\Vert \rho_1 & \Vert^2_{L^2(0,t;W^{2,2}(\Omega)))} \big)\Vert\bar{f}(t)\Vert_{L^2(\Upsilon)}^2 \\ 
        &\leq C \Big(  \Vert f_2 \Vert_{W^{1,2}(\Upsilon_t)}^2 \Vert \bar{\rho} \Vert_{L^\infty(\Omega_t)}^2  \!+\! C\Vert f_2 \Vert_{L^4(\Omega_t)}^2 \Vert \nabla\bar{\rho} \Vert_{L^4(\Omega_t)}^2 \Big), 
    \end{aligned}
\end{equation*}
By restricting to $t \in (0,t_*]$ for $t_* \in (0,T)$ sufficiently small, the results of Lemmas \ref{lem:bar rho quantitative Linfty bound} and \ref{lem:H2 bar rho} imply 
\begin{equation}\label{eq:penultimate line uniqueness small times}
    \begin{aligned}
       \Vert\bar{f}(t)\Vert_{L^2(\Upsilon)}^2 \leq C  \exp\big(C(t\!+\!\Vert \rho_1 \Vert^2_{L^2(0,t;W^{2,2}(\Omega)))} \big) \Big(  \Vert f_2 \Vert_{W^{1,2}(\Upsilon_t)}^2  \!+\! C\Vert f_2 \Vert_{L^4(\Omega_t)}^2  \Big)\Vert\bar{f}(t)\Vert_{L^2(\Upsilon)}^2. 
    \end{aligned}
\end{equation}

As previously mentioned, if the initial datum is assumed $f_0 \in L^2(\Upsilon)$, then the result of Lemma \ref{lem:H1 for f} holds up to $t=0$ (\textit{cf.}~Lemma \ref{lem:H1 p}), which implies $\nabla f_2 \in L^2(\Upsilon_T)$ and $f_2 \in L^4(\Upsilon_T)$ by Lemma \ref{lem:L4 f for uniqueness}. Restricting $t_* \in (0,T)$ to be smaller if needed, as $\lim_{t\to 0^+}\Vert f_2 \Vert_{L^4(\Upsilon_t)} = \lim_{t\to 0^+}\Vert \nabla f_2 \Vert_{L^2(\Upsilon_t)}=0$, we absorb the entire right-hand side of \eqref{eq:penultimate line uniqueness small times} and get: $$\Vert\bar{f}(t)\Vert_{L^2(\Upsilon)}^2  = 0\quad\text{for all } t \in (0,t_*],$$
and the proof is complete. 
\end{proof}

We conclude the paper with the proof of Theorem \ref{thm:uniqueness}. 

\begin{proof}[Proof of Theorem \ref{thm:uniqueness}]
The solution is unique on the time interval $[0,t_*]$ by Lemma \ref{lem:uniqueness small times}. Using the result of Theorem \ref{thm:smooth}, $f_1$ and $f_2$ are smooth in $[t_*,T) \times \Upsilon$, where $t_*$ is the positive time given by Lemma \ref{lem:uniqueness small times}. In turn, we infer from \eqref{eq:f bar eqn}, by testing with $\bar{f}$ as per \eqref{eq:first H1 bound f bar}, 
\begin{equation*}
    \begin{aligned}
        \Vert\bar{f}(t,\cdot)\Vert_{L^2(\Upsilon)}^2 + \Vert \nabla_{\bxi} \bar{f}\Vert_{L^2(\Upsilon_{t_*,t})}^2 \leq C\Big( & \Vert \bar{f} \Vert^2_{L^2(\Upsilon_{t_*,t})} + \Vert f_2 \Vert_{L^\infty(\Upsilon_{t_*,T})} \Vert \bar{\rho}\Vert_{L^2(\Upsilon_{t_*,t})}\Vert \nabla \bar{f} \Vert_{L^2(\Upsilon_{t_*,t})}  \\ 
        &+ \Vert \nabla \rho_1 \Vert_{L^\infty(\Upsilon_{t_*,T})} \Vert \bar{f} \Vert_{L^2(\Upsilon_{t_*,t})}\Vert \nabla \bar{f} \Vert_{L^2(\Upsilon_{t_*,t})} \\ 
        &+ \Vert \nabla f_2 \Vert_{L^\infty(\Upsilon_{t_*,T})} \Vert \bar{\rho} \Vert_{L^2(\Upsilon_{t_*,t})} \Vert \nabla \bar{f} \Vert_{L^2(\Upsilon_{t_*,t})} \\ 
        &+ \Vert f_2 \Vert_{L^\infty(\Upsilon_{t_*,T})} \Vert \nabla\bar{\rho} \Vert_{L^2(\Upsilon_{t_*,t})}\Vert \nabla \bar{f} \Vert_{L^2(\Upsilon_{t_*,t})} \Big). 
    \end{aligned}
\end{equation*}
By Jensen and Poincar\'e--Wirtinger inequalities, with $C=C(\Omega,T)>0$ and omitting $\Upsilon_{t_*,T}$, 
\begin{equation*}
    \begin{aligned}
        \Vert\bar{f}&(t,\cdot)\Vert_{L^2(\Upsilon)}^2 + \Vert \nabla_{\bxi} \bar{f} \Vert_{L^2(\Upsilon_{t_*,t})}^2 \\ 
       & \leq C\Big( \big(1 \!+\! \Vert f_2 \Vert_{L^\infty}^2 \!+\! \Vert \nabla \rho_1 \Vert_{L^\infty}^2 \!+\! \Vert \nabla f_2 \Vert_{L^\infty}^2   \big) \Vert \bar{f} \Vert_{L^2(\Upsilon_{t_*,t})}^2 \!+\! \Vert f_2 \Vert_{L^\infty(\Upsilon_{t_*,T})}^2 \Vert \nabla\bar{\rho} \Vert_{L^2(\Upsilon_{t_*,t})}^2\Big). 
    \end{aligned}
\end{equation*}
Using \eqref{eq:H1 bar rho} and letting $\sqrt{\kappa} := (1 + \Vert f_2 \Vert_{L^\infty} + \Vert \nabla \rho_1 \Vert_{L^\infty} + \Vert \nabla f_2 \Vert_{L^\infty} + \Vert f_2 \Vert_{L^\infty})$, we get 
\begin{equation*}
    \begin{aligned}
        \Vert\bar{f}(t,&\cdot)\Vert_{L^2(\Upsilon)}^2 + \Vert \nabla_{\bxi} \bar{f} \Vert_{L^2(\Upsilon_{t_*,t})}^2 \leq C \kappa \Vert \bar{f} \Vert_{L^2(\Upsilon_{t_*,t})}^2 \qquad \text{a.e.~}t \in [t_*,T]. 
    \end{aligned}
\end{equation*}
It follows from Gr\"onwall's Lemma that $\Vert\bar{f}\Vert_{L^\infty(t_*,T;L^2(\Upsilon))}^2  \leq \Vert\bar{f}(t_*,\cdot)\Vert_{L^2(\Upsilon)}^2 e^{C\kappa T}=0$.
\end{proof}

\appendix

\section{Proofs of Technical Lemmas}\label{sec:appendix}

\subsection{Proof of Lemma \ref{lem:big induction proof}}\label{app: proof of big induction}

We follow the same sequence of steps as \S \ref{sec:deriving eqn for dot f}--\S \ref{sec:moser dot f}. For succinctness, in some places, we only present formal arguments; these are justified by minor adaptations of the arguments in the aforementioned subsections.

\begin{proof}[Proof of Lemma \ref{lem:big induction proof}]
\noindent 1. \textit{Equation for $\rho^{(n+1)}$}: Introduce a further time derivative in \eqref{eq:general nth rho n}: 
\begin{equation}\label{eq:general n+1th rho n}
        \begin{aligned}
    \de_t \rr{n+1} - \Delta \rr{n+1} = -  \dv 
    \big( (1- \rho) \p^{(n+1)} -\rr{n+1} \p\big) + \sum_{k=1}^{n} \dv G^{n+1}_k; 
\end{aligned}
\end{equation}
note that the assumptions of the lemma imply that, for $k=1,\dots,n$ and all $q\in[1,\infty)$, 
\begin{equation}\label{eq:boundedness of En+1k G}
E^{n+1}_k , G^{n+1}_k , H^{n+1}_k \in L^q(t,T; W^{1,q}(\Upsilon)) \quad \text{a.e.~}t \in (0,T). 
\end{equation}
We perform the classical $H^1$-type estimate: testing the equation with $\rho^{(n+1)}$ yields 
\begin{equation*}
    \begin{aligned}
        \frac{1}{2}\frac{\der}{\der t}\!\int_\Omega\!\! |\rho^{(n+1)}|^2 \d x \!+\! \int_\Omega\!\! |\nabla \rho^{(n+1)}|^2 \d x \! = \! & \, \int_\Omega (1-\rho) \nabla \rho^{(n+1)} \cdot \p^{(n+1)} \d x \\ 
        &- \!\! \int_\Omega \rho^{(n+1)} \!\!\nabla\rho^{(n+1)} \!\cdot\! \p \d x \!-\! \sum_{k=1}^n \!\int_\Omega\!\! \nabla \rho^{(n+1)} \!\cdot\! G^{n+1}_k \d x. 
    \end{aligned}
\end{equation*}
In view of the assumptions of the lemma, since $f^{(n+1)} = \partial_t f^{(n)} \in L^p(\Upsilon_{t,T})$ for all $p\in[1,\infty)$, 
\begin{equation*}
    \begin{aligned}
       \frac{\der}{\der t}\int_\Omega \! |\rho^{(n+1)}|^2 \d x \!+\! \int_\Omega \! |\nabla \rho^{(n+1)}|^2 \d x \leq 4 \int_\Omega \Big(  |\p^{(n+1)}|^2 +  |\rho^{(n+1)}|^2 + \sum_{k=1}^n  |G^{n+1}_k|^2 \Big) \d x, 
    \end{aligned}
\end{equation*}
whence we deduce $\rho^{(n+1)} \in L^\infty(t,T;L^2(\Omega))$, $\nabla \rho^{(n+1)} \in L^2(\Omega_{t,T})$ {a.e.~}$t\in(0,T)$, and we have justified that $\rho^{(n+1)}$ satisfies \eqref{eq:general n+1th rho n} in the weak sense. Furthermore, the assumptions \eqref{eq:de giorgi up to n} imply $\partial_t \rho^{(n)} = \rho^{(n+1)} \in L^q(\Omega_{t,T})$ for all $q \in [1,\infty)$.

\smallskip 

\noindent 2. \textit{Equation for $f^{(n+1)}$ and $H^1$-type estimate}: We formally introduce a further time derivative into \eqref{eq:general nth time deriv eqn} and obtain \eqref{eq:general n+1th time deriv eqn}. We perform the $H^1$ estimate by testing with $f^{(n+1)}$: 
\begin{equation*}
    \begin{aligned}
        \frac{1}{2}\frac{\der}{\der t}\int_\Upsilon |f^{(n+1)}|^2 & \d \bxi + \int_\Upsilon \Big(  (1-\rho)|\nabla f^{(n+1)}|^2 + |\partial_\theta f^{(n+1)}|^2 \Big) \d \bxi \\ 
        =& \, \int_\Upsilon \Big( \rho^{(n+1)}\nabla f \cdot \nabla f^{(n+1)} - f^{(n+1)}\nabla \rho \!\cdot\! \nabla f^{(n+1)} - f \nabla \rho^{(n+1)}\!\cdot\! \nabla f^{(n+1)} \Big) \d \bxi \\ 
        &+\!\! \int_\Upsilon\!\! \Big( (1\!-\!\rho) f^{(n+1)} \!-\! \rho^{(n+1)}f \Big) \e(\theta) \!\cdot\! \nabla f^{(n+1)} \d \bxi \!+\! \sum_{k=1}^n \!\int_\Upsilon \!\! E^{n+1}_k \!\cdot\! \nabla f^{(n+1)} \d \bxi, 
    \end{aligned}
\end{equation*}
and thus, using Young's inequality and the lower bound on $1\!-\!\rho$ of Proposition \ref{prop:lower bound 1-rho}, 
\begin{equation*}
    \begin{aligned}
        \frac{\der}{\der t}\int_\Upsilon & |f^{(n+1)}|^2 \d \bxi + \int_\Upsilon |\nabla_{\bxi}f^{(n+1)}|^2 \d \bxi \\ 
        \leq& \, C \Big( \Vert \rho^{(n+1)}\Vert_{L^4(\Upsilon)}^2 \Vert \nabla f \Vert_{L^4(\Upsilon)}^2 \!+\! \Vert f^{(n+1)}\Vert_{L^4(\Upsilon)}^2 \Vert \nabla \rho \Vert_{L^4(\Upsilon)}^2 \!+\! \Vert f \Vert_{L^\infty(\Upsilon_{t_0,T})}^2 \Vert \nabla \rho^{(n+1)} \Vert^2_{L^2(\Upsilon)} \\ 
        &\,\,\,\,\,\,\,\,\,\,\,\,\,\,\,\, + \Vert f^{(n+1)} \Vert^2_{L^2(\Upsilon)} + \Vert \rho^{(n+1)} \Vert^2_{L^2(\Upsilon)} \Vert f \Vert_{L^\infty(\Upsilon_{t_0,T})}^2 + \sum_{k=1}^n \Vert E^{n+1}_k \Vert^2_{L^2(\Upsilon)} \Big) , 
    \end{aligned}
\end{equation*}
where we also used \eqref{eq:boundedness of En+1k G} and the assumptions \eqref{eq:de giorgi up to n}, and the positive constant $C$ depends on $t_0$. Using Young's inequality again and then integrating with respect to $t$ over the interval $[t_0,T]$ yields $f^{(n+1)} \in L^\infty(t_0,T;L^2(\Upsilon))$, $\nabla_{\bxi} f^{(n+1)} \in L^2(\Upsilon_{t_0,T})$ a.e.~$t_0 \in (0,T)$, and we have justified that $f^{(n+1)}$ satisfies \eqref{eq:general nth time deriv eqn} in the weak sense. Moreover, the assumptions \eqref{eq:de giorgi up to n} imply $\partial_t f^{(n)} = f^{(n+1)} \in L^4(\Upsilon_{t,T})$. 

\smallskip 

\noindent 3. \textit{$H^2$-type estimate for $\rho^{(n+1)}$}: The goal here is to raise the integrability of $\nabla \rho^{(n+1)}$ by obtaining estimates on $\nabla^2 \rho^{(n+1)}$. It follows from Step 2 and equation \eqref{eq:general n+1th rho n} that 
\begin{equation*}
    \partial_t \rho^{(n+1)} - \Delta \rho^{(n+1)} \in L^2(\Omega_{t,T}) \quad \text{a.e.~} t \in (0,T); 
\end{equation*}
the Calder\'on--Zygmund estimate for the heat equation yields $\partial_t \rho^{(n+1)}, \nabla^2 \rho^{(n+1)} \in L^2(\Omega_{t,T})$ a.e.~$t \in (0,T).$ Furthermore, we test the equation \eqref{eq:general n+1th rho n} with $\Delta \rho^{(n+1)}$, which is an admissible test function, and get, using \eqref{eq:boundedness of En+1k G}, the result of Step 2, and Young's inequality, 
\begin{equation*}
    \begin{aligned}    
        \frac{\der}{\der t}\int_\Omega |\nabla \rho^{(n+1)}|^2 \d x &+ \int_\Omega |\Delta \rho^{(n+1)}|^2 \d x \leq 4\Big( \Vert \nabla \p^{(n+1)}\Vert_{L^2(\Omega)}^2 + \Vert \nabla \rho \Vert^2_{L^4(\Omega)}\Vert \p^{(n+1)}\Vert^2_{L^4(\Omega)}\\ 
        &+ \Vert \nabla \rho^{(n+1)} \Vert^2_{L^2(\Omega)} + \Vert \rho^{(n+1)}\Vert^2_{L^4(\Omega)}\Vert \nabla \p \Vert^2_{L^4(\Omega)} + \sum_{k=1}^n \Vert \dv G^{n+1}_k \Vert^2_{L^2(\Omega)} \Big). 
    \end{aligned}
\end{equation*}
Using Young's inequality and then integrating with respect to time, we obtain $\nabla \rho^{(n+1)} \in L^\infty(t,T;L^2(\Omega)) \cap L^2(t,T;H^1(\Omega))$, and thus Lemma \ref{lem:dibenedetto classic} implies $\nabla \rho^{(n+1)} \in L^4(\Omega_{t,T})$. 

\smallskip 

\noindent 4. \textit{$H^2$-type estimate for $\p^{(n+1)}$}:   As per the proofs of Lemmas \ref{lem:H2 for p} and \ref{lem:H2 dot p}, we argue by Galerkin approximation. Denote the $i$-th component of $\p$ by $p_i$ and omit the $i$-subscript for clarity, and recall $\tilde{\P}_i = (P_{ij})_{j=1,2}$. We let $p^{(n+1),k}$ be the solution of 
\begin{equation*}
   \begin{aligned}
    \de_t p^{(n+1),k} - &\dv\Big[  (1-\rho)\nabla p^{(n+1),k}  + p^{(n+1),k} \nabla \rho  \Big]  + p^{(n+1),k}
\\ 
=& \, \dv\Big[ \!\! -\rr{n+1}\nabla p  \!+\! p \nabla \rr{n+1}  \Big] \!-\! \dv\Big[
 (1\!-\! \rho) \tilde{\P}^{(n+1)} \!-\! \rr{n+1} \tilde{\P}
    \Big] \!+\! \sum_{k=1}^{n}\! \dv H^{n+1}_k,  
\end{aligned}
\end{equation*}
with initial data $p^{(n+1),k}(t_0,\cdot) = \sum_{j=1}^k \langle p^{(n+1)}(t_0,\cdot) \varphi_j \rangle \varphi_j$. The existence of such $p^{(n+1),k}$ solving the above in duality with $\span\{\varphi_1,\dots,\varphi_k\}$ is guaranteed by the strategy established in the proofs of Lemmas \ref{lem:H2 for p} and \ref{lem:H2 dot p}. Furthermore, by testing the Galerkin approximation with $p^{(n+1),k} \in \span\{\varphi_1,\dots,\varphi_k\}$, using the boundedness from Steps 2 and 3, we obtain 
\begin{equation*}
   \begin{aligned}
       \Vert p^{(n+1),k}(t,\cdot) & \Vert^2_{L^2(\Omega)} + \int_{t_0}^t \Vert \nabla p^{(n+1),k}(s,\cdot) \Vert^2_{L^2(\Omega)} \d s \\ 
       &\leq C\Big( 1 + \Vert p^{(n+1),k}(t_0,\cdot)\Vert^2_{L^2(\Upsilon)} + \int_{t_0}^t \Vert \nabla \rho (s,\cdot) \Vert^2_{L^\infty(\Upsilon)} \Vert p^{(n+1),k} (s,\cdot) \Vert^2_{L^2(\Upsilon)} \d s \Big), 
   \end{aligned}
\end{equation*}
for a.e.~$t \in [t_0,T]$, for any chosen Lebesgue point $t_0 \in (0,T)$ common to $f,\rho,\p$ and all of their first $(n+1)$ time-derivatives, where the constant on the right-hand side depends on $t_0$ and $T$, but not on $k$ or $t$, and where we used the lower bound on $1-\rho$ provided by Proposition \ref{prop:lower bound 1-rho}. Using an analogous strategy to that used for \eqref{eq:conv of initial data for galerkin} to control the initial data independently of $k$, Gr\"onwall's Lemma and Lemma \ref{lem:dibenedetto classic} imply 
\begin{equation*}
  \Vert p^{(n+1),k} \Vert_{L^4(\Omega_{t,T})}  +  \Vert p^{(n+1),k} \Vert_{L^\infty(t,T;L^2(\Omega))} + \Vert \nabla p^{(n+1),k} \Vert_{L^2(\Omega_{t,T})} \leq C, 
\end{equation*}
with $C$ depending on $t$ and $q$ but not on $k$. The uniqueness argument in Step 3 of the proof of Lemma \ref{lem:H2 for p} implies, with Alaoglu's Theorem, the weak convergence $p^{(n+1),k} \rightharpoonup p^{(n+1)}$ in $L^2(t,T;H^1(\Omega))$, and weakly-* in $L^\infty(t,T;L^2(\Omega))$.

We perform the $H^2$  estimate for $p^{(n+1),k}$. Since this function is smooth with respect to  $x$, we expand the diffusions in non-divergence form. Test with $\Delta p^{(n+1),k} \in \span\{\varphi_1,\dots,\varphi_k\}$, use the lower bound on $1\!-\!\rho$ from Proposition \ref{prop:lower bound 1-rho} and Young's inequality, and get 
\begin{equation*}
    \begin{aligned}
        \frac{\der}{\der t}\Vert \nabla p^{(n+1),k}(t,&\cdot) \Vert^2_{L^2(\Omega)} + \Vert \Delta p^{(n+1),k}(t,\cdot) \Vert^2_{L^2(\Omega)} \\ 
        \leq C \Big( &\Vert \nabla p^{(n+1),k}(t,\cdot) \Vert^2_{L^2(\Omega)} + \Vert \nabla p^{(n+1),k}(t,\cdot) \Vert^4_{L^4(\Omega)} + \Vert \Delta \rho (t,\cdot) \Vert^4_{L^4(\Omega)}  \\ 
        &+ \!\Vert \nabla \rho^{(n+1)}(t,\cdot) \Vert^4_{L^4(\Omega)} \!\!+\! \Vert \nabla p (t,\cdot) \Vert^4_{L^4(\Omega)} \!\!+\! \Vert  \rho^{(n+1)}(t,\cdot) \Vert^4_{L^4(\Omega)} \!\!+\! \Vert \Delta p (t,\cdot) \Vert^4_{L^4(\Omega)} \\ 
        &+ \Vert \nabla p(t,\cdot) \Vert^4_{L^4(\Omega)} + \Vert \nabla \rho^{(n+1)} (t,\cdot) \Vert^4_{L^4(\Omega)} + \Vert  p \Vert_{L^\infty(\Omega_T)}^2 \Vert \Delta \rho^{(n+1)} (t,\cdot) \Vert^2_{L^2(\Omega)} \\ 
        &+ \Vert \dv[
 (1- \rho) \tilde{\P}^{(n+1)} -\rr{n+1} \tilde{\P}
    ]\Vert^2_{L^2(\Omega)} + \sum_{k=1}^{n} \Vert \dv H^{n+1}_k\Vert^2_{L^2(\Omega)}\Big). 
    \end{aligned}
\end{equation*}
Integrating in time, we get $\Vert \nabla p^{(n+1),k} \Vert^2_{L^\infty(t,T;L^2(\Omega))} + \Vert \Delta p^{(n+1),k} \Vert^2_{L^2(\Omega_{t,T})} \leq C$, where the positive constant $C$ depends on $t$ but not on $k$. The weak/weak-* lower semicontinuity of the norms and the periodic Calder\'on--Zygmund inequality (Lemma \ref{lem:CZ periodic}) imply 
\begin{equation*}
    \begin{aligned}
        \Vert \nabla \p^{(n+1)} \Vert^2_{L^\infty(t,T;L^2(\Omega))} + \Vert \nabla^2 \p^{(n+1)} \Vert^2_{L^2(\Omega_{t,T})} \leq C , 
    \end{aligned}
\end{equation*}
whence we also have $\nabla \p^{(n+1)} \in L^4(\Omega_{t,T})$ by the interpolation  of Lemma \ref{lem:dibenedetto classic}. The right-hand side of the equation \eqref{eq:general n+1th rho n} for $\rho^{(n+1)}$ belongs to $L^4(\Omega_{t,T})$, whence the Calder\'on--Zygmund estimate for the heat equation implies 
\begin{equation*}
     \partial_t \rho^{(n+1)} , \nabla^2 \rho^{(n+1)} \in L^4(\Omega_{t,T}). 
\end{equation*}
It now follows that $\rho^{(n+1)} \in W^{1,4}(\Omega_{t,T})$, whence Morrey's Embedding implies $\rho^{(n+1)} \in L^\infty(\Omega_{t,T})$. In turn, arguing by the Gagliardo--Nirenberg inequality as per \eqref{eq:GN grad rho dot L8}, we obtain $\nabla \rho^{(n+1)} \in L^8(\Omega_{t,T})$. 

\smallskip 

\noindent 5. \textit{$H^2$-type estimate for $f^{(n+1)}$}:  Our goal in this step is to raise the integrability of $\nabla f^{(n+1)}$ and $\nabla^2 f^{(n+1)}$. As per the proofs of Lemmas \ref{lem:H2 for f} and \ref{lem:H2 dot f}, we construct Galerkin approximations $f^{(n),k} \in \span\{\varphi_1,\dots,\varphi_k\}$; the eigenfunctions of the Laplacian on $\Omega$, which form a basis of $L^2_\per(\Omega)$. These Galerkin approximations solve the linear problems 
\begin{equation*}
  \left\lbrace  \begin{aligned}
    \de_t f^{(n+1),k}& - \dv\Big[ (1-\rho)\nabla f^{(n+1),k} + f^{(n+1),k}\nabla \rho \Big]  - \partial^2_\theta f^{(n+1),k} \\ 
     = \!- &\underbrace{\dv \! \Big[
    \big( (1\!-\! \rho) \ff{n+1}\! -\!\rr{n+1} f
    \big) \e(\theta)
    \Big]}_{\in L^2}\! +\! \underbrace{\dv\!\Big[\! f\nabla \rr{n+1}  \!-\!\rr{n+1}\nabla f \Big]}_{\in L^4 } \! + \!\sum_{k=1}^{n} \underbrace{\dv E^{n+1}_k}_{\in L^q}, \\ 
    f^{(n+1),k}(&t_0,\cdot,\theta) = \sum_{j=1}^k \langle f^{(n+1)}(t_0,\cdot,\theta) , \varphi_j \rangle \varphi_j, 
\end{aligned}\right. 
\end{equation*}
weakly in duality with $\span\{\varphi_1,\dots,\varphi_k\}$. We argue similarly to Step 4. By testing the above with $f^{(n+1),k} \in \span\{\varphi_1,\dots,\varphi_k\}$, we obtain the $H^1$ estimate 
\begin{equation*}
   \begin{aligned}
       \Vert f^{(n+1),k}(t,\cdot) & \Vert^2_{L^2(\Upsilon)} + \int_{t_0}^t \Vert \nabla_{\bxi} f^{(n+1),k}(s,\cdot) \Vert^2_{L^2(\Upsilon)} \d s \\ 
       &\leq C\Big( 1 + \Vert f^{(n+1),k}(t_0,\cdot)\Vert^2_{L^2(\Upsilon)} + \int_{t_0}^t \Vert \nabla \rho (s,\cdot) \Vert^2_{L^\infty(\Upsilon)} \Vert f^{(n+1),k} (s,\cdot) \Vert^2_{L^2(\Upsilon)} \d s \Big), 
   \end{aligned}
\end{equation*}
for a.e.~$t \in [t_0,T]$, for any chosen Lebesgue point $t_0 \in (0,T)$, where the constant on the right-hand side depends on $t_0$ and $T$, but not on $k$ or $t$, and where we used the lower bound on $1-\rho$ provided by Proposition \ref{prop:lower bound 1-rho}. Using an analogous strategy to that used for \eqref{eq:conv initial data galerkin f} to control the initial data independently of $k$, Gr\"onwall's Lemma 
\begin{equation*}
  \Vert f^{(n+1),k} \Vert^2_{L^\infty(t_0,T;L^2(\Upsilon))} + \Vert \nabla_{\bxi} f^{(n+1),k} \Vert_{L^2(\Upsilon_{t_0,T})} \leq C, 
\end{equation*}
where $C$ is independent of $k$. We deduce the weak convergence $f^{(n+1),k} \rightharpoonup f^{(n+1)}$. 

We now perform the $H^2$-type estimate. We test the equation with the diffusive term written in non-divergence form against $\Delta f^{(n+1),k} \in \span\{\varphi_1,\dots,\varphi_k\}$ and obtain 
\begin{equation*}
   \begin{aligned}
       \Vert \nabla f^{(n+1),k}(t,\cdot) & \Vert^2_{L^2(\Upsilon)} + \int_{t_0}^t \Vert \Delta f^{(n+1),k}(s,\cdot) \Vert^2_{L^2(\Upsilon)} \d s + \int_{t_0}^t \Vert \nabla \partial_\theta f^{(n+1),k}(s,\cdot) \Vert^2_{L^2(\Upsilon)} \d s \\ 
       &\leq C\Big( 1 + \Vert \nabla f^{(n+1),k}(t_0,\cdot)\Vert^2_{L^2(\Upsilon)} + \Vert \Delta \rho \Vert^4_{L^4(\Upsilon_{t_0,T})} + \Vert f^{(n+1),k} \Vert^4_{L^4(\Upsilon_{t_0,T})} \Big), 
   \end{aligned}
\end{equation*}
where we note that the final term is bounded due to the assumptions \eqref{eq:de giorgi up to n} of the lemma and Sobolev's embedding. 
Arguing as per \eqref{eq:conv initial data galerkin f}, we control the initial data independently of $k$, and get 
\begin{equation*}
    \Vert \nabla f^{(n+1),k} \Vert^2_{L^\infty(t_0,T;L^2(\Upsilon))} + \Vert \Delta f^{(n+1),k} \Vert_{L^2(\Upsilon_{t_0,T})} \leq C. 
\end{equation*}
Replicating this strategy while testing against $\partial^2_\theta f^{(n+1),k}$ yields $\Vert \partial^2_\theta f^{(n+1),k}\Vert_{L^2(\Upsilon_{t_0,T})} \leq C$. The weak lower semicontinuity of the norms and Lemma \ref{lem:CZ periodic} imply 
\begin{equation*}
    \nabla_{\bxi} f^{(n+1)} \in L^\infty(t,T;L^2(\Upsilon)), \quad \nabla_{\bxi}^2 f^{(n+1)} \in L^2(\Upsilon_{t,T}) \qquad \text{a.e.~} t \in (0,T), 
\end{equation*}
and it follows from Lemma \ref{lem:dibenedetto classic} that $\nabla_{\bxi}f^{(n+1)} \in L^{\frac{10}{3}}(\Upsilon_{t,T})$ a.e.~$t\in(0,T)$.

\smallskip 

\noindent 6. \textit{De Giorgi iterations for $f^{(n+1)}$}: We repeat the strategy of the proof of Proposition \ref{prop:Linfty for dot f}. Define $V := (
        -\rho^{(n+1)}f \e(\theta) -\rho^{(n+1)}\nabla f + f \nabla \rho^{(n+1)} + \sum_{k=1}^n E^{n+1}_k , 0 )^\intercal$, and, recalling also the definition of $U$ from \eqref{eq:A def matrix}, we rewrite \eqref{eq:general n+1th time deriv eqn} in the form: 
\begin{equation}\label{eq:div form f n+1 eq for moser}
    \partial_t f^{(n+1)} + \dv_{\bxi}(U f^{(n+1)} + V) = \dv_{\bxi}(A \nabla_{\bxi} f^{(n+1)}). 
\end{equation}
Observe that the results of the previous steps imply $V \in L^8(\Upsilon_{t,T})$ a.e.~$t \in (0,T)$. In turn, the equation \eqref{eq:div form f n+1 eq for moser} is of the same form as \eqref{eq:eqn for dot f in moser}, and hence the proof of Proposition \ref{prop:Linfty for dot f} implies the required boundedness 
$f^{(n+1)} \in L^\infty(\Upsilon_{t,T})$ a.e.~$t\in(0,T)$. In turn, using the Gagliardo--Nirenberg inequality as per \eqref{eq:L4 grad f dot using GN}, we obtain $\nabla_{\bxi} f^{(n+1)} \in L^4(\Upsilon_{t,T})$. 

\smallskip 

\noindent 7. \textit{Higher integrability for $\partial_t f^{(n)}$ and $\nabla^2_{\bxi} f^{(n)}$}: We follow the proof of Corollary \ref{cor:arbitrary integ dot f} to the letter. Using the boundedness of Steps 5 and 6, we obtain directly from \eqref{eq:general nth time deriv eqn} that 
\begin{equation*}
    \partial_t f^{(n+1)} - A:\nabla^2_{\bxi} f^{(n+1)} \in L^{4}(\Upsilon_{t,T}) \qquad \text{a.e.~}t\in(0,T), 
\end{equation*}
and by the Schauder-type Lemma \ref{lem:Lp schauder} $\partial_t \dot{f} , \nabla^2_{\bxi} \dot{f} \in L^{4}(\Upsilon_{t,T})$. Next, for all $m\in\mathbb{N}$, we show 
\begin{equation}\label{eq:for induction proof just space dot f at n time deriv}
    \partial_t f^{(n+1)}, \nabla^2_{\bxi} f^{(n+1)} \in L^{2^m}(\Upsilon_{t,T}) \quad \text{a.e.~}t\in(0,T). 
\end{equation}
We prove this by induction; for now we have showed that the result is true for $m=1$. Suppose that the result is true up to $m\in\mathbb{N}$, and consider the case $m+1$. The Gagliardo--Nirenberg inequality yields, using the boundedness in $L^\infty$ obtained in Proposition \ref{prop:Linfty for dot f}, 
\begin{equation*}
    \Vert \nabla_{\bxi} {f}^{(n+1)} \Vert_{L^{2^{m+1}}(\Upsilon_{t,T})} \leq C \Big( \Vert \nabla_{\bxi}^2 {f}^{(n+1)}  \Vert_{L^{2^{m}}(\Upsilon_{t,T})}^{2^{m}} \Vert {f}^{(n+1)} \Vert_{L^\infty(\Upsilon_{t,T})}^{2^{m}} + \Vert {f}^{(n+1)} \Vert_{L^\infty(\Upsilon_{t,T})}^{2^{m+1}} \Big). 
\end{equation*}
By returning to the equation \eqref{eq:general n+1th rho n} for $\rho^{(n+1)}$, we see that $\partial_t \rho^{(n+1)} - \Delta \rho^{(n+1)} \in L^{2^{m+1}}(\Omega_{t,T})$ and thus the Calder\'on--Zygmund estimate implies $\partial_t \rho^{(n+1)} , \nabla^2 \rho^{(n+1)} \in L^{2^{m+1}}(\Omega_{t,T})$. The previous steps and \eqref{eq:general n+1th time deriv eqn} imply $\partial_t f^{(n+1)} - A:\nabla^2_{\bxi} f^{(n+1)} \in L^{2^{m+1}}(\Upsilon_{t,T})$. We deduce from the Schauder-type result Lemma \ref{lem:Lp schauder} that $\partial_t f^{(n+1)}, \nabla^2_{\bxi} f^{(n+1)} \in L^{2^{m+1}}(\Upsilon_{t,T})$. By induction, we have proved \eqref{eq:for induction proof just space dot f at n time deriv} for all $m\in\mathbb{N}$. The result now follows from \eqref{eq:for induction proof just space dot f at n time deriv} by interpolation using H\"older's inequality on the bounded domain $\Upsilon_{t,T}$, which concludes the proof for the $(n+1)$-th time derivative. By induction, the result holds for all $n\in\mathbb{N}$. 
\end{proof}

\subsection{Proof of Lemma \ref{lem:periodic heat kernel integ gradient}}\label{app:proff periodic heat kernel}

Recall the expressions \eqref{eq:Phi def}. We have the formula $\nabla \Psi(t,x) = - \frac{x}{8 \pi t^2}e^{-\frac{|x|^2}{4t}} = - \frac{x}{2t}\Psi(t,x),$ as well as the estimate ($1 \leq q < \frac{4}{3}$) 
\begin{equation}\label{eq:Lp bound grad heat kernel}
    \begin{aligned} 
        \Vert \nabla \Psi \Vert_{L^q(\Omega_t)}^q \leq \Vert \nabla \Psi \Vert_{L^q((0,t)\times\mathbb{R}^2)}^q \leq C_q t^{\frac{4-3q}{2}} \qquad \text{for all } t \in (0,\infty),  q \in [1,\frac{4}{3}). 
    \end{aligned}
\end{equation}
with $C_q = \frac{4}{4-3q} (4\pi)^{1-q} q^{-\frac{q}{2}-1}\int_{0}^\infty r^{q+1} e^{-r^2} \d r$, which are verified by direct computation.

\begin{proof}[Proof of Lemma \ref{lem:periodic heat kernel integ gradient}]
Define $\Phi_1(t,x) := \sum_{|n| \leq 2}\Psi(t,x+2\pi n)$ and $\Phi_2 := \Phi - \Phi_1$. Observe 
    \begin{equation*}
        \Vert \nabla \Phi_1 \Vert_{L^q(\Omega_t)} \leq \sum_{|n| \leq 2} \Vert \nabla \Psi(\cdot,\cdot+2\pi n) \Vert_{L^q((0,t)\times\mathbb{R}^2)} = 12 \Vert \nabla \Psi \Vert_{L^q((0,t)\times\mathbb{R}^2)} \leq 12 C_q t^{\frac{4-3q}{2q}}, 
    \end{equation*}
    where we used the translation invariance of Lebesgue measure and the bound \eqref{eq:Lp bound grad heat kernel}. It therefore suffices to estimate $\nabla \Phi_2$. For all $x\in\Omega$ and $|n| \geq 3$, it is an exercise to show 
\begin{equation*}
   \begin{aligned} (8\pi)^{q} t^{2q}  \int_\Omega  |\nabla\Psi &(t,x+2\pi n)|^q \d x = \int_{2\pi n_2}^{2\pi(n_2+1)} \!\! \int_{2\pi n_1}^{2\pi(n_1+1)} |z|^q e^{-\frac{q|z|^2}{4t}} \d z_1 \d z_2 \\ 
   &\leq \int_{2\pi (|n_2|-1)}^{2\pi(|n_2|+1)} \!\! \int_{2\pi (|n_1|-1)}^{2\pi(|n_1|+1)} |z|^q e^{-\frac{q |z|^2}{4t}} \d z_1  \d z_2 \leq \int_{\{\frac{2\pi}{3}|n| \leq |z| \leq 4\pi |n| \}} \!\!\!\!\!\! |z|^q e^{-\frac{q |z|^2}{4t}} \d z. 
   \end{aligned}
\end{equation*}
In polar coordinates, $\Vert \nabla\Psi(t,\cdot +2\pi n)\Vert_{L^q(\Omega)}^q \leq (8\pi)^{1-q} t^{-2q} \int_{\frac{2\pi}{3}|n|}^{4\pi |n| } r^{q+1} e^{-\frac{q r^2}{4t}} \d r$, and thus
\begin{equation*}
   \begin{aligned} 
   \Vert \nabla\Psi(t,\cdot +2\pi n)\Vert_{L^q(\Omega)}^q \! &\leq \!\!(4q^{-1})^{1+\frac{q}{2}} (8\pi)^{1-q} t^{1-\frac{3q}{2}} \int_{\frac{\pi\sqrt{q}}{3}\frac{|n|}{\sqrt{t}}}^{\infty} r^{q+1} e^{-r^2} \d r \\ 
   &\leq \!\!(4q^{-1})^{1+\frac{q}{2}} (8\pi)^{1-q} t^{1-\frac{3q}{2}} e^{-\frac{\pi^2 q}{18}\frac{|n|^2}{t}} \!\!\int_{0}^{\infty}\!\!\!\!\! r^{q+1} e^{-\frac{1}{2}r^2} \!\! \d r \!=\! C_q' t^{1-\frac{3q}{2}} e^{-\frac{\pi^2 q}{18}\frac{|n|^2}{t}}, 
   \end{aligned}
\end{equation*}
with $C_q' = C_q'(q)>0$. Integrating time we get, with $C_q$ changing from line to line, 
\begin{equation*}
     \Vert \nabla\Phi_2 \Vert_{L^q(\Omega_t)} \leq \sum_{|n| \geq 3} \Vert \Psi(\cdot,\cdot +2\pi n)\Vert_{L^q(\Omega_t)} \leq C_q t^{\frac{4-3q}{2q}} \sum_{|n|\geq 3} e^{-\frac{\pi^2 }{18}\frac{|n|^2}{t}} < \infty \quad \text{for all } t \in (0,\infty).
\end{equation*}
 Thus, $\Vert \nabla \Phi \Vert_{L^q(\Omega_t)} \leq C_q t^{\frac{4-3q}{2q}} \sum\limits_{n\in\mathbb{Z}^2} e^{-\frac{\pi^2}{18}\frac{|n|^2}{t}} \leq C_q \Big(\sum\limits_{n\in\mathbb{Z}^2} e^{-\frac{\pi^2}{18}\frac{|n|^2}{T}}\Big)t^{\frac{4-3q}{2q}}$ for all $t \in (0,T)$. 
\end{proof}

\begin{small}
\noindent\textbf{Acknowledgements.} The authors ackowledge partial financial support of CNRS and INSMI via the grant PEPS JCJC 2024 ``Interaction-advection-diffusion equations with phase separation and sub-populations''  (at Laboratoire Jacques-Louis Lions, SU), and thank Prof. Endre S\"uli for useful discussions. 

\noindent\textbf{Data availability statement.} No datasets were generated or analysed during the current study.

\noindent\textbf{Declarations.} The authors declare no competing 
 interests.
\end{small}


\begin{thebibliography}{99}

\begin{footnotesize}

\bibitem{LMSS}\textsc{L. C. B. Alasio, M. Bruna, S. Fagioli and S. M. Schulz},
\newblock Existence and regularity for a system of porous medium equations with small cross-diffusion and nonlocal drifts, 
\newblock \textit{Nonlinear Anal.} \textbf{223} (2022) 113064.

\bibitem{reg1}
\newblock \textsc{L. C. B. Alasio, J. Guerand and S. M. Schulz},
\newblock Regularity and trend to equilibrium for a non-local advection-diffusion model of active particles,
\newblock \textit{Kinetic and Related Models} \textbf{18} (2025) 426-462. 

\bibitem{Amann}
\newblock \textsc{H. Amann},
\newblock Nonhomogeneous Linear and Quasilinear Elliptic and Parabolic Boundary Value Problems, 
\newblock in: Schmeisser, HJ., Triebel, H. (eds) Function Spaces, Differential Operators and Nonlinear Analysis. Teubner-Texte zur Mathematik, vol 133. Vieweg+Teubner Verlag, Wiesbaden, 1993. 

\bibitem{AGS}
\newblock \textsc{L. Ambrosio, N. Gigli and G. Savar\'e},
\newblock \textit{Gradient Flows:
In Metric Spaces and in the Space of Probability Measures}, 
\newblock Lectures in Mathematics ETH Zürich, Springer, 2008. 

\bibitem{AuscherGehring}
\newblock \textsc{P. Auscher, S. Bortz, M. Egert, and O. Saari},
\newblock On regularity of weak solutions to linear parabolic
systems with measurable coefficients, 
\newblock \textit{J.~Math.~Pures et Appl.} \textbf{121} (2019) 216-243. 

\bibitem{bacterial suspensions}
\newblock \textsc{H. C. Berg},
\newblock \textit{Random Walks in Biology}, 
\newblock Princeton University Press, 1993.




\bibitem{BonforteDolbeaultNazaretSimonov}
\newblock \textsc{M. Bonforte, J. Dolbeault, B. Nazaret and N. Simonov}, 
\newblock Stability in Gagliardo-Nirenberg-Sobolev inequalities: flows, regularity and the entropy method, 
\newblock to appear in \textit{Memoirs of the Amer.~Math.~Soc.} (2025).

\bibitem{BrunaBurgerDeWit}
\newblock \textsc{M. Bruna, M. Burger, and O. de Wit},
\newblock Lane formation and aggregation spots in a model of ants, 
\newblock to appear in \textit{SIAM J.~Appl.~Dyn.~Sys.} (2025). 

\bibitem{bbes analysis}
\newblock \textsc{M. Bruna, M. Burger, A. Esposito and S. M. Schulz}
\newblock Well-posedness of an integro-differential model for active Brownian particles, 
\newblock \textit{SIAM J.~Math.~Anal.} \textbf{54} (2022) 5662-5697. 

\bibitem{bbes model}
\newblock \textsc{M. Bruna, M. Burger, A. Esposito and S. M. Schulz},
\newblock Phase separation in systems of interacting active Brownian particles, 
\newblock \textit{SIAM J.~Appl.~Math.} \textbf{82} (2022) 1635-1660. 

\bibitem{JamesMariaErignoux}
\newblock \textsc{M. Bruna, C. Erignoux, R. L. Jack and J. Mason},
\newblock Exact hydrodynamics and onset of phase separation for an active exclusion process, 
\newblock \textit{Proc.~R.~Soc.~A} \textbf{479} (2023) 20230524. 

\bibitem{BoundEntropy}
\newblock \textsc{M. Burger, M. Di Francesco, J.-F. Pietschmann and B. Schlake},
\newblock Nonlinear cross-diffusion with size exclusion, 
\newblock \textit{SIAM J.~Math.~Anal.} \textbf{42} (2010) {2842-2871}. 

    \bibitem{Martin}
    \newblock \textsc{M. Burger and S. M. Schulz}, 
    \newblock Well-posedness and stationary states for a crowded active Brownian system with size-exclusion, 
    \newblock to appear in \textit{Discrete Contin.~Dyn.~Syst.} (2025). 

    \bibitem{self propelled colloids}
    \newblock \textsc{I. Buttinoni, J. Bialk\'e, F. K\"ummel, H. L\"owen, C. Bechinger and T. Speck},
    \newblock Dynamical Clustering and Phase Separation in Suspensions of Self-Propelled Colloidal Particles, 
    \newblock \textit{Phys. Rev. Lett.} \textbf{110} (2013) 238301.

  
  



    \bibitem{Cates:2013ia} 
    \newblock \textsc{M. E. Cates and J. Tailleur},
    \newblock {When are active Brownian particles and run-and-tumble particles equivalent? Consequences for motility-induced phase separation}, 
    \newblock \textit{EPL} \textbf{101} (2013) 20010. 

    \bibitem{pedestrian}
    \newblock \textsc{E. Cristiani, B. Piccoli and A. Tosin},
    \newblock \textit{Multiscale modeling of pedestrian dynamics}, Springer, 2014.
    

 \bibitem{degiorgi1}
 \newblock \textsc{E. De Giorgi},
 \newblock {Sull’analiticità delle estremali degli integrali multipli}, 
 \newblock \textit{Atti Accad. Naz.
Lincei Rend. Cl. Sci. Fis. Mat. Natur.} \textbf{20} (1956) {438-441}.




\bibitem{DegondEtAl}
\newblock \textsc{P. Degond, A. Frouvelle, S. Merino-Aceituno and A. Trescases},
\newblock Alignment of Self-propelled Rigid Bodies: From Particle Systems to Macroscopic Equations, 
\newblock in: Giacomin, G., Olla, S., Saada, E., Spohn, H., Stoltz, G. (eds) Stochastic Dynamics Out of Equilibrium. \textit{IHPStochDyn} 2017. Springer Proceedings in Mathematics \& Statistics, vol 282. Springer.

\bibitem{DesvillettesLaurencotEtAl}
\newblock \textsc{L. Desvillettes, P. Lauren\c{c}ot, A. Trescases and M. Winkler},
\newblock Weak solutions to triangular cross diffusion systems modeling
chemotaxis with local sensing, 
\newblock \textit{Nonlinear Analysis} \textbf{226} (2023) 113153.


    \bibitem{DiBenedetto}
    \newblock \textsc{E. DiBenedetto},
    \newblock \textit{Degenerate Parabolic Equations}, 
    \newblock {Springer Science \& Business Media}, 2012. 

    \bibitem{Evans}
    \newblock \textsc{L. C. Evans},
    \newblock \textit{Partial Differential Equations}, 
    \newblock 2$^{nd}$ edition, Graduate Studies in Mathematics, American Mathematical Society, Providence, RI, 2010. 

    \bibitem{MinTang1}
    \newblock \textsc{J. Fu, B. Perthame and M. Tang},
    \newblock Fokker-–Plank System for Movement of Micro-organism Population in Confined Environment, 
    \newblock \textit{J.~Stat.~Phys.} \textbf{184} (2021).




\bibitem{MoussaGallagher}
\newblock \textsc{I. Gallagher and A. Moussa},
\newblock The Cauchy problem for Quasilinear Parabolic Systems Revisited, 
\newblock \texttt{arXiv:2407.08226}.

    \bibitem{KZ}
    \newblock \textsc{I. Kim and Y. P. Zhang},
    \newblock Regularity properties of degenerate diffusion equations with drifts, 
    \newblock \textit{SIAM J.~Math.~Anal.} \textbf{50} (2018) 4371–4406. 

\bibitem{DungLe}
\newblock \textsc{D. Le},
\newblock Uniqueness and regularity of unbounded weak solutions to a class of cross diffusion systems, 
\newblock \textit{Nonlinear Differ. Equ. Appl.} \textbf{28} (2021) 24.


\bibitem{RosHeat}
\newblock \textsc{J. Lewenstein-Sanpera and X. Ros-Oton},
\newblock $L^p$ estimates for the Laplacian via blow-up, 
\newblock preprint (2024) \texttt{arXiv:2407.05882}. 

\bibitem{WinklerRefinedReg}
\newblock \textsc{G. Li and M. Winkler},
\newblock Refined regularity analysis for a Keller-Segel-consumption system involving signal-dependent motilities, 
\newblock \textit{Applicable Analysis} \textbf{103} (2023) 45-64.

\bibitem{lieberman}
\newblock \textsc{G. Lieberman},
\newblock \textit{Second Order Parabolic Differential Equations}, 
\newblock World Scientific Publishing Co., Inc., River Edge, NJ, xii+439 pp, 1996. 




\bibitem{mccann}
\newblock \textsc{R. J. McCann},
\newblock A convexity principle for interacting gases, 
\newblock \textit{Adv.~Math.} \textbf{128} (1997) 153-179.

\bibitem{moser1}
\newblock \textsc{J. Moser},
\newblock {A new proof of De Giorgi's theorem concerning the regularity problem for elliptic differential equations}, 
\newblock \textit{Comm. Pure Appl. Math.} \textbf{17} (1960) {457-468}.

\bibitem{nash1}
\newblock \textsc{J. Nash},
\newblock {Parabolic equations}, 
\newblock \textit{Proc.~Nation.~Acad.~Sci.~United States of America} \textbf{43} (1957) {754–758}.





\bibitem{Redner.2013}
\newblock \textsc{G. S. Redner, M. F. Hagan and A. Baskaran}, 
\newblock {Structure and Dynamics of a Phase-Separating Active Colloidal Fluid}, 
\newblock \textit{Phys.~Rev.~Lett.} \textbf{110} (2013) 055701. 


\bibitem{Romanczuk:2012iz}
\newblock \textsc{P. Romanczuk, W. Ebeling, B. Lindner and L. Schimansky-Geier},
\newblock {Active Brownian particles}, 
\newblock \textit{Eur. Phys. J. Special Topics} \textbf{202} (2012) 1-162. 

\bibitem{Filippo}
\newblock \textsc{F. Santambrogio},
\newblock \textit{Optimal Transport for Applied Mathematicians
Calculus of Variations, PDEs, and Modeling}, 
\newblock Progress in Nonlinear Differential Equations and Applications \textbf{87}, Birkh\"auser, 2015. 

\bibitem{Speck:2015um}
\newblock \textsc{T. Speck, A. Menzel, J. Bialk{\'e} and H. L{\"o}wen},
\newblock {Dynamical mean-field theory and weakly nonlinear analysis for the phase separation of active Brownian particles}, 
\newblock \textit{J.Chem.Phys.} \textbf{142} (2015) {224109}.






\bibitem{collective}
\newblock \textsc{T. Vicsek and A. Zafeiris},
\newblock Collective motion, 
\newblock \textit{Phys. Rep.} \textbf{517} (2012) 71-140.

\bibitem{WinklerOrlicz}
\newblock\textsc{M. Winkler}, 
\newblock A Result on Parabolic Gradient Regularity in Orlicz Spaces and Application to Absorption Induced Blow-Up Prevention in a Keller–Segel-Type Cross-Diffusion System, 
\newblock \textit{International Mathematics Research Notices} \textbf{19} (2023) 16336–16393.


\bibitem{Yeomans:2015dt}
\newblock \textsc{J. Yeomans}, 
\newblock {The hydrodynamics of active systems}, 
\newblock \textit{Rivista del Nuovo Cimento} \textbf{40} (2017) 1-31.

\end{footnotesize}
    
\end{thebibliography}
\end{document}